\documentclass[11pt]{amsart}
\usepackage{amsmath,amssymb}
\usepackage{graphicx}
\usepackage{amsfonts,amscd,latexsym,bbm,epsfig,epic,eepic,oldgerm,psfrag}

\newcommand{\proofend}{\hspace*{\fill} $\Box$\\}
\newcommand{\diam}{\hspace*{\fill} $\Diamond$}

\newcommand{\labell}[1] {\label{#1}}

\newcommand{\llaa}{\overset%
{\raisebox{-.05ex}[0ex][-.0ex]{\mbox{$\scriptscriptstyle\leftleftarrows$}}}a{}}

\newcommand{\raa}{\overset%
{\raisebox{-.22ex}[0ex][-.0ex]{\mbox{$\scriptscriptstyle\rightarrow$}}\mskip-0mu}a{}}
\newcommand{\laa}{\overset%
{\raisebox{-.22ex}[0ex][-.0ex]{\mbox{$\scriptscriptstyle\leftarrow$}}}a{}}
\newcommand{\rau}{\overset%
{\raisebox{-.22ex}[0ex][-.0ex]{\mbox{$\scriptscriptstyle\rightarrow$}}}u{}}
\newcommand{\lau}{\overset%
{\raisebox{-.22ex}[0ex][-.0ex]{\mbox{$\scriptscriptstyle\leftarrow$}}}u{}}

\newcommand{\HA}{{\Hat{A}}}
\newcommand{\HB}{{\Hat{B}}}
\newcommand{\HV}{{\Hat{V}}}
\newcommand{\Cr}{{{\mathcal C} \;\!\!{\it r}}}

\def\NN{\mathbbm{N}}
\def\PP{\mathbbm{P}}

\def\RR{\mathbbm{R}}

\def\ZZ{\mathbbm{Z}}

\def\ww{\boldsymbol{w}}
\def\mm{\boldsymbol{m}}
\def\s{\smallskip}
\def\ni{\noindent}
\def\b{\bigskip}
\def\m{\medskip}

\def\eps{\epsilon}

\def\gl{\lambda}

\def\gs{\sigma}

\newcommand{\1}{{{\mathchoice {\rm 1\mskip-4mu l} {\rm 1\mskip-4mu l}
{\rm 1\mskip-4.5mu l} {\rm 1\mskip-5mu l}}}}

\newlength{\facewd} \newlength{\faceht}%

\renewcommand{\Hat}{\widehat}

\newcommand{\less}{{\smallsetminus}}

\newcommand{\al}{{\alpha}}

\newcommand{\be}{{\beta}}

\newcommand{\om}{{\omega}}
\renewcommand{\eps}{{\varepsilon}}
\newcommand{\vareps}{{\epsilon}}
\newcommand{\de}{{\delta}}

\newcommand{\ga}{{\gamma}}

\newcommand{\ka}{{\kappa}}
\newcommand{\la}{{\lambda}}

\newcommand{\si}{{\sigma}}

\newcommand{\Cc}{{\mathcal C}}

\newcommand{\Ee}{{\mathcal E}}

\newcommand{\ts}{\textstyle}

\newcommand{\N}{{\mathbb N}}
\newcommand{\Q}{{\mathbb Q}}
\newcommand{\R}{{\mathbb R}}
\newcommand{\C}{{\mathbb C}}

\newcommand{\Z}{{\mathbb Z}}

\newcommand{\SSS}{{\smallskip}}
\newcommand{\QED}{{\hfill $\Box$\MS}}
\newcommand{\se} {{\stackrel{s}\hookrightarrow}}

\newtheorem{theorem}{Theorem}[subsection]
\newtheorem{thm}[theorem]{Theorem}
\newtheorem{corollary}[theorem]{Corollary}
\newtheorem{cor}[theorem]{Corollary}
\newtheorem{lemma}[theorem]{Lemma}

\newtheorem{sublemma}[theorem]{Sublemma}
\newtheorem{proposition}[theorem]{Proposition}
\newtheorem{prop}[theorem]{Proposition}

\newtheorem{defn}[theorem]{Definition}
\newtheorem{example}[theorem]{Example}
\newtheorem{remark}[theorem]{Remark}
\newtheorem{rmk}[theorem]{Remark}

\numberwithin{figure}{section}
\numberwithin{equation}{section}
\numberwithin{table}{section}
\newcommand{\MS}{{\medskip}}

\newcommand{\NI}{{\noindent}}

\begin{document}

\title{The embedding capacity of $4$-dimensional symplectic ellipsoids}
\author{Dusa McDuff} \thanks{partially supported by NSF grant DMS 0604769.}
\address{(D.~McDuff)
Department of Mathematics, 
Barnard College, Columbia University, New York, 
NY 10027-6598, USA.}
\email{dmcduff@barnard.edu}
\author{Felix Schlenk} \thanks{partially supported by SNF grant 200021-125352/1.}
\address{(F.~Schlenk) 
Institut de Math\'ematiques,
Universit\'e de Neuch\^atel, 
Rue \'Emile Argand~11, 
CP~158,
2009 Neuch\^atel,
Switzerland} 
\email{schlenk@unine.ch}
\date{\today}

\begin{abstract} 
This paper calculates the function~$c(a)$ whose value at~$a$ is 
the infimum of the size of a ball that contains a symplectic image of the ellipsoid~$E(1,a)$.
(Here $a\ge 1$ is the ratio of the area of the large axis to that of the smaller axis.)  
The structure of the graph of~$c(a)$ is surprisingly rich. 
The volume constraint implies that $c(a)$ is always greater than or equal to the square root 
of~$a$, and it is not hard to see that this is equality for large~$a$.  
However, for~$a$ less than the fourth power $\tau^4$ of the golden ratio, $c(a)$ is piecewise linear, 
with graph that alternately lies on a line through the origin and is horizontal.  
We prove this by showing that there are exceptional curves in blow ups of the complex projective plane 
whose homology classes are given by the continued fraction expansions of ratios of Fibonacci numbers.
On the interval $\left[ \tau^4,7 \right]$ we find $c(a) = \frac{a+1}{3}$.
For $a \ge 7$, the function $c(a)$ coincides with the square root except on a finite number of intervals where it is again piecewise linear.
The embedding constraints coming from embedded contact homo\-lo\-gy give rise to another capacity function~$c_{ECH}$ 
which may be computed by counting lattice points in appropriate right angled triangles. 
According to Hutchings and Taubes,
the functorial properties of embedded contact homo\-lo\-gy imply that 
$c_{ECH}(a) \le c(a)$ for all~$a$.
We show here that $c_{ECH}(a) \ge c(a)$ for all~$a$. 
\end{abstract}

\maketitle

\tableofcontents
%%%%%%%%%%%%%%%%%%%%%%%%%%%%%%%%%%%%%%%%%%%%%%%%%%%%%%%%%%%%
\section{Introduction}
%%%%%%%%%%%%%%%%%%%%%%%%%%%%%%%%%%%%%%%%%%%%%%%%%%%%%%%%%%%%

%%%%%%%%%%%%%%%%%%%%%%%%%%%%%%%%%%%%%%%%%%%%%%%%%%%%%
\subsection{Statement of results}\labell{ss:state}
%%%%%%%%%%%%%%%%%%%%%%%%%%%%%%%%%%%%%%%%%%%%%%%%%%%%%

As has been known since the time of Gromov's Nonsqueezing Theorem, 
questions about symplectic embeddings lie at the heart of symplectic geometry.   
To date, most results have concerned the embeddings of balls or of products of balls 
since these are most amenable to analysis. 
(See Cieliebak, Hofer, Schlenk and Latschev~\cite{CHLS} for a comprehensive survey 
of embedding problems.)
However, ellipsoids are another very natural class of examples.  
As pointed out by Hofer, the simplicity of the characteristic flow on their boundary makes them a natural test case for understanding the role of variational properties in symplectic geometry. 
One would like to understand the extent to which obstructions coming from periodic orbits capture all symplectic invariants. Judging from the evidence of the current work, it seems one cannot take a naive approach. As pointed out in McDuff~\cite{M}, the Ekeland--Hofer capacities of~\cite{EH} 
(which are purely variational) do not give all obstructions. 
Instead one must use invariants coming from the Hutchings--Taubes~\cite{HT} 
embedded contact homo\-lo\-gy, 
which has an unavoidably geometric flavor.  
Indeed Taubes~\cite{T} has recently shown it equals a version of Seiberg--Witten Floer homology 
and so is a gauge theory.

In view of the work of Guth~\cite{Gu} on higher dimensional symplectic embedding questions, 
there has been renewed interest in this kind of question.
However, we restrict consideration to four dimensions 
since the methods and results in this case are very different from those in higher dimensions; 
cf.~Remark~\ref{rmk:gen}.  
For relevant background and a survey of the results of the current paper see~\cite{Mcf}.

Given a real number $a\ge 1$ denote by $E(1,a)$ the 
closed ellipsoid
$$
E(1,a) \,:=\, \Bigl\{ {x_1^2+x_2^2} + \frac {x_3^2+x_4^2}{a}
\le 1\Bigr\}\;\subset\;\R^4.
$$
This paper studies the function $c \colon [1,\infty) \to \R$ defined by
\begin{equation} \labell{eq:ca}
c(a) \,:=\, \inf \left\{ \mu : E(1,a) \,\se\, B(\mu) \right\}
\end{equation}
where $B(\mu) := \Bigl\{ \sum x_i^2 \le \mu \Bigr\}$ 
is the ball of 
radius $\sqrt \mu$,
and $A \,\se\, B$ means that $A$ embeds symplectically in~$B$.
This is one of a range of symplectic capacity 
functions defined by  Cieliebak, Hofer, Latschev and Schlenk in~\cite{CHLS}, 
and the first to be calculated. 

Since $E(1,a)$ has volume $a\pi^2/2$, we must have $c(a) \ge \sqrt a$.
Here is another elementary result.

\begin{lemma}\labell{le:1}  
The function $c$ is nondecreasing and continuous.  
Further, it has the following {\it scaling property}:
\begin{equation}\labell{eq:scal} 
\frac {c(\la a)}{\la a} \le \frac {c(a)}{a} \quad\mbox{ when }\, \la >1.
\end{equation}
\end{lemma}

\begin{proof}   
The first statement is clear.  
The second holds because 
$E(1,\la a) \subset \sqrt \la \,E(1,a)$  
when $\la>1$ and also $E(1,a) \,\se\, B(\mu)$ if and only if
$\sqrt \la \,E(1,a) \,\se\, \sqrt \la \,B(\mu)=B(\la \mu)$.
\end{proof}
 
The function $c(a)$ was calculated\footnote
{
The first nontrivial result here, that $c(4)=2$, was proved earlier by Opshtein in~\cite{Op}.
}
in~\cite{M} for integral~$a$ as follows:
\begin{gather}\labell{eq:int}
c(a) = \sqrt a \,\mbox{ if } a \in \N \mbox{ is }1,4 \mbox{ or } \ge 9, \\ \notag
c(2)=c(3)=c(4)=2, \;\; c(5)=c(6) = \ts{\frac 52,\;\;
c(7) = \frac 83,\;\; c(8) = \frac{17}6}.
\end{gather}
Its monotonicity and scaling property are then enough to determine
all its values for $a\le 6$:
it is constant on the intervals $[2,4]$ and $[5,6]$ and otherwise linear, 
with graph along appropriate lines through the origin. 
(See Figure~\ref{figure.stairs} and Corollaries~\ref{cor:2} and \ref{cor:1} below.)

It turns out that the two steps of $c(a)$ that we described above extend to an infinite stairs 
for $a \in [1,\tau^4]$ where $\tau^4 = \frac {7+3\sqrt5}2$ is the fourth power of the golden ratio 
$\tau := \frac{1+\sqrt 5}{2}$. 
We call this Fibonacci stairs. 
Denote by $g_n: = f_{2n-1}$, $n \ge 1$, the terms in the odd places of the
Fibonacci  sequence $f_n$.    
(For short, we call these the \lq\lq odd Fibonacci numbers".)
Thus the sequence~$g_n$ starts with 
$
1,\; 2,\; 5,\; 13,\; 34,\;\dots$.
Set $g_0=1$ and for each $n \ge 0$ define
$$
a_n \,=\, \left( \tfrac{g_{n+1}}{g_n} \right)^2
\quad \text{ and } \quad
b_n  \,=\, \tfrac{g_{n+2}}{g_n} .
$$
Then 
\begin{gather*}
a_0 = 1 \;<\; b_0 = \tfrac 21 = 2 \;<\; 
a_1 = (\tfrac 21)^2 = 4 \;<\; b_1 = \tfrac 51 = 5 \;<\; \\
a_2 = (\tfrac 52)^2 = 6\tfrac14 \;<\; b_2 = \tfrac {13}2 = 6\tfrac12 \;<\;
a_3 = \left( \tfrac{13}5 \right)^2 = 6 \tfrac{19}{25} \;<\; b_3 = \tfrac{34}5 = 6 \tfrac 45 \;<\; 
\dots.
\end{gather*}
More generally,
$$
\dots < a_n < b_n < a_{n+1} < b_{n+1} < \dots,\quad\mbox{ and }
\lim a_n = \lim b_n = \tau^4 \approx 6.854 .
$$

\begin{figure}[ht]
 \begin{center}
  \psfrag{1}{$1$}
  \psfrag{2}{$2$}
  \psfrag{4}{$4$}
  \psfrag{5}{$5$}
  \psfrag{25}{$\tfrac{25}{4}$}
  \psfrag{132}{$\tfrac{13}{2}$}
  \psfrag{t2}{$\tau^2$}
  \psfrag{t4}{$\tau^4$}
  \psfrag{a}{$a$}
  \psfrag{c}{$c(a)$}
  \psfrag{52}{$\tfrac 52$}
  \psfrag{13}{$\tfrac{13}{5}$}
 \leavevmode\epsfbox{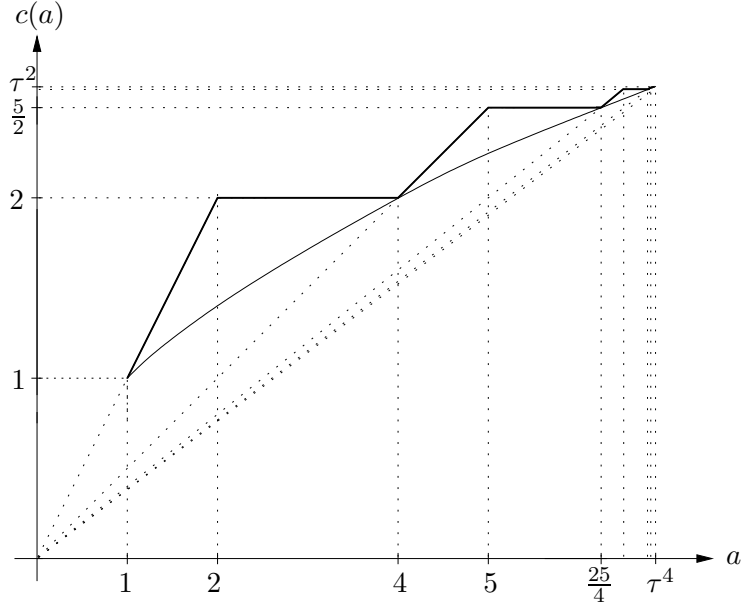}
 \end{center}
 \caption{The Fibonacci stairs: The graph of $c(a)$ on $\left[ 1,\tau^4 \right]$.}
 \label{figure.stairs}
\end{figure}
%magnification: 75
%
%

\begin{thm} \labell{thm:main}  
\begin{itemize}
\item[(i)]  
For each $n\ge 0$, $c(a) = \frac a{\sqrt{a_n}}$ for $a \in [a_n,b_n]$,
and $c$ is constant with value $\sqrt{a_{n+1}}$ on the interval $[b_n,a_{n+1}]$.
\s
\item[(ii)]
$c(a)=\frac {a+1}3$ on $[\tau^4,7]$.

\s
\item[(iii)] 
There are a finite number of closed disjoint intervals $I_j \subset [7, 8\frac1{36}]$ 
such that $c(a) = \sqrt a$ for all $a >7$, $a \notin I_j$.  
Moreover, $c$ is piecewise linear in each $I_j$,
with one non-smooth point in the interior of~$I_j$. 

\s
\item[(iv)]  
$c(a) = \sqrt a$ for $a \ge 8\frac 1{36}$.
\end{itemize}
\end{thm}

The argument proving part~(i) hinges on the existence of an unexpected relation 
between the function~$c(a)$ and the Fibonacci numbers.
We shall see 
that this relation persists for $a$ just larger than $\tau^4$, 
and so we 
also
deal with the interval $[\tau^4,7]$ by largely arithmetic means.  
However, the analysis of $c(a)$ for $a>\tau^4$ gets easier the larger~$a$ is.
As we show in Corollary~\ref{cor:2}, it is almost trivial
to see that $c(a)=\sqrt a$ when $a\ge 9$, 
and it is not much harder to see that $c(a)=\sqrt a$ when $a\ge 8 \frac 1{36}$. 
The method used also shows that there are finitely many obstructions when $a\ge 7$.  
The full analysis of $c(a)$ on $[7,8\frac 1{36}]$ takes more effort.
The intervals~$I_j$ contain rational numbers with small denominators; 
for example the three longest contain $7, 7\frac{1}2$, and $8$, cf.~Figure~\ref{figure.78}. 

\begin{figure}[ht]
 \begin{center}
  \psfrag{t4}{$\tau^4$}
  \psfrag{a}{$a$}
  \psfrag{c}{$c(a)$}
  \psfrag{7}{$7$}
  \psfrag{79}{$7\tfrac 19$}
  \psfrag{74}{$7\tfrac 14$}
  \psfrag{73}{$7\tfrac 13$}
  \psfrag{72}{$7\tfrac 12$}
  \psfrag{8}{$8$}
  \psfrag{81}{$8\tfrac1{36}$}
  \psfrag{t2}{$\tau^2$}
  \psfrag{83}{$\frac 83$}
  \psfrag{11}{$\frac{11}4$}
  \psfrag{17}{$\frac{17}6$}  
 \leavevmode\epsfbox{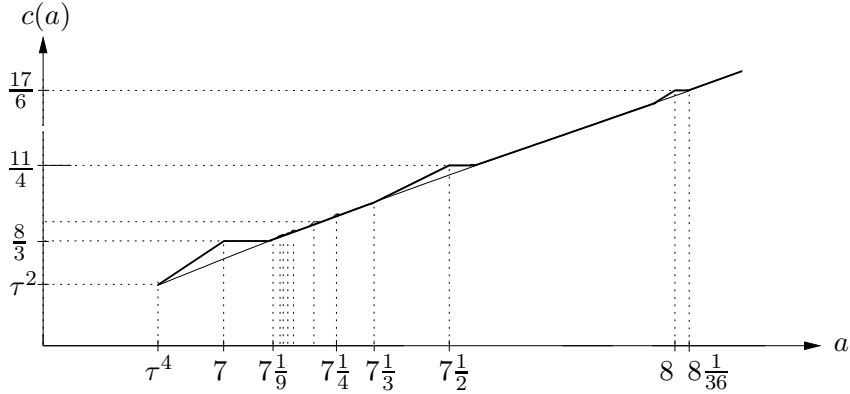}
 \end{center}
 \caption{The graph of $c(a)$ on $\left[ \tau^4, 8 \frac 1{36} \right]$.}
 \label{figure.78}
\end{figure}
%magnification: 75
%
%

\ni
We refer to Theorem~\ref{thm:78} for a full description of~$c(a)$ on the interval $[7,8\frac 1{36}]$.  One point to note here is that although all the 
flatter portions of the  graph of $c$ are  horizontal when $a<\tau^4$, this is not  true when  $a\in [7,8]$; for example the two parts of the graph of~$c$ centered at
$a=7\frac 18$ both have positive slope.

\b
\NI {\bf Connection with counting lattice  points}

As we explain in~\S\ref{ss:out} below, the obstructions to embeddings $E(1,a) \,\se\, B(\mu)$ 
that we consider come from exceptional spheres in blow ups of $\C P^2$. 
Hofer\footnote
{Private communication.} 
suggested that one should also be able to obtain a complete set of obstructions from the
embedded contact homology theory recently developed by Hutchings and Taubes~\cite{HT}.  
The embedded contact homology $ECH_*\bigl(E(a,b)\bigr)$ of a $4$-dimensional ellipsoid has one generator in each even degree with action of the form $ma+nb$; $m,n\ge 0$.  
Since the action is a nondecreasing function of degree, the actions of the generators 
arranged in the order of increasing degree form the sequence $N(a,b)$ obtained by arranging 
all numbers of the form $ma+nb$; $m,n \ge 0$, in nondecreasing order (with multiplicities).
We will say that $N(a,b) \preccurlyeq N(a',b')$ 
if each term in $N(a,b)$ is no greater than the corresponding term in $N(a',b')$.  
It is likely (though not yet fully proven) that $ECH_*$ is functorial, 
so that there is an embedding $E(a,b) \,\se\, E(a',b')$ only if there is an injective map 
$ECH_* \bigl(E(a,b)\bigr) \to ECH_* \bigl(E(a',b')\bigr)$
that increases action.  
If this were true, then the sequence $N(a,b)$ would be a monotone invariant of $E(a,b)$. 
Thus, Hofer's suggestion is that 
\begin{equation}
E(a,b) \,\se\, E(a',b') 
\;\Longleftrightarrow\;
N(a,b) \preccurlyeq N(a',b') .
\tag{C}
\end{equation}

\NI
One might be able to prove this directly by showing that the embedded curves
that provide the morphism $ECH_* \bigl(E(a,b)\bigr) \to ECH_* \bigl(E(a',b')\bigr)$ 
correspond precisely to the exceptional spheres that give our obstructions.  
We will take a more indirect approach.

For $a \ge 1$ define
$$
c_{ECH}(a) \,:=\, \inf \left\{ \mu >0 \mid N(1,a) \preccurlyeq N(\mu,\mu) \right\} .
$$
Then Conjecture~(C) 
for the case that the target ellipsoid is a ball
becomes
\begin{equation} \label{e:cc}
c_{ECH}(a) \,=\, c(a) \quad \text{ for all }\, a \ge 1.
\end{equation}
The expected functorial properties of embedded contact homology should 
imply that $c_{ECH}(a) \le c(a)$ for all~$a$; 
see Hutchings--Taubes~\cite{HT2}, 
as well as Remark~\ref{rmk:hid}~(ii).
In this paper we prove the converse.

\begin{thm}\labell{thm:ECH} 
$c_{ECH}(a) \ge c(a)$ for all $a \ge 1$.
\end{thm}

Thus, in the end, we should have $c_{ECH}(a) = c(a)$.

\begin{rmk}\labell{rmk:gen}\rm  
(i) The methods  used to analyze the embedding of a $4$-dimensional ellipsoid 
into a ball work equally well when one considers embeddings from one ellipsoid to another.  
In other words, Theorem~\ref{thm:wgt} below has an analog that is applicable to this setting; 
see~\cite[Theorem~1.5]{M}. 
One can also use much the same method to analyze embeddings of an ellipsoid into $S^2 \times S^2$;
see~\cite{DoM}.
 
\MS 
(ii) One might  wonder if these results can be extended to higher dimensions.  
For example, in dimension~$6$ is there is a symplectic embedding 
$E(a,b,c) \,\se\, E(a',b',c')$ if and only if $N(a,b,c) \preccurlyeq N(a',b',c')$?
Guth's construction in~\cite{Gu} of an embedding $E(1,R,R) \,\se\, E(2,10,2R^2)$ for {\it all}~$R>1$ 
shows that the answer is no. It is not at present clear what the correct condition should be 
in this case.
%% added
  (See Hind--Kerman \cite{HK} for a more precise version of Guth's result.)
Note that embedded contact homology is a specifically $4$-dimensional theory, 
as are the results stated in Theorem~\ref{thm:wgt} and 
Proposition~\ref{prop:eek} below 
on which our calculation of $c$~is based.

\MS 
(iii) 
There are two early papers by Biran with constructions that are
somewhat similar to ours. In~\cite{B-99} he uses an iterated ball packing
construction to obtain information on the K\"ahler cone of blow ups of~$\C P^2$.  
Continued fractions are relevant here, but Biran does not use them in the way we do.
The survey article~\cite{B-01} mentions how an understanding of embeddings of ellipsoids 
might help calculate this K\"ahler cone, and hence suggests a potential application of our work.  
However, for this one would need to understand which embeddings of ellipsoids give rise to
K\"ahler forms, a question that we do not consider.
\diam
\end{rmk}

%%%%%%%%%%%%%%%%%%%%%%%%%%%%%%%%%%%%%%%%%%%%%%%%%%%
\subsection{Method of proof}\labell{ss:meth}
%%%%%%%%%%%%%%%%%%%%%%%%%%%%%%%%%%%%%%%%%%%%%%%%%%%

The first author showed in~\cite{M} that if 
$a \ge 1$ is rational there is a finite sequence 
$\ww(a) := (w_1,\dots,w_M)$ of rational numbers such that the ellipsoid $E(1,a)$ 
embeds symplectically in the ball $B(\mu)$ exactly if the corresponding collection 
$\sqcup_i B(w_i)$ of~$M$ disjoint balls embeds symplectically in~$B(\mu)$.   
This ball embedding problem was reduced in 
McDuff--Polterovich~\cite{MP} to the question of understanding the symplectic cone of the $M$-fold 
blow up of $\C P^2$. 
After further work by Biran~\cite {B} and McDuff~\cite{Mdef}, 
the structure of this cone was finally elucidated in Li--Liu~\cite{LL} and Li--Li~\cite{LL2}. 

The key to understanding this cone is the following set~$\Ee_M$. 

\begin{defn}\labell{def:ee} 
Denote by $X_M$ the $M$-fold blow up of $\C P^2$
with any symplectic structure~$\om_M$ obtained by blow-up from the standard structure on~$\C P^2$. 
Let $L, E_1,\dots,E_M \in H_2(X_M)$ be the homology classes of the line and the $M$~exceptional divisors.  
We define~$\Ee_M$ to be the set 
consisting of $(0;-1,0,\dots,0)$ and of
all tuples $(d;\mm)$ of nonnegative integers
$(d;m_1,\dots,m_M)$ with $m_1 \ge \dots \ge m_M$ and such that
the class 
$
E_{(d;\mm)} \,:=\, dL - \sum_i m_i E_i
$
is represented in $(X_M,\om_M)$ by a symplectically embedded sphere of self-intersection~$-1$. 
If there is no danger of confusion, we will write $\Ee$ instead of~$\Ee_M$. 
Clearly, $\Ee_M \subset \Ee_{M'}$ whenever $M \le M'$.
\end{defn}

Since these classes $E_{(d;\mm)}, (d;\mm) \in \Ee_M,$ have nontrivial Gromov invariant, they have symplectically embedded representatives for all choices of the blow-up form~$\om_M$.  
Therefore the above definition does not depend on this choice. 

Denote by $-K := 3L-\sum E_i$ the standard anti-canonical divisor in~$X_M$, 
and consider the corresponding symplectic cone~$\Cc_K$, 
consisting of all classes on~$X_M$ that may be represented by a symplectic form with first Chern class Poincar\'e dual to~$-K$.  Then Li--Li show in~\cite{LL2} that
$$
\Cc_K \,=\, \left\{ \al\in H^2(X_M) \mid \al^2>0,\;\; \al(E)>0 \;\mbox{ for all } E \in \Ee_M \right\}.
$$
Proposition~\ref{prop:eek} below gives necessary and sufficient conditions for an element $(d;\mm)$ 
to belong to~$\Ee_M$. Before discussing this, we explain the relevance of~$\Ee_M$ to our problem. 
The following result is proved in~\cite{M}. We will denote by 
$\ell, e_i\in H^2(X_M)$ the Poincar\'e duals to $L, E_i$ and by
$\mm\cdot \ww = \sum_{i=1}^M m_i \;\!w_i$ the Euclidean scalar product in~$\RR^M$.

\begin{thm}\labell{thm:wgt}  
For each rational $a \ge 1$ there is a finite weight expansion $\ww(a) = (w_1,\dots,w_M)$ 
such that $E(1,a)$ embeds symplectically in the interior of $B^4(\mu)$ 
if and only if 
$\mu\ell  -\sum w_i e_i\in \Cc_K$.  
Moreover $w_i \le 1$ for all $i$ and $\sum_i w_i^2 = a$.
\end{thm}

\begin{cor}\labell{cor:wgt}
If the rational number $a\ge 1$ has weight expansion $\ww(a) = \ww=(w_i)$, then
$$
c(a) \,=\, \sup \Bigl(\sqrt a,\;\mu(d;\mm)(a) \mid (d;\mm)\in \Ee\Bigr),
$$
where
$\mu(d;\mm)(a) := \frac{\mm \cdot \ww(a)}{d}$.
\end{cor}

\begin{proof}  
The above description of $\Cc_K$ shows that $E(1,a)$ embeds into the interior of $B^4(\mu)$
if and only if
the tuple $(\mu,\ww)$ satisfies the conditions
\begin{itemize}
\item[(i)]   $\mu^2> \ww\cdot\ww =: \sum w_i^2$,

\SSS
\item[(ii)]  $d\mu> \mm\cdot\ww =: \sum m_iw_i$\; for all $(d;\mm) \in \Ee_M$.
\end{itemize}
The corollary now follows because 
$\ww \cdot \ww = a$.
\end{proof}
 
Biran showed in \cite{B} that  
$\mu(d;\mm) (k) \le \sqrt k$ for all $(d;\mm) \in \Ee$ 
for all integers $k\ge 9$.  
His argument extends to all $a\ge 9$ 
and shows:
 
\begin{cor}\labell{cor:2}  
$c(a)=\sqrt{a}$ when $a\ge 9$.
\end{cor}

\begin{proof} 
Since $c$ is continuous by Lemma~\ref{le:1}, it suffices to check this for rational $a$. 
Fix $(d;\mm) \in \Ee$.
The corresponding symplectically embedded $(-1)$-sphere~$E$ has $c_1(E) = 1$,
and so $3d-1=\sum_i m_i$.
Therefore, $\sum_i m_i\;\! w_i \le \sum_i m_i = 3d-1$.
For $a \ge 9$ we thus find
$$
\mu(d;\mm)(a) \,:=\, \tfrac{\mm \cdot \ww}d \,<\, 3 \,\le\, \sqrt a .
$$
Now use Corollary~\ref{cor:wgt}.
\end{proof}

In view of Corollary~\ref{cor:wgt}, 
our task is two-fold; first to understand the weight expansions and then to understand 
the restrictions placed on embeddings by the elements of~$\Ee_M$.
The description that we now give for $\ww(a)$ is convenient for calculations but is somewhat different from that in~\cite{M}. 
The equivalence of the two definitions is established in Corollary~\ref{cor:x}.

\begin{defn}\labell{def:wa}
Let $a=p/q \in \Q$ written in lowest terms.
The {\bf weight expansion} $\ww := (w_i) := (w_1,\dots, w_M)$ of $a\ge 1$
is defined recursively as follows:

\MS
$\bullet$   
$w_1 = 1, $ and $ w_n \ge w_{n+1}>0$ for all $n$;

\SSS 
$\bullet$  
if $ w_i>w_{i+1} = \dots = w_{n}$ (where we set $w_0 := a$), then
$$
w_{n+1} = 
\left\{\begin{array} {ll}
w_{n} &\mbox{if }\;  w_{i+1} + \dots + w_{n+1} = (n-i+1)w_{i+1} \le w_i\\
w_i - (n-i) w_{i+1} & \mbox{otherwise;} 
\end{array}\right.
$$

$\bullet$ 
the sequence stops at $w_n$ if the above formula gives $w_{n+1}=0$.

\SSS \NI
The number~$M$ of entries in $\ww(a)$
is called the {\bf length} $\ell(a)$ of~$a$.
\end{defn}

\NI
For example, $a=25/9$ has weight expansion 
$\ww(a) = (1,1,\frac 79,\frac 29,\frac 29,\frac 29,\frac 19,\frac19)$, which we will abbreviate as
$(1^{\times2}, \frac 79, \frac 29\,\!^{\times3},\frac 19\,\!^{\times2})$. 

We may also think of this expansion $(w_i)$ as consisting of $N+1$ blocks of
length $\ell_s$ of the (decreasing) numbers $x_s$ where $x_0=1$; viz:
\begin{eqnarray}\labell{eq:xs}
\ww(a) &:=& \bigl(\underbrace{1,\dots,1}_{\ell_0}, \,
\underbrace{x_1,\dots,x_1}_{\ell_1}, \,
\dots, \, \underbrace{x_N,\dots,x_N}_{\ell_N} \bigr) \\ \notag
&=& \bigl( 1^{\times \ell_0},\, x_1^{\times \ell_1}, \, \dots, \, x_N^{\times \ell_N} \bigr) .
\end{eqnarray}
Then $x_1 = a-\ell_0 < 1$, $x_2 = 1-\ell_1 x_1 < x_1$, and so on.
In this form the sequence can be generated as follows.  
If $a=\frac pq$, first draw a rectangle of length $p$ and height $q$, 
then mark off as many (say~$\ell_0$) squares of side length $q$ as possible,  
then in the remaining rectangle of size $q\times (p-\ell_0 q)$ mark off
as many (say $\ell_1$) squares of side length $(p-\ell_0 q)$ as possible, 
continuing in this way until the rectangle is completely filled.
Then $qx_j$ is the side length of the $(j+1)$st set of squares, while the $\ell_j$ are the multiplicities: 
see~Figure~\ref{figure.expa}.  
As is well known, the multiplicities $\ell_j, 0 \le j \le N$, give the continued fraction expansion $[\ell_0; \ell_1,\dots,\ell_N]$ of $p/q$. 
For example, $25/9=[2;1,3,2]$ and   
$$
\frac {25}9=2 + \frac 1{1+\frac 1{3+\frac 12}} \,.
$$
Notice also that $25/9 = 2\cdot 1^2 + (7/9)^2 + 3(2/9)^2 + 2(1/9)^2$. 

\begin{figure}[ht]
 \begin{center}
  \psfrag{9}{$9$}
  \psfrag{7}{$7$}
  \psfrag{2}{$2$}
  \psfrag{1}{\footnotesize$1$}
  \leavevmode\epsfbox{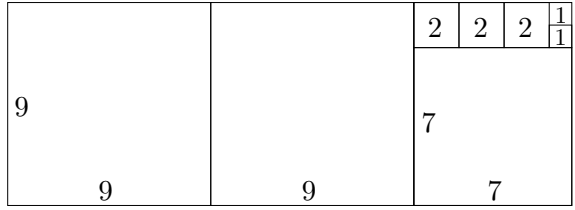}
 \end{center}
 \caption{The expansion for $a=25/9$.}
 \label{figure.expa}
\end{figure}
%magnification: 75
%capell0.jpg}

\begin{lemma}\labell{le:ww} 
Let $\ww := (w_1,\dots,w_M)$ be the weight expansion of $a=\frac pq \ge 1$. Then
\begin{eqnarray} \notag
w_M &=& \textstyle{\frac 1q} , \\ \labell{eq:rsq}
\ww\cdot\ww \,:=\, \sum_{i=1}^M w_i^2 &=& a,\\ \labell{eq:rsum}
\sum_{i=1}^M  w_i & =& a+1-\tfrac 1q .
\end{eqnarray}
\end{lemma}

\begin{proof}  
Equation~\eqref{eq:rsq} holds because the total area of all the squares is $p\:\!q$.  
To understand the sum,
suppose that there are $N+1$ sets of squares in the expansion~\eqref{eq:xs} 
so that $x_{N+1}=0$ and write
\begin{eqnarray*}
\sum_i w_i &=& 1 +\Bigl( \underbrace{1 + \dots +1}_{\ell_0-1} + x_1\Bigr) + \Bigl( \underbrace{x_1 + \dots +x_1}_{\ell_1-1}+ x_2\Bigr) + \dots \\
&&\qquad\qquad \qquad \qquad \qquad \qquad\qquad  +
\Bigl( \underbrace{x_N + \dots +x_N}_{\ell_N-1}+ x_{N+1}\Bigr)\\
&=& 1 +  (a-1) + (1-x_1) + \dots + (x_{N-1} - x_N) \\
&=& 1 + a-x_N.
\end{eqnarray*}
It remains to note that $x_N = w_M=1/q$.
This is obvious from the geometric construction. 
For, $q \;\!x_N= q\;\!w_M$ is the side length of the smallest square in the decomposition 
of the rectangle.
If this length were divisible by~$s$, then the side lengths of all the squares would be divisible 
by~$s$. Hence both $p$ and $q$ would be divisible by~$s$. But they are mutually prime by hypothesis.
\end{proof}

We next describe the sets $\Ee_M$.  
The first lemma is well known, and can be easily deduced from Proposition~\ref{prop:eek} below.
 
\begin{lemma} \labell{le:eekfin}
The set $\Ee_M$ is finite for $M\le 8$ with elements $(d; m_1,\dots,m_M)$ equal to:
\begin{gather*}  
(0;-1),\;\; 
(1; 1, 1),\;\; (2; 1^{\times 5}),\;\;
(3; 2, 1^{\times 6}), \\
(4;2^{\times3}, 1^{\times5}),\;\;
(5; 2^{\times6}, 1,1),\;\; (6; 3, 2^{\times7}).
\end{gather*}
\end{lemma}

From this one can immediately calculate $c(a)$ for those $a$ whose 
weight expansion has $k\le 8$.

\begin{cor}\labell{cor:1} The function $c$ takes the following values:
\begin{gather}\notag
c(2)=c(3)=c(4) = 2, \quad c(5)=c(6) = \ts{\frac 52},\\ \notag
\ts{c\bigl(\frac {13}2\bigr) = \frac{13}5,\quad
c(7) = \frac 83,\quad c(8) = \frac{17}6.}
\end{gather}
Moreover, its graph  is linear on each subinterval
$[1,2]$, $[2,4]$, $[4,5]$, $[5,6]$.
\end{cor} 

\begin{proof} 
The values of $c(a)$ for integers $a \in [1,8]$ were calculated in \cite[Cor~1.2]{M}.  
One can similarly calculate $c(\frac{13}2)$ since the length of $\ww \bigl(\frac {13}2 \bigr)$ is~$<9$.
The second statement then follows from Lemma~\ref{le:1}.
\end{proof}

From Lemma~\ref{le:eekfin} we can also compute $c$ near~$7$.

\begin{prop} \labell{prop:7easy} 
For $a \in [6\frac{11}{12},7]$ we have $c(a) = \frac13(a+1)$. 
Also $c(a) = \frac 83$ for $a \in [7, 7\frac 19]$.
\end{prop}

\begin{proof}
Since $c$ is continuous, it suffices to prove these identities for $a \in \Q$. 
First assume that $a<7$ and write $a = 6 +x$. Then
$$
\ww(a) \,=\, \left( 1^{\times 6}, x, w_8,\dots, w_M \right),
$$
where $0<w_i < 1-x$ for $i\ge 8$.   
The element $(3;2,1^{\times 6})\in \Ee_7$ gives the constraint
$c(a)\ge \mu_0= \frac13(a+1)$.
 
Since $1-x\le x/9 $, at least the first $9$ of the weights $w_8,w_9,\dots$ are equal.  
Hence, by Corollary \ref{cor:2},
we can fully pack all but the first~$7$ balls into one ball of size~$\la$ 
where $a=6+x^2+\la^2$. It remains to show that the $8$~balls of sizes
$$
W=(1,\dots,1,x,\la)
$$
fit into $B(\mu_0)$, that is, $W\cdot \mm \le \frac  d3(7+x)$ for all $(d;\mm)\in \Ee$.  
 
This is clear for classes in~$\Ee_7$.  
The classes in $\Ee_8 \less\, \Ee_7$ are 
$(4;2^{\times3}, 1^{\times5})$, 
$(5; 2^{\times6}, 1,1)$,
$(6; 3, 2^{\times7})$. 
The strongest constraint comes from
$(5;2^{\times 6},1,1)$ and equals
$$
\mu_1 \,=\, \ts{\frac 15}\bigl(12+x+\la\bigr).
$$
The desired inequality $\mu_1\le \mu_0$ is equivalent to $1+3\la\le 2x$.
Since $\la^2 = x(1-x)$ we need $13x^2-13x+1\ge 0$,
which is satisfied when $x\ge \frac {11}{12}$.
 
We know $c(7) = \frac 83$ by Corollary~\ref{cor:1}.  Therefore it suffices to show that
$c(7\frac 19) = \frac 83$. As above, since the nine balls 
$B(\frac 19)$ fully fill $B(\frac 13)$, we just need to check that
the finite number of elements in $\Ee_8$ give no obstruction to embedding 8~balls, 
seven of size $1$ and one of size $\frac 13$, into $B(\frac 83)$.  
\end{proof}

\begin{rmk} \labell{rmk:8}
{\rm
Similarly, there is an obstruction at $a=8$ given by the class 
$(d;\mm)=(6;3,2^{\times 7})$. 
For $a<8$ with $\ww(a) = (1^{\times 7}, a-7,\dots)$
this gives the constraint $\mu(a) = \frac {1+2a}{6}$, while for $a \ge 8$ we get
$\mu(a) = \frac {17}{6}$.  
Therefore $c(a) \ge \mu(a) \ge \sqrt a$ for
$\frac{8+3\sqrt 7}2 \le a \le 8\frac 1{36}$.  
However, unfortunately, one cannot argue as above
to show that $c = \mu$ on some interval $(8-\eps, 8 \frac 1{36}]$
because the auxiliary packings would involve $9$~balls and $\Ee_9$ is infinite. 
We shall prove that $c = \mu$ on $\bigl[ \frac{8+3\sqrt 7}2, 8\frac 1{36} \bigr]$ 
in Sections~\ref{ss:78} and \ref{ss:89} by different methods.
\diam
}
\end{rmk}

Now consider $\Ee_M$, $M \ge 9$.
We say that a tuple of integers $(d;\mm)=(d;m_1,\dots,m_M)$ is {\it ordered}\/
if $m_i \ge m_{i+1}$ 
 when $m_i\neq 0$, $m_{i+1}\neq 0$,
and if the $m_i$ with $m_i=0$ are at the end.
For instance, the elements of $\Ee_M$ are ordered in view of Definition~\ref{def:ee}.
To characterize~$\Ee_M$ 
when $M\ge 9$ we need the following definition.

\begin{defn}\labell{def:Crem} 
The {\bf Cremona transform} of an ordered tuple $(d;\mm)$
is
$$
\left( 2d-m_1-m_2-m_3;\, d-m_2-m_3,\, d-m_1-m_3,\, d-m_1-m_2,\, m_4,\, m_5, \,\dots \right).
$$ 
A {\bf standard Cremona move} $\Cr$ takes an ordered tuple $(d;\mm)$  
to the tuple obtained by ordering the Cremona transform of $(d;\mm)$. 
More generally, a {\bf Cremona move} is the composite of a Cremona transform with any permutation
of~$\mm$.
\end{defn}

Standard 
Cremona moves preserve $\Ee_M$ because they are achieved by Cremona transformations, which (modulo permutations of the $E_i$) are just reflections $A\mapsto A + (A\cdot C)\, C$
in the $(-2)$-sphere in the class $C := L-E_1-E_2-E_3$; cf.~\cite{LL2}. 
\footnote{
If $\om$ is a symplectic form on $X_M$ for which
the class~$C$ is represented by a Lagrangian sphere $S_L$, then 
the Cremona transformation can be realized by the Dehn twist in
$S_L$; cf.~Seidel~\cite{S}.
}  
In particular these moves preserve the intersection product and the first Chern class $c_1(M) := 3L -\sum_i E_i$.

\begin{prop}\labell{prop:eek}  
\begin{itemize}
\item[(i)]  
The following identities hold for all $(d;\mm)\in \Ee_M$.
\begin{equation}\labell{eq:ee}
\sum_i m_i = 3d-1, \qquad \mm\cdot\mm := \sum_i m_i^2 = d^2+1.
\end{equation}
\item[(ii)] For all pairs $(d;\mm), (d';\mm')$ of distinct elements of $\Ee_M$ we have 
$$\mm\cdot \mm' := \sum_i m_i\,m_i'\; \le \;d\,d'.
$$
\item[(iii)] 
A tuple $(d;\mm)$  satisfying the Diophantine conditions in 
Equation~\eqref{eq:ee} belongs to~$\Ee_M$ exactly if it may be reduced to  
$(0;-1,0,\dots,0)$  
by repeated standard Cremona moves.
\end{itemize}
\end{prop}

\begin{proof}  
The two equations in~(i) express the fact that any symplectically embedded $(-1)$-sphere~$E$ 
has $c_1(E) = 1$ and $E\cdot E = -1.$  
Since the elements in~$\Ee_M$ are all represented by embedded $J$-holomorphic spheres 
for generic $J$, 
part~(ii) holds by positivity of intersections. 

Part~(iii) for a class $(d;\mm)$ with $d=0$ is clear.
Part~(iii) 
for non-negative tuples $(d;\mm)$
may be deduced from Li--Li's arguments in~\cite[Lemma~3.4]{LL2}.  
In this paper the authors work in a more general context than ours, considering {\it all} symplectic forms on~$X_M$, while we consider only those symplectic forms with the standard first Chern class (or anticanonical class) $-K: = 3L-\sum E_i$.   
They also introduce many ideas, such as the symplectic genus.
However, the concept relevant here 
is
that of a {\it reduced}\/ class.  
This is a class 
$A := dL -\sum m_iE_i$ with
$$
d>0,\quad m_1\ge m_2\ge \dots \ge 0, \quad d\ge m_1+m_2+m_3.
$$
We can clearly assume that $M \ge 3$. In this case
they show that for every class $A$ with $A^2=-1$ and $c_1(A)>0$
there is a combination of Cremona moves and reflections 
$E_i \mapsto -E_i$
that transform~$A$ either into~$E_1$ 
(which corresponds to $(0;-1,0,\dots,0)$) 
or into a reduced class~$A'$. 
Each step consists of a Cremona transform, followed by an adjustment of signs 
to make the coefficients of the~$-E_i$  non-negative and then a permutation
to reorder the~$m_i$. Since the Cremona transform takes $(d;\mm)$ to
$$
(2d-m_1-m_2-m_3;\;d-m_2-m_3,\;d-m_1-m_3,\;d-m_1-m_2,\; m_4, \,\dots)
$$
this decreases~$d$ unless $A$ is reduced. 
Li--Li show by a simple algebraic computation that, because we start with a class with 
$A \cdot A\ge -1$ and the reduction process preserves the intersection form, 
the coefficient~$d$ cannot become negative.
Thus one stops the process when the coefficient~$d$ is at its minimum.  
If $d=0$ then the final class~$A_0$ is $(0;-1,0\dots,0)$.  
On the other hand, if~$A_0$ is reduced, another essentially algebraic argument 
(part~3 of their 
Lemma~3.4) shows that $A_0 \cdot E \ge 0$ for all solutions~$E$ to the equations~\eqref{eq:ee}.  
Hence $A_0 \notin \Ee_M$ since $A_0 \cdot A_0= -1$.  

It is easy to adapt the results of this lemma to our situation. 
We are interested here only in symplectic forms with the standard canonical class~$K$, 
and therefore cannot change the signs of the~$E_i$.
However, if there is a sequence~$\si$ of 
standard Cremona moves that takes a tuple $(d;\mm)$ that satisfies
\eqref{eq:ee}  to a tuple $(d';\mm')$
with some $m_i'<0$, then either  $d'=0$ and $(d';\mm') = (0;-1,0,\dots,0)$,
or $d'\ne 0$. But in the latter case  
$(d;\mm)$ cannot be in $\Ee_M$ since we would have
$$
E_{(d;\mm)}\cdot E \,=\, E_{(d';\mm')} \cdot E_i \,=\, m_i' \,<\, 0,
$$
where $E$ is the image of $E_i$ under the reverse sequence $\si^{-1}$ 
of Cremona moves. Therefore~(iii) must hold.
\end{proof}

\begin{defn} 
A class $E=(d;\mm)\in \Ee$ is called {\bf obstructive}
if
$\mu(d;\mm)(z) > \sqrt z$ on some nonempty interval~$I$.   
Further we say that  $E$ is {\bf obstructive at} $a$
 if $\mu(d;\mm)(a) > \sqrt a$.
\end{defn}

Thus our task is to understand enough about the obstructive classes to figure out the supremum of the corresponding constraint functions $\mu(d;\mm)$.

\begin{rmk}\labell{rmk:EM}
\rm   
(i)  
Later we will expend considerable effort to 
show that certain classes $E=(d;\mm)$ 
that satisfy the identities~\eqref{eq:ee}
do in fact lie in~$\Ee_M$.
In some cases, the corresponding constraints $\mu(d;\mm)(a)$ contribute to $c(a)$.  
However, in many other cases 
(for example the classes $E \bigl(b_k(i)\bigr)$ of Proposition~\ref{prop:bkiE} for $i \ge 3$, 
see Lemma~\ref{le:notobs})
the constraint $\mu(d;\mm)(a)$ does not contribute to~$c(a)$; 
rather $E$ influences~$c(a)$ because $E \cdot E' \ge 0$ for all $E' \in \Ee_M \less E$.  
For this positivity of intersections to hold, 
it is not necessary that $E \in \Ee_M$.  
As explained in the proof of Proposition~\ref{prop:eek},
it suffices that when we apply standard Cremona moves to~$E$ 
we do not arrive at a class with $d>0$ and some $m_i<0$, but instead end up at a reduced class.  (By \cite[Lemma~3.6]{LL2}, the class~$E$ then must have positive symplectic genus, 
and therefore cannot be represented by a smoothly embedded sphere.)  
However, it is just as difficult to check this condition as it is to check whether 
$E \in \Ee_M$, and, in fact, it turns out that $E \in \Ee_M$ in all cases of interest to us here.

\MS
\NI (ii)   
As we show in Proposition~\ref{prop:7finite}, there are only finitely many tuples
$(d;\mm) \in \Ee_M$ that are obstructive at some $a \ge 7$.
In fact, there are precisely $13$~such classes;  
the class $\bigl(3;2,1^{\times 6}\bigr)$ centered at $a=7$, 
another $8$~classes that contribute to $c(a)$ as described in Theorem~\ref{thm:78}, 
and $4$~more classes listed in Lemma~\ref{le:7k}, 
that are \lq\lq hidden" in the sense that they contribute nothing new to~$c(a)$.
Although we work mostly by hand, 
we do use the computer programs of Appendix~\ref{app:comp}
to prove  Corollary~\ref{c:zk}, which states
that there are no other relevant classes.
 
In contrast, there are infinitely many classes  
that are obstructive somewhere on the interval~$[1,7]$,
and we do not try to compute them all.
As shown in Example~\ref{exa:Ea2}, the part of the graph of 
an obstruction $\mu(d;\mm)(a)$ 
that lies above~$\sqrt a$ can be quite complicated
and need not have the scaling or positivity properties of~$c$
that are described in Lemma~\ref{le:1}.
Further, even though Corollary~\ref{cor:fin} states that at each point~$a$
where $c(a)>\sqrt a$ there are only finitely many obstructive classes with 
$\mu(d;\mm)(a) = c(a)$, we do not know if for some $a_0$ there are infinitely many classes with $\mu(d;\mm)(a_0) > \sqrt a_0$.
By Remark~\ref{rmk:7finite} this cannot happen when $a_0>\tau^4$.
When $a_0=\tau^4$, Proposition~\ref{prop:newmd} shows that 
there are infinitely many classes that are obstructive on an interval 
whose closure contains~$a_0$.  
However, because $c(\tau^4) = \tau^2$ no class is obstructive at $\tau^4$ itself. 
We have no relevant results when $a_0<\tau^4$.

\MS
\NI (iii)
We compute $c(a)$ for $a \le 7$ by looking not only at classes with 
$\mu(d;\mm)(a)>\sqrt a$ but also at some other classes that influence~$c(a)$ indirectly.
The most interesting of these are the classes described in Proposition~\ref{prop:newmd} 
that are made from the even terms of the Fibonacci sequence. 
They play a dual role.  
Though obstructive, they contribute nothing new to~$c(a)$ and so the corresponding graph is called the {\it ghost stairs}. 
Their importance is rather that they allow one to calculate~$c(a)$ 
at a series of points~$e_k$ where they are not obstructive; 
cf.\/ the proof of Corollary~\ref{c:42}.
\diam
\end{rmk}

%%%%%%%%%%%%%%%%%%%%%%%%%%%%%%%%%%%%%%%%%%%%%%%%%%%
\subsection{Outline of paper}\labell{ss:out}
%%%%%%%%%%%%%%%%%%%%%%%%%%%%%%%%%%%%%%%%%%%%%%%%%%%

Corollary~\ref{cor:wgt} gives a formula for $c(a)$ that we can interpret using
the description of~$\Ee$ contained in Proposition~\ref{prop:eek}.
However, this formula is not at all explicit, and the methods needed to understand it 
depend on the size of~$a$. 
One can compute $c(a)$ by direct methods when $a<\tau^4$, 
but for larger~$a$ our arguments require a deeper understanding of the constraint functions $\mu(d;\mm)(a)$.  
Therefore we begin in Section~\ref{s:basic} by developing
some tools to distinguish the obstructive classes in~$\Ee$.

First, we show in Proposition~\ref{prop:obs} that $(d;\mm)$ is
obstructive at~$a$ 
only if the vector~$\mm$ is almost parallel to $\ww(a)$.
If $\mm$ is  parallel to $\ww(a)$ for some $a$, then we call $(d;\mm)$ a 
{\it perfect}\/ obstruction at~$a$.  
Lemma~\ref{le:perf} shows that these 
elements determine the function~$c(z)$ 
for $z$ near~$a$, while Corollary~\ref{cor:perf} shows that the only perfect
obstructions occur at the numbers~$b_n$ of the Fibonacci stairs. 
Second, we show in Lemma~\ref{le:I} that if 
$\mu(d;\mm)(a)>\sqrt a$ on the interval~$I$,
then~$I$ has a unique central point~$a_0$ distinguished by the fact 
that $\ell(a_0) = \ell(\mm)$ while $\ell(a)> \ell(\mm)$ for all other $a \in I$.  
(Here, $\ell(\mm)$ denotes the number of positive entries in $\mm$.)
Lemmas~\ref{le:atmost1} and \ref{le:irrel} describe other useful properties of 
obstructive classes.  
     
These are the basic results needed to determine $c(a)$ when $a \ge 7$.   
(Since by Proposition~\ref{prop:7finite} there are only finitely many 
obstructive $(d;\mm)$ for $a \ge 7$, we can analyze these 
on a case by case basis, without using more general results.)
However, we continue in Section~\ref{s:basic} with a deeper analysis of the functions $\mu(d;\mm)$, 
so that we can explain the relation of $c(a)$ to the 
lattice point counting problem. 
This analysis is based on Proposition~\ref{prop:mirror}
which derives  
surprising identities satisfied by weight expansions. 
(These are quadratic identities involving the weight expansions of~$a$ 
and its \lq\lq mirror" $\laa$.) 
This proposition 
also turns out to be helpful in understanding $c(a)$ on $[\tau^4,7]$, 
where there are infinitely many obstructive classes.
 
Our next main result is 
Proposition~\ref{prop:mainc}
which shows that the central point of~$I$
is the break point of $\mu(d;\mm)$ in the sense that
$\mu(d;\mm)$ is linear on each component of $I\less \{a_0\}$.   
Moreover,  
one can apply Proposition~\ref{prop:mirror} to show
that the coefficients of these linear functions
are remarkably close to those of the linear functions that occur in the counting problem.
In~\S\ref{ss:latt}, we explain this connection and prove
Theorem~\ref{thm:ECH} (which states that $c_{ECH} \ge c$).

In Section~\ref{s:JL} we calculate $c(a)$ for $a \in [1,\tau^4]$ by direct methods.
Theorem~\ref{thm:ladder} states that there are classes
 $E(a_n)$ and $E(b_n)$
in~$\Ee$ given by tuples $(d;\mm)$ constructed from the weight expansions
of ratios of odd Fibonacci numbers. 
As we see in Corollary~\ref{cor:ladder},
because these classes are perfect their very existence
together with the scaling property of~$c$
is enough to calculate~$c(a)$ in this range.
The difficulty here is to prove that these classes really do belong to~$\Ee$. 
In particular, 
describing what happens to these classes under Cremona moves involves establishing many quadratic identities for Fibonacci numbers.  
Therefore in Section~\ref{ss:fib.id} we develop an inductive method to prove such identities; 
cf.~Proposition~\ref{prop:proof}.  
The proof of Theorem~\ref{thm:ladder} is completed in~\S\ref{ss:fib.red}.
This section is essentially independent of Section~\ref{s:basic}.

We next compute $c(a)$ on $[\tau^4,7]$.
The obstruction $(3;2,1^{\times 6})$ centered at~$7$
gives the lower bound $c(a) \ge \frac{a+1}3$ on this interval, 
and our task in Section~\ref{s:tau7} is to show that no obstruction exceeds this one. 
One difficulty is that the quantity $y(a) := a+1-3\sqrt a$ tends to~$0$
as~$a$ approaches~$\tau^4$, 
permitting the existence of infinitely many obstructive classes;
cf.~Proposition~\ref{prop:newmd}.
Another is that the line $\frac{a+1}3$ does not pass through the origin. 
Therefore we can no longer use the scaling property of~$c$, 
which in the case of the interval $[1, \tau^4]$
allowed us to restrict attention to the points $a_n,b_n$.
Nevertheless, by using the results of Section~\ref{s:basic} we show in Proposition~\ref{p:fewpoints} that there 
is only a double sequence of relevant points.

One could then attempt a direct calculation of 
$c(a)$ at these points, combining more elaborate versions of the estimation techniques used in Section~\ref{s:78} with arithmetic results 
based on Corollary~\ref{cor:mirror2}. This is possible.
However it is very complicated and 
it turns out that there is a much easier proof.
The sequence $E(b_n)$ of perfect classes that determine the Fibonacci stairs really consists of 
two subsequences $E_k(0)$ and $E_k(1)$ that are the first two members of an infinite family  $E_k(i)$, $i \ge 0$, of sequences of \lq\lq nearly perfect" classes in~$\Ee$.  
The classes $E_k(2)$, $k \ge 1$, form the ghost stairs 
discussed in Remark \ref{rmk:EM}(iii), 
while the classes $E_k(i)$, $i>2$, are not obstructive 
by Lemma~\ref{le:notobs}.   
Nevertheless, as we show in Lemma~\ref{le:ek}, the fact that they are nearly perfect puts constraints on the possible obstructive classes.  
The desired conclusion follows by combining this result with Proposition~\ref{p:fewpoints}.

Section~\ref{s:78} carries out a detailed analysis of the obstructions in the interval $[7,9]$.  
The argument is based on the equality in Proposition~\ref{prop:obs}~(iv).  
This estimates the ``error" 
(the difference between~$\mm$ and a suitable multiple of $\ww(a)$) 
at a rational point $a=p/q$ in terms of the quantity $y(a) -1/q$,
where again $y(a) = a+1-3\sqrt a$.  
Since $y(\tau^4) = 0$, this estimate gets better the 
further~$a$ is from $\tau^4$ and the larger~$q$ is.  
One easy consequence is Proposition~\ref{prop:7finite},  
stating that there are only finitely many obstructive  
classes $(d;\mm)$ for $a \ge 7$.
The proof (in~\S\ref{ss:78}) that $c(a) = \sqrt a$ for $a \ge 8\frac 1{36}$ is also easy.
  
To work out exactly what the constraints are requires some computation.  
It would be possible, though very tedious, to do this entirely by hand.  
We have aimed to use the computer as little as possible, and so have developed quite 
a few techniques for estimating the error.
Since $y(a)-\frac 1q$ is negative when $a=7\frac 1k$, we must treat these points separately, 
by purely arithmetic means. Thus in this case we simply look for suitable solutions
$(d;\mm)$ of the Diophantine equations~\eqref{eq:ee} with centers at these points, 
using Lemma~\ref{le:atmost1} to limit possibilities.
This computation (in Lemma~\ref{le:7k}) finds several classes 
that contribute to $c(a)$ as well as some interesting ``hidden" classes for which 
$\mu(d;\mm)(a) = c(a)$ at just one point, namely the center.  
There are some other obstructive classes centered at points of the form
$a=7\frac 2{2k+1}$.  
(The table in Theorem~\ref{thm:78} lists all classes that contribute to $c(a)$.)
We show that there are no other obstructive classes in~\S\ref{ss:78} using estimates developed 
in~\S\ref{ss:greater7} as well as two computer programs that are described in Appendix~\ref{app:comp}.

Finally, Appendix~\ref{app:wt} explains the connection between our current definition of the weight expansion of~$a$ in terms of the continued fraction expansion of~$a$
and the definition used in~\cite{M}, which came from a blow up construction.   
The results here are no doubt well known; we included them for the sake of completeness.

\MS\MS\NI 
{\bf Acknowledgments.}   
We wish to thank Dylan Thurston for making some very helpful observations at the beginning 
of this project 
(he was the first person to point out a connection with Fibonacci numbers);  
Michael Hutchings for explaining the obstructions coming from embedded contact homology; 
and Helmut Hofer and Peter Sarnak for their inspiration and encouragement.

%%%%%%%%%%%%%%%%%%%%%%%%%%%%%%%%%%%%%%%%%%%%%%%%%%%%%%%%%
\section{Foundations} \labell{s:basic}
%%%%%%%%%%%%%%%%%%%%%%%%%%%%%%%%%%%%%%%%%%%%%%%%%%%%%%%%%

%%%%%%%%%%%%%%%%%%%%%%%%%%%%%%%%%%%%%%%%%%%%%%%%%%%%%%%%%
\subsection{Basic observations}\labell{ss:basob}
%%%%%%%%%%%%%%%%%%%%%%%%%%%%%%%%%%%%%%%%%%%%%%%%%%%%%%%%%

Given $a$ with weight expansion $\ww(a)$ of length $\ell(a) = M$
and  $(d;\mm)\in \Ee_M$, we define $\eps := \eps(a) = (\eps_1,\dots,\eps_M)$ by setting
\begin{equation} \labell{eq:eps}
\mm = \frac d{\sqrt a} \ww(a) + \eps.
\end{equation}
We will refer to the vector $\eps$ as the {\it error}, and to quantities such as $\sum \eps_i^2$ as the {\it squared error}.
We need to understand the function $\mu(d;\mm)(a)= \mm\cdot\ww(a)/d$  defined in 
Corollary~\ref{cor:wgt}.

Our arguments will be based on the following observations.

\begin{prop}  \labell{prop:obs} 
For all $(d;\mm)\in \Ee$ and $a$, we have

\begin{itemize}
\item[(i)]
$\mu(d;\mm) := \mu(d;\mm)(a) \,\le\, \sqrt{a} \sqrt{1 + 1/d^2}$;
\SSS

\item[(ii)]
$\mu(d;\mm)>\sqrt a\;\; \Longleftrightarrow\;\; \eps\cdot\ww > 0$;
\SSS

\item[(iii)] 
If $\mu(d;\mm)>\sqrt a$, then $E := \eps\cdot\eps = \sum \eps_i^2 < 1$;
\SSS

\item[(iv)]
Let $y(a): = a+1-3\sqrt a$ where $a=p/q$.  
Then
\begin{equation}\labell{eq:eps0}
{\ts -\sum\eps_i= 1+\frac d{\sqrt a}\bigl(y(a)-\frac 1q\bigr).}
\end{equation}
\end{itemize}
\end{prop}

\proof
Lemma~\ref{le:ww}
and Proposition~\ref{prop:eek} imply that
$$
 \mu (d;\mm) \,d \,=\, \ww\cdot \mm  \,\le\, \|\ww\| \, \|\mm\| \,=\, \sqrt{a} \sqrt{d^2+1}.
%\mu (d;\mm) \,d \,=\, \langle \ww, \mm \rangle \,\le\, \|\ww\| \, \|\mm\| \,=\, \sqrt{a} \sqrt{d^2+1} ,
$$
This proves (i).  
(ii) is immediate, while (iii) follows from~(ii) because
$$
d^2 + 1 \,=\, \mm\cdot\mm \,=\, \bigl( \tfrac d{\sqrt a} \ww(a) + \eps \bigr) \cdot \bigl( \tfrac d{\sqrt a}\ww(a) + \eps \bigr) \,=\, d^2 + 2 \tfrac d{\sqrt a}\ww(a) \cdot \eps + \eps\cdot\eps.
$$
Finally, to prove (iv) observe that
$$
3d-1 \,=\, \ts\sum m_i \,=\, 
\tfrac d{\sqrt a} \bigl( a+1-\frac 1q \bigr) + \ts\sum\eps_i.
$$  
Hence 
$$
\tfrac d{\sqrt a} \bigl(a+1-3\sqrt a\bigr) - \tfrac d{\sqrt a\, q} + 1 + \ts\sum \eps_i = 0.
$$
This completes the proof.
\proofend

\begin{rmk}\labell{rmk:near}
\rm
(i)  
Proposition~\ref{prop:obs}~(iii) implies that an element $(d;\mm)\in \Ee$ 
gives an obstruction at~$a$ (i.e.\ has $\mu(d;\mm)(a)>\sqrt a$)
only if the vector~$\mm$ is ``almost parallel'' to the vector $\ww(a)$.
In particular, if we are interested in solutions that provide obstructions when $a = k+x$ 
for $x\in (0,1)$ we need the first~$k$ entries to be equal within the allowable error.  
As we will see in Lemma~\ref{le:atmost1} below, 
this means that the first~$k$ entries of~$\mm$ 
must lie in the set $\{m_1, m_1-1\}$, with at most one entry different from the others.  
Some elements of $\Ee$ with $d\le 9$ that satisfy these conditions 
for $a \in [6,8]$ are
$$
(2; 1^{\times5}),\;\; (3; 2, 1^{\times6}),\;\;
(5; 2^{\times6}, 1^{\times2}),\;\; (8; 3^{\times7}, 1^{\times2}).
$$
It turns out that these elements all do give obstructions. 

\m
(ii) 
Another noteworthy point is that $y(a)=0$ when $a=\tau^4$.
Therefore (iv) gives most information when $a-\tau^4$ is quite large, e.g.\ if $a>7$; see \S\ref{ss:greater7}.
\diam
\end{rmk}

The next result explains the basic structure of the constraints.
Throughout we  write $\ell(\mm)$ for the number of positive entries in $\mm$,
and $\ell(a)$ for the length of the weight sequence $\ww(a)$.

\begin{lemma}\labell{le:I} 
Let $(d;\mm)\in \Ee$, and suppose that $I$ is a maximal nonempty open interval such that
$\sqrt {a} < \mu(d;\mm)(a)$ for all $a\in I$.
Then there is a unique 
$a_0 \in I$ such that $\ell(a_0) = \ell(\mm)$.
Moreover $\ell(a) \ge \ell(\mm)$ for all $a \in I$. 
\end{lemma}

\begin{proof} 
Denote by $w_i(a)$ the $i$th weight of $a$ considered as a function of~$a$.  
Then it is piecewise linear, and is linear on any open interval that does not contain 
an element~$a'$ with length $\ell(a') \le i$.  
That is, the formula\footnote
{See Lemma~\ref{le:w1} for an explicit expression.}
for $w_i(a)$ can change only if it or one of the earlier weights becomes zero.

Therefore if $\ell(a) > \ell(\mm)$ for all $a\in I$, the function $\mu(d;\mm)(a)$ is linear in~$I$. But this is impossible since the function $\sqrt {a}$ is concave and $I \subset (1,9)$ is bounded. 

Thus there is $a_0\in I$ with $\ell(a_0) \le \ell(\mm)$.
On the other hand, if $\ell(a)<\ell(\mm)$, then $\sum_{i\le \ell(a)} m_i^2 <d^2+1$, 
so that
$$
|\ww\cdot \mm| \,\le\, \|\ww\| \sqrt{\sum_{i\le \ell(a)} m_i^2} \,\le\, d\|\ww\| \,=\, d\sqrt a.
$$
Hence $\mu(d;\mm)(a)\le \sqrt a$, i.e.\/ $a \notin I$.

The uniqueness follows from the properties of continued fractions.
We claim that if $b>a$ and $\ell(b) = \ell(a)$ then there must be some
number $y\in (a,b)$ with $\ell(y) < \ell(a)$.
Since such~$y$ cannot be in~$I$, this gives the required uniqueness.
To prove the claim,
let~$a$ have continued fraction expansion $[\ell_0;\ell_1,\dots,\ell_N]$ 
and consider the functions $x_j(a)$ as in
equation~\eqref{eq:xs}. If~$N$ is even, the function $x_N(z)$ decreases as~$z$ increases.  
Hence $\ell_N$ increases and so the next number $z>a$ with length $\le \ell(a)$ is 
    $[\ell_0;\ell_1,\dots,\ell_{N-1}]$ 
which has length $<\ell(a)$.  
Similarly, if we look for numbers $z<a$ with $\ell(z) \le \ell(a)$,
then the first one is $[\ell_0;\dots, \ell_N-1]$. Similar arguments apply if~$N$ is odd.
\end{proof}

\begin{cor}\labell{cor:fin} 
Suppose that $c(a)>\sqrt a$.  
Then 
\begin{itemize}
\item[(i)] 
There are (possibly equal) elements $(d^\pm;\mm^\pm)\in \Ee$ and $\eps>0$ such that 
$$
c(z) \,=\,
\left\{
\begin{array}{ll}
\mu(d^-;\mm^-)(z) &\mbox{ for all }\; z \in (a-\eps,a], \\
\mu(d^+;\mm^+)(z) &\mbox{ for all }\; z \in [a,a+\eps).
\end{array}
\right.
$$
\item[(ii)]  
On each of the intervals in (i) there are rational numbers  $\al, \be \ge 0$ such that
$c(a) = \al+\be a$.

\s
\item[(iii)]  
The set of $(d;\mm)$ such that $c(a) = \mu(d;\mm)(a)$
is finite.
\end{itemize}
\end{cor}

\begin{proof}
Since $c(a)>\sqrt a$, there exists $D \in \NN$ with $\sqrt{1+1/D^2} < c(a)/\sqrt a$.
If $(d;\mm) \in \Ee$ is such that $\mu(d;\mm)(a) = c(a) > \sqrt a$,
then $d \le D$ by Proposition~\ref{prop:obs}~(i). 
But there are only finitely many elements $(d;\mm) \in \Ee$ with $d\le D$.   
Since $c(a)$ is continuous by Lemma~\ref{le:1}, we must have
$\sqrt{1+1/D^2} < c(z)/\sqrt z$ for all~$z$ sufficiently close to~$a$. 
Further, 
as we have seen in the proof of Lemma~\ref{le:I},
each function $\mu(d;\mm)(z)$ is piecewise linear, and it has rational coefficients because 
the weight functions $w_i(z)$ do.  
Hence, 
$c(z)$ is the supremum of a finite number 
of rational linear functions. This proves~(i) and (iii).  
Moreover, if 
near~$a$ we write $c(z) = \al+\be z$ for some rational numbers $\al, \be$, 
then $\be\ge 0$ since $c$~is nondecreasing, while
$\al\ge0$ because of the scaling property in equation~\eqref{eq:scal}.
\end{proof}

Let us call an element $(d;\mm)\in \Ee$ {\it perfect}\/ if~$\mm$ ~
is a multiple of the weight vector $\ww(b)$ of some $b>1$.
The next lemma combined with Corollary~\ref{cor:fin} shows that these elements determine $c(a)$   for $a$ near~$b$. 

\begin{lemma}\labell{le:perf} 
Suppose that $(d;\mm)\in\Ee$ is perfect: $\mm = \ka\,\ww(b)$ for some $b>1$.
Then
\begin{itemize}
\item[(i)]
$\mu(d;\mm)(b)=c(b) > \sqrt b$, and $(d;\mm)$ is the only class with $\mu(d;\mm)(b) = c(b)$.

\s
\item[(ii)]
$\mm = q\;\!\ww(b)$ where $b= p/q$ in lowest terms, and $b<\tau^4$.
\end{itemize}
\end{lemma}

\begin{proof}   
(i)
Since $(d;\mm) \in \Ee$ and $\mm = \ka\,\ww(b)$, we have
$d^2<d^2+1 = \mm \cdot \mm = \ka^2 \ww(b) \cdot \ww(b) =\ka^2 b$,
whence $d<\ka\sqrt b$.
Therefore, 
$$
\mu(d;\mm)(b) \,:=\ \frac{\mm \cdot \ww(b)}{d} \,=\, \frac{\ka b}{d} \,>\, \frac{\ka b}{\ka\sqrt b} \,=\, \sqrt b .
$$
Let $(d';\mm') \in \Ee$ be another solution.
Since $(d';\mm') \neq (d;\mm)$, positivity of intersections (part~(ii) of Proposition~\ref{prop:eek})
shows that $dd' \ge \mm \cdot \mm' = \ka\, \mm' \cdot \ww(b)$.
This and $d^2<d^2+1 =\mm \cdot \mm = \ka\, \mm \cdot \ww(b)$ yield
$$
\mu (d';\mm')(b) \,=\, \frac{\mm' \cdot \ww(b)}{d'} \,\le\, \frac d{\ka} \,<\, \frac{\mm \cdot \ww(b)}{d} \,=\, \mu(d;\mm)(b) .
$$

(ii)
Let $\ww(b) = (w_1,\dots,w_M)$.  
Since $w_M=\frac 1q$ and $m_M\in\Z$, we have $\ka=sq$ for some integer~$s$.
Equations~\eqref{eq:ee} and Lemma~\ref{le:ww} give
\begin{eqnarray*}
3d-1 &=& \sum m_i \,=\, \ka \sum w_i \,=\, sq\bigl(b +1-\tfrac 1q\bigr),\\
d^2+1&=&  \sum m_i^2 \,=\, \ka^2 \sum w_i^2 \,=\, (sq)^2b.
\end{eqnarray*}
If $s=1$, then $3d = q(b+1)$ so that
$$
1+b-3\sqrt b = 3\tfrac dq- 3\tfrac{\sqrt {d^2+1}}q < 0.
$$
Thus $\sqrt b <\tau^2$. 

Otherwise, adding the two above equations gives $s|d(d+3)$.  
But the first equation shows that $s,d$ are mutually prime.  
Therefore $s|d+3$ and $s|3d-1$; hence $s|10$. Therefore $s=2$, $5$ or~$10$.

The identity $(d+3)^2 = (d^2+1) + 2(3d-1) +10$ shows that $s^2|2(3d-1) +10$.  
If $2|s$ this means that $d$ is even which is impossible since $3d-1$ is even.  

If $s=5$ then $3d +4= 5q(b+1)$ so that
\begin{equation} \labell{eq:yb}
y(b) \,:=\, 1+b-3\sqrt b = \tfrac 1{5q}\bigl(3d + 4- 3\sqrt {d^2+1})> 0
\end{equation}
Thus $\sqrt b >\tau^2$. But this is impossible:  
For $(d;\mm) \in \Ee$ must have nonnegative intersection with the class $(3;2,1^{\times 6})\in \Ee$.  
Therefore
$$
2m_1+m_2+\dots + m_7 \le 3d.
$$
If $b \in [\tau^4,7]$, then $\mm=5q(1^{\times 6},b-6,\dots)$, and we obtain
$$
5q (1+b) \le 3d,  
$$
a contradiction.  
% the argument below is the easiest I found. The first argument excluding p=1 is
% necessary, without it, the algorithm after it gets stuck at q=5.
Assume now that $b=7 \frac rq >7$.
Assume first that $b=7 \frac 1q$.
Then $3d-1= 40q$ and $d^2+1=25 q^2(7+\frac 1q)$.
Solving the first equation for~$d$ and inserting the result into the second equation,
we get the equation
$$
5q^2-29q+2 \,=\,0
$$
whose solutions are not integral.
Thus $b=7 \frac rq$ with $r \ge 2$.
By~\eqref{eq:yb}, 
$$
q \,=\, \frac{3d+4-3\sqrt{d^2+1}}{5 y(b)} \,<\, \frac{4}{5 y(b)} \,<\, \frac{4}{5 y(7)} \,<\, 13,
$$
whence $q \le 12$.
Thus $b \ge 7 \frac{2}{12}$, and so $q<\frac{4}{5y(7 \frac 16)} <6$, whence $q\le 5$.
Thus $b \ge 7 \frac{2}{5}$, and so $q<\frac{4}{5y(7 \frac 25)} <4$, whence $q\le 3$.
Thus $b \ge 7 \frac{2}{3}$, and so $q<\frac{4}{5y(7 \frac 23)} <3$, whence $q\le 2$.
Thus $b \ge 8$, and so $q<\frac{4}{5y(8)} <2$,                      whence $q=1$.
Thus $b \ge 9$, and so $q<\frac{4}{5y(9)} <1$, which is impossible.
\end{proof}

\begin{remark}
\rm
We show later that
the only perfect elements are those at the numbers~$b_n$ 
of the Fibonacci stairs; see~Corollary~\ref{cor:perf}.
However there are many nearly perfect elements that are relevant to the problem such as 
the classes $E(a_n)$ of Theorem~\ref{thm:ladder} and
the classes $E\bigl(b_k(i)\bigr)$ of Proposition~\ref{prop:bkiE}.  
These elements are perfect except for some adjustments on the last block.
\diam
\end{remark}

The next lemma expands on the first part of Remark~\ref{rmk:near}.

\begin{lemma}  \labell{le:atmost1}
Assume that
$(d;\mm)\in \Ee$ is such that $\mu (d;\mm)(a) > \sqrt{a}$.  
Let $J := \{k, \dots, k+s-1\}$ be a 
block of $s\ge 2$ consecutive integers for which $w(a_i)$, $i \in J$, is constant.
Then

\begin{itemize}
\item[(i)] 
One of the following holds: 
\begin{eqnarray*}
&& m_k = \dots = m_{k+s-1}  \quad\text{or}\\
&& m_k = \dots = m_{k+s-2} = m_{k+s-1}+1 \quad\text{or}\\
&& m_k-1 = m_{k+1} = \dots = m_{k+s-1} .
\end{eqnarray*}

\item[(ii)]  
There is at most one block of 
length $s\ge 2$ on which the $m_i$ are 
not all equal.
\SSS

\item[(iii)]
If there is a block $J$ of length $s$ on 
which the $m_i$ are not all equal then  $\sum_{i\in J} \eps_i^2 \ge \frac {s-1}s$.
\end{itemize}
\end{lemma}

\proof  Let $w_i(a) = x$ for $i\in J$.  
By Proposition~\ref{prop:obs} (iii) we have
$$
\sum_{i=k}^{k+s-1} \left| \tfrac{dx}{\sqrt a} - m_i \right|\,\!^2 \;=\; 
\sum_{i=k}^{k+s-1} \eps_i^2< 1.
$$
Thus $\{m_k, \dots, m_{k+s-1}\}$ can contain at most two different integers, which must be neighbors if they are different, say $m$, $m+1$. We can also clearly assume that $m < \frac{dx}{\sqrt{a}} < m+1$. Therefore~(i) holds when $s<4$.  

So suppose that $s\ge 4$ and
assume that $m+1$ occurs $t$~times.
Set $v = \frac{dx}{\sqrt a}-n\in [0,1)$, where $n\in\Z$.  
Then the squared error on this block is
\begin{eqnarray*}
\sum_{i=k}^{k+s-1} \left| \tfrac{dx}{\sqrt a} - m_i \right|^2 &=&
\sum_{i=k}^{k+t-1} \left| v-1 \right|^2 + \sum_{i=k+t}^{k+s-1}\left| v\right|^2 \\
&=& t(v-1)^2 + (s-t)v^2 \\
&\ge& \tfrac {2(s-2)}{s}\;\text{ for all } v \in (0,1) \;\text{ if } t \in \{2,\dots,s-2\}.
\end{eqnarray*}
But $2(s-2)/s\ge 1$ when $s\ge 4$.  
This proves~(i).
Parts (ii) and (iii) follow from the fact that the minimum squared
error on a block of length $s$ on which the $m_i$ are not all equal is $1-\frac 1s\ge \frac12$.
\proofend

The following lemma will also be important for detecting 
potentially obstructive solutions $(d;\mm)$.

\begin{lemma} \label{le:irrel}
Let $(d;\mm) \in \Ee$ be such that $\mu(d;\mm) > \sqrt a$ for some~$a$ 
with $\ell(a) = \ell(\mm) = M$.
Let $w_{k+1}, \dots, w_{k+s}$ be a block, but not the first block, of $\ww(a)$. 

\begin{itemize}
\item[(i)]
If this block is not the last block, then
$$
\left| m_k- \left( m_{k+1} + \dots + m_{k+s} + m_{k+s+1}\right) \right| \,<\, \sqrt{s+2} .
$$
If this block is the last block, then
$$
\left| m_k - \left( m_{k+1} + \dots + m_{k+s} \right) \right| \,<\, \sqrt{s+1} .
$$

\item[(ii)]
Always,
$$
m_k- \sum_{i=k+1}^M m_i  \,<\, \sqrt{M-k+1} .
$$
\end{itemize}
\end{lemma}

\proof
(i)
We prove the first claim, the second claim is proven in the same way.
By definition~\eqref{eq:eps} of the errors, $m_i = \frac{d}{\sqrt a} w_i + \eps_i$ for all~$i$.
Since $w_{k+1} = \dots = w_{k+s}$ and $w_{k+s+1} = w_k - s \;\!w_{k+1}$, we have that
$$
m_k - \left( m_{k+1} + \dots + m_{k+s} + m_{k+s+1} \right) \,=\,
\eps_k - \left( \eps_{k+1} + \dots + \eps_{k+s} + \eps_{k+s+1} \right) ,
$$
and so 
$$
\left| m_k - \left( m_{k+1} + \dots + m_{k+s} + m_{k+s+1} \right) \right| 
\,\le\,
\left| \eps_k \right| + \left| \eps_{k+1} \right| + \dots + \left| \eps_{k+s} \right| + \left| \eps_{k+s+1} \right| .
$$
Since $\eps \cdot \eps = \sum \eps_i^2 < 1$ by Proposition~\ref{prop:obs}~(iii),
the latter sum is $< \sqrt{s+2}$.

\s
(ii)
If the block $w_{k+1}, \dots, w_{k+s}$ is the last block, then $M=k+s$, and so the stated estimate is the same as in~(i).
So assume that $w_{k+1}, \dots, w_{k+s}$ is not the last block.
Since $w_k =  s\;\! w_{k+1} + w_{k+s+1}$, we then have
$w_k < \sum_{i=k+1}^M w_i$.
Therefore, 
$$
m_k \,<\, -\eps_k + \sum_{i=k+1}^M \left( m_i + \eps_i\right) \,\le\, 
\sum_{i=k+1}^M m_i + \sum_{i=k}^M |\eps_i| \,<\, \sum_{i=k+1}^M m_i + \sqrt{M-k+1} ,
$$
as claimed.
\proofend

%%%%%%%%%%%%%%%%%%%%%%%%%%%%%%%%%%%%%%%%%%%%%%%%%%%%%%%%%%%%
\subsection{Some identities for weight expansions} \labell{ss:wt}
%%%%%%%%%%%%%%%%%%%%%%%%%%%%%%%%%%

This subsection establishes some rather surprising identities for weight expansions.  

Let $a>1$ be a rational number and consider its continued fraction expansion 
\begin{equation}\labell{eq:aa}
a \,:=\,  [\ell_0; \ell_1,\dots, \ell_N] \,=\, \ell_0 +\frac 1 {\ell_1 + \frac 1{\ell_2 + \dots}}.
\end{equation}
Usually, to avoid ambiguity, we assume that $\ell_N\ge 2$, but in this section only it is convenient to permit the case $\ell_N=1$ as well.
We define the sequence
$\al^{a} \,:=\, (\al_j^{a})_{j=0}^{N+1}$ by setting $\al^{a}_0=1$, $\al^{a}_1=-\ell_0$, and
\begin{equation}\labell{eq:itere}
\al_j^{a} \,=\, \al^{a}_{j-2} - \ell_{j-1}\al_{j-1}^{a}, \quad j=2,\dots, N+1.
\end{equation}
Similarly, define $\be^{a} = (\be_j^{a})_{j=0}^{N+1}$ by the same recursive formula, 
but starting with $\be^{a}_0=0$, $\be^{a}_1 = 1$. 
Thus $\be^{a}$ does not depend on~$\ell_0$.  
Both sequences alternate in sign.
%If $\ell_j=1$ for all $j$ the sequences $(|\al^{a}_j|)$ and $(|\be^{ao}_j|)$ are (renumbered) Fibonacci sequences. In general, they are known as Lucas numbers.

We chose the notation $\al^a$, $\be^a$ for these sequences because, as we now see, 
they are the coefficients of the linear functions $w_i(z)$ for $z$ lying on 
the appropriate side of~$a$.
Write the weight expansion $\ww(a)$ of~$a$ as
$$
\ww(a) \,=\, \bigl( 1^{\times \ell_0},(x_1(a))^{\times \ell_1}, \dots, (x_N(a))^{\times \ell_N} \bigr) 
       \,=\, \bigl( \underbrace{1,\dots,1}_{\ell_0},\underbrace{x_1,\dots,x_1}_{\ell_1},\dots \bigr).
$$
Then the weights $x_j := x_j(a)$ satisfy the recursive formula~\eqref{eq:itere} with $x_0=1$, 
and $x_1=a-\ell_0$, so that $x_j = \al_j^{a} +a \be^{a}_j$ for $j\le N$.  
If~$N$ is odd, this formula for $x_j(z)$, $j \le N$, continues to hold for all $z<a$ 
that are so close to $a$ that the $x_j(z)$ are positive.   
In this case,
$z = [\ell_0; \ell_1,\dots, \ell_N,h,\dots]$ where $h\ge 1$.
Similarly, if $N$ is even, it holds for $z>a$ and sufficiently close to~$a$.
This proves the following result.

\begin{lemma} \labell{le:w1} 
Let $a$ and $N$ be as above. 
Then if $N$ is odd there is
$\eps>0$ and $h \ge 1$ such that for $z\in (a-\eps,a)$ we have
$$
\ww(z) \,=\,  \bigl( 1^{\times \ell_0},\bigl(x_1(z)\bigr)^{\times \ell_1}, \dots, \bigl(x_N(z)\bigr)^{\times \ell_N}, \bigl(x_{N+1}(z)\bigr)^{\times h},\dots \bigr)
$$
where $x_j(z) = \al^{a}_j +z \be^{a}_j$ for $j\le N+1$.
If $N$ is even, the same statement holds for $z\in (a,a+\eps)$.
Moreover, in  
both
cases $x_j(z)$ is an increasing function of~$z$ for $j$~odd, 
and a decreasing function for $j$~even.
\end{lemma}

We now give a second description for the sequences $\ww(a)$, $\al^a$, and $\be^a$ 
in terms of the convergents of a related number $\laa$ (the mirror of $a$)
which helps to explain their symmetry properties.  

\begin{defn}\labell{def:mirror}
Let $\ell_j \ge 1$ for $0\le j\le N$, where $N\ge 1$. 
We define
$$
\raa := [\ell_0; \ell_1,\dots,\ell_N], \quad 
\laa := [\ell_N; \ell_{N-1},\dots,\ell_0],
$$
and call $\laa$ the {\bf mirror} of~$\raa$. 
The {\it convergents} $(\raa)_k$, $0 \le k \le N$, to~$\raa$ are defined by setting
$$
(\raa)_k \,:=\, [\ell_0;\ell_1,\dots,\ell_k] \,=:\, \frac{p_k(\raa)}{q_k(\raa)} \,=:\, \frac{p_k}{q_k},
$$
where $p_k(\raa)$, $q_k(\raa)$ are the numerator and denominator of 
the rational number represented by~$(\raa)_k$.
In particular, $(\raa)_N = p_N/q_N = a$, and for short we write $a=\raa$.    

We define the
{\bf normalized weight sequence} of $\raa$ as
$$
W(\raa) := q_N \ww(\raa) = \Bigl( \bigl(X_0(\raa)\bigr)^{\times\ell_0}, \dots, \bigl(X_N(\raa)\bigr)^{\times\ell_N} \Bigr).
$$   
In particular, $X_N = 1$ always.
Finally, we define the {\bf (signed) mirror} of a sequence  
$W := \bigl( X_0^{\times\ell_0}, X_1^{\times\ell_1},\dots, X_N^{\times\ell_N} \bigr)$ to be
$$
\Hat W := \Bigl( X_N^{\times\ell_N}, \bigl(- X_{N-1}\bigr)^{\times \ell_{N-1}},\dots, \bigl((-1)^NX_0\bigr)^{\times\ell_0} \Bigr)
$$
where we reverse the order and change signs. 
\end{defn}

Note that the sequence of weights $\ww(a)$ is independent of the ambiguity in the $\ell_j$, 
but its block description and the $X_j$ do depend on this choice.  
 
The first part of the next lemma is well known.
We then show that the normalized weights of $a=\raa$ are equal to the numerators 
of the convergents of $\laa$.  
Further the coefficients $\al^a$ (resp.~$\be^a$) are the numerators (resp.\/ denominators) 
of the convergents of $a=\raa$ shifted by~$1$.  
Define $p_{-1}(a)=1$, $q_{-1}(a)=0$.

\begin{lemma} \labell{le:mirror}   
Let $a=\raa$ be as above. Then:
\begin{itemize}
\item[(i)] 
$\laa  = p_N(\raa)/p_{N-1}(\raa)$;  
in particular, $p_N(\raa) = p_N(\laa)$.

\s 
\item[(ii)]   
For $0\le j\le N$, we have 
$$
X_j(\raa) = |\al^{\laa}_{N-j}| ,  \qquad 
X_{N-j}(\laa) = |\al^{\raa}_{j}| .
$$
Further $|\al^a_j| = p_{j-1}(a)$ for all $a$ and $0 \le j \le N+1$.
  
\s  
\item[(iii)]   
Define $u = \rau := [\ell_1;\dots,\ell_N]$.  
Then $\lau = (\laa)_{N-1}$, and for $1\le j\le N+1$ we have
$$
|\be^a_j| = |\al^u_{j-1}| = q_{j-1}(a).
$$
\end{itemize}
\end{lemma}

\proof  
The following matrix identity holds by induction on $k$:
\begin{equation} \labell{eq:matrix}
\begin{pmatrix}\; \ell_0&1\\1&0 \;\end{pmatrix}\,
\begin{pmatrix}\; \ell_1&1\\1&0 \;\end{pmatrix}\, \cdots \,
\begin{pmatrix}\; \ell_k&1\\1&0 \;\end{pmatrix} \;=\; 
\begin{pmatrix}\; p_k(\raa)&p_{k-1}(\raa)\\q_k(\raa)&q_{k-1}(\raa)  \;\end{pmatrix} .
\end{equation}
(i) follows by considering its transpose. 
    
To prove (ii), observe that the elements $X_j: = X_j(a)$ are decreasing positive integers with 
$X_N = 1$, $X_{N+1} = 0$, and by Definition~\ref{def:wa} may be defined backwards by the 
iterative relation
$$
X_{N-j-1} = X_{N-j+1} + \ell_{N-j}\,X_{N-j}, \qquad j\ge 1.
$$
Similarly, $\al^a_0=1$ and, if we set $\al_{-1}=0$, the identity~\eqref{eq:itere} 
implies that the $|\al^a_j|$ are increasing positive integers
satisfying
$$
|\al^a_{j+1}| = |\al_{j-1}^a| +  \ell_{j} \,|\al^a_{j}|,\qquad j\ge 0 .
$$
But $\laa = [\ell_0';\dots, \ell_N']$ where $\ell_j' = \ell_{N-j}$.
Therefore $X_j(\raa) = |\al^{\laa}_{N-j}|$ for $j=0,\dots, N$.
Since $\llaa = \raa$, we also have that
$X_j(\laa) = |\al^{\raa}_{N-j}|$ for $j=0,\dots, N$.

To understand the relation between $\al^a_j$ and $p_{j-1} := p_{j-1}(a)$, 
note first that $p_0 = \ell_0 = |\al^a_1|$.  
Further, \eqref{eq:matrix} implies that
$p_j = p_{j-2} + \ell_j p_{j-1}$. 
Therefore $|\al^a_j|=p_{j-1}(a)$ for $j=0,\dots, N+1$, as claimed.
This proves~(ii).

The first claim in (iii) is obvious. 
To prove the second, note that the identity 
$|\be^a_j| = |\al^u_{j-1}|$ follows 
immediately from the definitions of~$\al$ and~$\be$.  
Therefore, by~(ii), $|\be^a_j | = |\al^u_{j-1}| = p_{j-2}(u)$.  
But if $v := [0;\ell_1,\dots,\ell_N]$, then
$\frac 1u = v = a-\lfloor a\rfloor$.  
Hence $p_{0}(u) = q_{1}(v) = q_1(a)$, and more generally,     
$p_{k}(u) = q_{k+1}(a)$.  
Hence $|\be^a_j | = q_{k-1}(a)$. 
\proofend

\begin{cor} \labell{cor:mirror}
\begin{eqnarray*}
\Hat W(\laa) &:=&\left( \bigl(X_N(\laa)\bigr)^{\times\ell_0},  \bigl(-X_{N-1}(\laa)\bigr)^{\times\ell_{1}}, \dots, 
\bigl( (-1)^N X_0(\laa) \bigr)^{\times\ell_N} \right) \\
&=& 
\left( \bigl(\al_0^a\bigr)^{\times\ell_0}, \bigl(\al_{1}^a\bigr)^{\times\ell_{1}},\dots, \bigl(\al_N^a\bigr )^{\times\ell_N} \right)
\end{eqnarray*} 
\end{cor}

We now show that there is a mirror version of the quadratic relation $\ww(a) \cdot \ww(a) = a$. Since $W(\raa) = q_N(\raa)\, \ww(a)$,  this relation is equivalent to the identity
$$
W(\raa)\cdot W(\raa) \,=\, p_N(\raa) \,q_N(\raa)
$$

We need the following lemma.
 
\begin{lemma}\labell{le:sumwt}  
Let $b \in \Q$ have weight expansion
$
\ww(b) = (1^{\times \ell_0},(y_1)^{\times \ell_1}, \dots, (y_N)^{\times \ell_N}),
$ 
and set $y_{N+1} := 0$.  Then, if $N=2J+1$ is odd
$$
y_{2k}   =  \sum_{j\ge k}^{J}\ell_{2j+1}\, y_{2j+1}, \quad 
y_{2k+1} = y_{2J+1} + \sum_{j> k}^{J}\ell_{2j}\, y_{2j}\;\; \mbox{ for each } k\ge 0,
$$
while if $N=2J$ is even
$$
y_{2k+1} = \sum_{j>k}^{J}\ell_{2j}\, y_{2j}, \quad 
y_{2k}   = y_{2J} + \sum_{j \ge k}^{J-1}\ell_{2j+1}\, y_{2j+1} \;\;\mbox{ for each } k\ge 0.
$$
\end{lemma}

\begin{proof}
This follows immediately from the construction of the weight expansion as in diagram~\eqref{figure.expa};
interpret each sum as the length of a side of some subrectangle.
\end{proof}

\begin{prop} \labell{prop:mirror}
Let $\raa = [\ell_0;\dots,\ell_N]$.
Then
$W(\raa) \cdot \Hat W(\laa) = p_N(\raa)$ if $N$ is even, 
and $=0$ if $N$ is odd.
\end{prop}

\begin{proof} 
Write 
$$
W(\raa) := (X_0^{\times \ell_0},\dots, X_N^{\times\ell_N}),\quad W(\laa) := 
(Y_0^{\times\ell_N},\dots, Y_N^{\times\ell_0}).
$$   
Then 
$\Hat W(\laa) = \bigl( Y_N^{\times\ell_0},\dots, ((-1)^N Y_0)^{\times\ell_N} \bigr)$.
Hence Lemma~\ref{le:mirror} part~(i) implies that
\begin{equation}\labell{eq:p}
p := p_N(\raa) = \ell_0 X_0 + X_1=\ell_N Y_0 + Y_{1}.
\end{equation}

Suppose first that $N=2J$. We must show that
\begin{equation}\labell{eq:summ}
\sum_{j=0}^J \ell_{2j} Y_{2(J-j)}\, X_{2j} = p+ \sum_{j=0}^{J-1} \ell_{2j+1} 
\, Y_{2(J-j)-1} \,X_{2j+1}.
\end{equation}
Lemma~\ref{le:sumwt} implies that
\begin{equation}\labell{eq:claim}
Y_{2(J-k)-1} = \sum _{j\ge 0}^k \ell_{2j}\,Y_{2(J-j)}
 \quad \mbox{ for each } k\le J.
\end{equation}
Hence
\begin{eqnarray*}
\sum_{j=0}^{J-1} \ell_{2j+1}\, Y_{2(J-j)-1}\, X_{2j+1}&=&
\sum_{j=0}^{J-1} \Bigl(\ell_{2j+1} X_{2j+1}\bigl(\sum_{s\le j} \ell_{2s} \,Y_{2(J-s)}\bigr)\Bigr)\\
&=& \sum_{s=0}^{J-1} \Bigl(\ell_{2s} \,Y_{2(J-s)} \bigl(\sum_{j\ge s}^{J-1} \ell_{2j+1}\,X_{2j+1}\bigr)\Bigr)\\
&=& \sum_{s=0}^{J-1} \ell_{2s} \,Y_{2(J-s)}\,\bigl( X_{2s}- X_{2J}\bigr)\\
&=& \sum_{s=0}^{J-1} \ell_{2s} \,Y_{2(J-s)}\,X_{2s} - X_{2J} Y_{1},
\end{eqnarray*}
where the penultimate equality follows from Lemma~\ref{le:sumwt} and the last one uses 
equation~\eqref{eq:claim}.
Therefore equation~\eqref{eq:summ} will follow if we show that
$(\ell_{N} Y_0 + Y_{1}) X_N = p$. Since  $X_N = 1$, this holds by
equation~\eqref{eq:p}.

The proof when~$N$ is odd is similar, and is left to the reader.
\end{proof}

\begin{cor}\labell{cor:mirror2} 
Let $x_j:= x_j(a)$, where $a=[\ell_0;\ell_1,\dots,\ell_N]>1$ 
and define $\al_j^{a}, \be^{a}_j$ as in equation~\eqref{eq:itere}.   
Then:
\begin{itemize}
\item[(i)] 
If $N$ is even,  $\sum_{j=0}^N \ell_j\, x_j\, \al^{a}_j = a$ and 
$\sum_{j=0}^N \ell_j\, x_j\, \be^{a}_j = 0$; 

\m
\item[(ii)]
If $N$ is odd, $\sum_{j=0}^N \ell_j\, x_j\, \al^{a}_j = 0$ and 
$\sum_{j=0}^N \ell_j\, x_j\, \be^{a}_j = 1$.
\end{itemize}
\end{cor}

\begin{proof}    
The sums involving $\al$ have the stated value by Proposition~\ref{prop:mirror} 
and Corollary~\ref{cor:mirror}.
To prove the claims involving~$\be^a$, write
$$
\rau \,=\, [\ell_1;\ell_2,\dots,\ell_N] \,=\, [\ell_0';\ell_1',\dots,\ell_{N'}'],
$$
where $N'=N-1$.  
Note that $X_i(u) = X_{i+1}(a)$, for $0\le i\le N'$.
Thus, 
with $\beta_0^a=0$, and
since $\be^a_j = \al^u_{i}$ where $i=j-1$ by Lemma~\ref{le:mirror}~(iii), we find that
\begin{eqnarray*}
q_N(a) \sum_{j=0}^N \ell_j\, x_j(a)\, \be^{a}_j &=& 
\sum_{j=1}^N \ell_j\, X_j(a)\, \al^{u}_{j-1}  \\
&=& \sum_{i=0}^{N'}\ell_i'\, X_i(u)\, \alpha_i^u .
\end{eqnarray*}
By what we have already shown, 
this sum is~$0$ when $N'$ is odd (i.e.~$N$ is even)
and equals the numerator $p_{N'}(u)$ of $u$ when $N'$ is even (i.e.~$N$ is odd).
But $p_{N'}(u)$ is the denominator of 
$\frac 1u = [0;\ell_1,\dots,\ell_N] = a-\lfloor a\rfloor$, and so equals $q_N(a)$.
The result follows.
\end{proof}

%%%%%%%%%%%%%%%%%%%%%%%%%%%%%%%%%%%%%%%%%%%%%%%%%%%%%%%%%%%%
\subsection{The nature of the obstructions}\label{ss:nat}
%%%%%%%%%%%%%%%%%%%%%%%%%%%%%%%%%%%%%%%%%%%%%%%%%%%%%%%%%%%%

We saw in Corollary~\ref{cor:fin} that near each point~$a$
where $c(a)>\sqrt a$, the function~$c$ is the supremum of a finite number of piecewise linear functions $\mu(d;\mm)$, 
and that each linear segment of~$c$ has the form  $z \mapsto \al + \be z$ with rational 
and nonnegative coefficients. The next example shows that the coefficients of the functions 
$\mu(d;\mm)$, though rational, are not restricted in this way even if we suppose that
$\mu(d;\mm)(z) > \sqrt z$.

\begin{example}  \label{exa:Ea2}
{\rm
Consider the class $(d;\mm) = \left( 10; 4^{\times 6}, 1^{\times 5} \right)$ in~$\Ee$. 
(Under the name $E(a_2)$, this class will play a role in Section~\ref{s:JL}.) 
Abbreviate $\mu(z) = \mu(d;\mm) (z)$.
We compute $d\,\mu(z) = 10\,\mu(z)$ on the interval $\left[ 6, 6 \frac 12 \right]$.
$$
\begin{array}{lrr}
\mbox{on } I_1 = \left[ 6, 6 \frac 14 \right] \colon 
& -6+5z \,\mbox{ on } \left[ 6\phantom{\frac 14}, 6 \frac 15 \right] , \quad
& 25 \phantom{ - 33z} \,\mbox{ on } \left[ 6 \frac 15, 6 \frac 14 \right] ; \\
\mbox{on } I_2 = \left[ 6 \frac 14, 6 \frac 13 \right] \colon 
& 4z \,\mbox{ on } \left[ 6 \frac 14, 6 \frac 27 \right] , \quad
& 44-3z \,\mbox{ on } \left[ 6 \frac 27, 6 \frac 13 \right] ; \\
\mbox{on } I_3 = \left[ 6 \frac 13, 6 \frac 25 \right] \colon 
& -13+6z \,\mbox{ on } \left[ 6 \frac 13, 6 \frac 38 \right] , \quad
& 38-2z \,\mbox{ on } \left[ 6 \frac 38, 6 \frac 25 \right] ; \\
\mbox{on } I_4 = \left[ 6 \frac 25, 6 \frac 12 \right] \colon 
& 6+3z \,\mbox{ on } \left[ 6 \frac 25, 6 \frac 37 \right] , \quad
& 51-4z \,\mbox{ on } \left[ 6 \frac 37, 6 \frac 12 \right] ;
\end{array}
$$
see Figure~\ref{fig.ex}. 
The figure also shows the graph of $\sqrt{z}$ and of
$c(z)$ (dashed) on $\left[ 6, 6 \frac 12\right]$,
which by Theorem~\ref{thm:main}~(i) is
$$
c(z) = \tfrac 52 \,\mbox{ on } \left[ 6, 6 \tfrac 14\right], \qquad
c(z) = \tfrac 25 \:\!z\,\mbox{ on } \left[ 6 \tfrac 14, 6 \tfrac 12\right] . 
$$
Note that $\ell (\mm) = \ell(a) = 11$ at $6 \frac 15$, $6 \frac 27$, $6 \frac 38$, $6 \frac 37$.
Also note that at $a=6 \frac 38$ we have $\mu(a) > \sqrt{a}$ 
while at $a =6 \frac 37$ we have $\mu(a) < \sqrt{a}$.

Similar results hold for the functions 
$\mu(d;\mm)$ given by the classes $E(a_n)$, $n>2$, of Theorem~\ref{thm:ladder}. 
For example, one can use Corollary~\ref{cor:mirror2} to show that these functions equal~$c(z)$
for~$z$ near~$a_n$. 
\diam

\begin{figure}[ht]
 \begin{center}
  \psfrag{z}{$z$}
  \psfrag{6}{$6$}
  \psfrag{15}{$6\frac 15$}
  \psfrag{14}{$6\frac 14$} 
  \psfrag{27}{$6\frac 27$}
  \psfrag{13}{$6\frac 13$}
  \psfrag{38}{$6\frac 38$}
  \psfrag{25}{$6\frac 25$}
  \psfrag{37}{$6\frac 37$}
  \psfrag{12}{$6\frac 12$}   
  \psfrag{52}{$\frac 52$}  
  \psfrag{c}{$c(z)$}
  \psfrag{s}{$\sqrt z$}    
  \psfrag{m}{$\mu(z)$}    
 \leavevmode\epsfbox{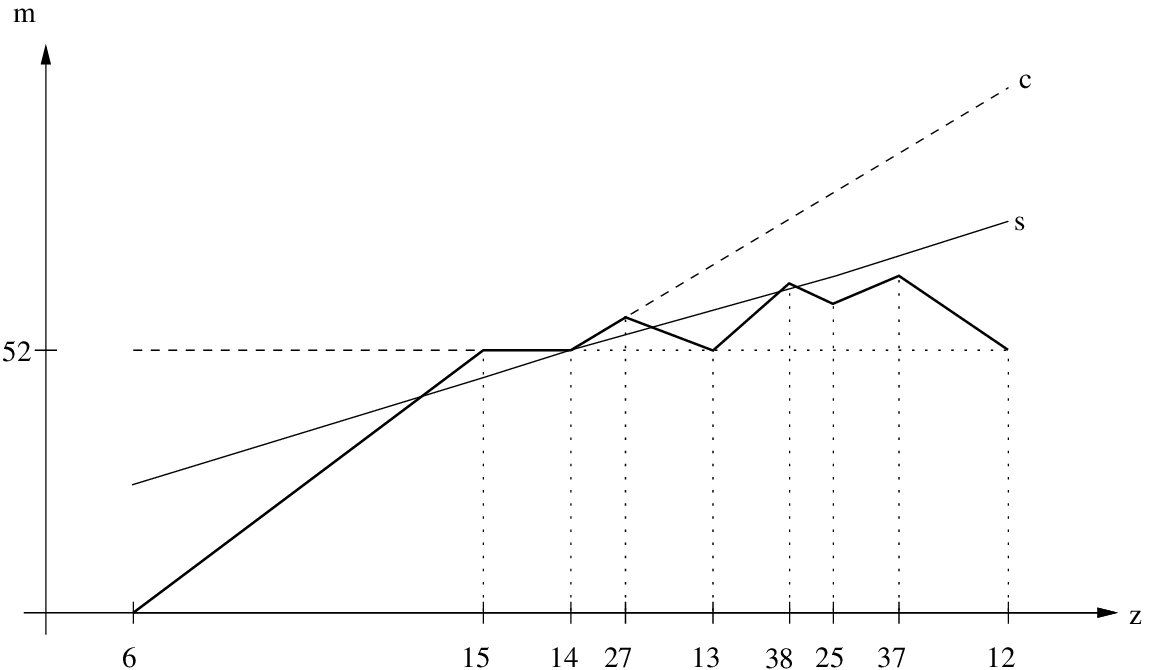}
 \end{center}
 \caption{The graph of $\mu$ on $\left[ 6,6\frac 12\right]$.}
 \label{fig.ex}
\end{figure}
%magnification: 100
%
%

}
\end{example}

We now show that although the coefficients $\al$, $\be$ may be negative, 
they are somewhat restricted.

\begin{prop} \labell{prop:mainc}
Let $(d;\mm)\in \Ee$ and $a\in \Q$ be such that $\ell(\mm) = \ell(a)$
and $\mu(d;\mm)(a)>\sqrt a$.   
Write $a =: p/q$ in lowest terms, let $m := m_M$ be the last nonzero entry in~$\mm$ 
and let~$I$ be the connected component of the set 
$\{ z \mid \mu(d;\mm)(z) > \sqrt z \}$ that contains~$a$. 
Then there are integers $A<p$ and $B<(m+1)q$ such that
$$
d\,\mu(d;\mm)(z) \,=\,
\left\{
\begin{array}{ll}
\phantom{(} 
A+Bz           & \mbox{ if }\;  z<a,\; z\in I, \\
(A+mp)+(B-mq)z & \mbox{ if }\;  z>a,\; z\in I.
\end{array}
\right.
$$
\end{prop}

We begin the proof by establishing the following lemma.
 
\begin{lemma} \labell{le:AB} 
Consider $(d;\mm) \in \Ee$ and $a= \frac pq$ (in lowest terms) 
such that $\ell(\mm) = \ell(a) =: M$.
Let $d \:\! \mu(d;\mm)(z) = A + Bz$ on a nonempty interval of the form $(a-\eps,a)$.    
Then 
there is $\eps'>0$ so that for $z' \in (a,a+\eps')$ 
$$
d \;\! \mu(d;\mm)(z') \,=\, A + Bz' +  m(p-qz') \,=:\, A'+B'z'.
$$
\end{lemma}

\begin{proof}  
Suppose first that $N$ is odd. Then by Lemma~\ref{le:w1} 
when $z<a$,  $x_j(z) = \al^{a}_j + z\be^{a}_j$, for $j \le N+1$. 
Further,
\begin{equation}\labell{eq:XNN}
x_{N+1}(z) \,=\, x_{N-1}(z) - \ell_N x_N(z) \,=\, 
\al^{a}_{N+1} + z\be^{a}_{N+1} = p-qz
\end{equation}
where the last equality holds by Lemma~\ref{le:mirror}.  

When $z'$ is just larger than~$a$, its $N$th multiplicity is 
$\ell_N - 1$, and $\ell_{N+1}$ is very big.  
(Since $\ell_N\ge 2$ this still gives an allowed set of multiplicities.)  
Hence for such~$z'$ the formula for the linear functions
$x_j(z'), j \le N,$ is unchanged, but now $x'_{N+1}(z') = x_{N-1}(z') - (\ell_N-1)x_N(z')$. 
(For clarity we denote by $x_{N+1}'$ the formula that holds for $z'>a$ and by $x_{N+1}$ the formula that holds for $z<a$.)
Note that because $\ell(\mm) = \ell(a)$, just one term from the $(N+1)$st block is counted 
in $\mu(d;\mm)(z')$. Hence, with $m := m_M$, we have
\begin{eqnarray*}
d\mu(d;\mm)(z') - (A + Bz') &=& -m x_N(z') + mx'_{N+1}(z') \\
&=& m\bigl(-x_N(z') + x_{N-1}(z') - (\ell_N - 1)x_N(z')\bigr)\\
&=& m\bigl(x_{N-1}(z') - \ell_N x_N(z')\bigr)  = m(p-qz'),
\end{eqnarray*}
where the last equality uses equation~\eqref{eq:XNN}.

Now suppose that $N$ is even.  
Then the formulas 
$x_j(z') := \al^a_j + z\be^a_j$
give the (beginning of the) weight expansion for~$z'$ just larger than~$a$.
As above, when~$z$ is just less than~$a$, we must modify the last multiplicities of~$a$, 
reducing $\ell_N$ by~$1$, and making $\ell_{N+1}$ arbitrarily large. 
Thus as above, the formulas for the weights $x_j(z)$, $j \le N,$ 
are unchanged but that for the $(N+1)$st weight is modified. 
As above we denote by $x_{N+1}'$ the formula that holds for $z'>a$ and by $x_{N+1}$ the formula that holds for $z<a$.
Then $x_{N+1}'(z') = -p+qz'>0$.
Further, if $d\mu(d;\mm)(z') = A' + Bz'$ for $z'>a$, 
we find for $z<a$ that 
\begin{eqnarray*}
d\mu(d;\mm)(z) -(A' + B'z) &=& -m x_N(z) + mx_{N+1}(z) \\
&=& m\bigl(-x_N(z) + x_{N-1}(z) - (\ell_N-1)x_N(z) \bigr) \\
   &=&  m\,x_{N+1}'(z) \,=\, -m(p-qz).
\end{eqnarray*}
Therefore $A+Bz = A'+B'z - m(p-qz)$, as claimed.
\end{proof}

To complete the proof of Proposition~\ref{prop:mainc} we need to estimate
the size of $A, B$.  
Here is an auxiliary lemma.

\begin{lemma}\labell{le:est} 
Let $\ell_0;\ell_1,\dots,\ell_N$ 
be any sequence of positive integers with $\ell_N \ge 2$, 
and let $\eta_j$, $j \ge 0$, be one of the sequences $|\al^{a}_j|$, $|\be^{a}_j|$.  
Then
$\sum_{j=0}^N \ell_j |\eta_j|^2 < \frac 12 |\eta_{N+1}|^2$.
\end{lemma}

\begin{proof} 
By definition
$\eta_j = \eta_{j-2} + \ell_{j-1}\eta_{j-1}$.
The inequality
\begin{equation} \label{eq:letak}
\ell_k \left(\sum_{j=0}^k \ell_j \eta_j^2 \right)\,\le\, \eta^2_{k+1} 
\end{equation}
holds for $k=0$, and 
may be proved for all larger~$k$ by induction.
Setting $k=N$ yields the lemma.
\end{proof}

\NI {\bf Proof of Proposition~\ref{prop:mainc}.}
Suppose first that~$N$ is odd, and  write $m_i = \frac d{\sqrt a} w_i(a) + \eps_i$ 
as in equation~\eqref{eq:eps}. For notational convenience, let us first assume that the~$m_i$ 
are constant on each of the blocks of length~$\ell_j$.  
Then define~$n_j$ to be this constant value on the $j$th block.  
If this assumption holds, then the $\eps_i$ are also constant on the blocks, 
and we denote their values by~$\de_j$. 
Then $d \;\! \mu(\mm;d)(a) = A + Ba$ where, by Lemma~\ref{le:w1}, we have
$$
A = \sum \ell_j\, n_j\, \al^{a}_j, \qquad B=\sum \ell_j \, n_j\, \be^{a}_j.
$$
Therefore, substituting
$n_j = \frac d{\sqrt a} x_j + \de_j$, we find
\begin{eqnarray*} 
A &=& \sum \ell_j\, n_j\, \al^{a}_j
\;=\; \sum \tfrac d{\sqrt a} \ell_j x_j \al^{a}_j +\sum \ell_j \de_j\al^{a}_j \\
&=& 0 +\sum \ell_j \de_j\al^{a}_j\\
&\le & \Bigl(\sum \ell_j\de_j^2\Bigr)\,\!^{1/2}\;\Bigl(\sum \ell_j |\al^{a}_j|^2\Bigr)\,\!^{1/2} \;<\;
\sqrt {E/2} \,|\al^{a}_{N+1}| \,<\, p.
\end{eqnarray*}
Here we used Corollary~\ref{cor:mirror2} for the third equality, 
and for the inequalities used the Cauchy--Schwarz inequality, 
$\sum \ell_j\de_j^2 =:E <1$ from Proposition~\ref{prop:obs}, 
Lemma~\ref{le:est} and 
finally the fact that $|\al^a_{N+1}| = p$ from Lemma~\ref{le:mirror}~(ii).
 
This argument is also valid if the $m_i$ are not constant on the blocks. 
In this case, by Lemma~\ref{le:atmost1} the values of~$n_j$ and $\de_j$ may vary by~$1$ 
over the entries of one block,
but that variation can be absorbed into the sum that gives~$\sqrt E$ 
and will not increase it above $\sqrt {E+1} < \sqrt 2$. 

Similarly,
\begin{eqnarray*} 
B &=& \sum \ell_j\, n_j\, \be^{a}_j \;=\;
\sum \tfrac d{\sqrt a} \ell_j x_j \be^{a}_j + \sum \ell_j \de_j\be^{a}_j \\
&=& \tfrac d{\sqrt a} +\sum \ell_j \de_j\be^{a}_j 
\;=:\; \tfrac d{\sqrt a} + S,\\
\end{eqnarray*}
where $S := \sum_{j\le N} \ell_j \de_j \be^{a}_j$.
By definition, $m = m_M = \frac d{\sqrt a}\,x_N + \de_N$, 
and $x_N = \frac 1q$.
Therefore, assuming that the~$m_i$ are constant on the blocks we have 
$$
B-(m+1)q
\,=\, \tfrac d{\sqrt a} -qm -q +S \,=\, -q(1+\de_N) + S.
$$
We need to show that $S < q(1+\de_N) = |\be^a_{N+1}|\,(1+\de_N)$.
If $\de_N\ge 0$ we may estimate~$S$ as before by
$$
S \,\le\, \sqrt {E} \Bigl(\sum_{j=0}^{N}\ell_j (\be^{a}_j)^2\Bigr)^{1/2}
\,<\,
\tfrac 1{\sqrt 2} |\be^a_{N+1}| \,<\, q.
$$
Now assume that $\de_N = -\de$ is negative and note that  
$\be_N^{a}>0$ because $N$ is odd.  
Therefore
$$
S \,:=\, \sum_{j\le N} \ell_j \de_j\be^{a}_j\le -\ell_N\de \be^{a}_{N} + 
\sqrt E \Bigl( \sum_{j=0}^{N-1}\ell_j (\be^{a}_j)^2 \Bigr)^{1/2}
\,\le\, \be^{a}_N \bigl( \sqrt E-\ell_N\de \bigr),
$$
where we used equation~\eqref{eq:letak} with $k=N-1$ and $\ell_{N-1}\ge 1$.
Since $\be_N^{a}<|\be^a_{N+1}|/2=q/2$ by the inductive formula, the desired result follows easily.

Suppose now that the $m_i$ are not constant on the $j$th~block.
If $j<N$, 
then, as before, we simply need to replace~$E$ by $E+1$ in the above estimates.  
It is easy to check that the argument still goes through. 

It remains to consider the case when the $m_i$ are not constant on the last block.
Define $\delta_N$ again by $m=m_M=\frac d{\sqrt a} x_N + \delta_N$.
By Lemma~\ref{le:atmost1} the last block of $\mm$ is either 
$(m+1)^{\times \ell}, m$ with errors $(\delta_N+1)^{\times \ell}, \delta_N$; or 
$m+1, m^{\times \ell}$ with errors $\delta_N+1, \delta_N^{\times \ell}$,
where $\ell := \ell_N-1$.
Note that $\delta_N =: -\delta$ is negative.
The sum $S_N$ of $\de_i \be^a_i$ over the last block is either
$\bigl( \ell(1-\de) - \de \bigr) \be^a_N$ or $\bigl( (1-\delta) - \ell \de \bigr) \be^a_N$.  
Since $\ell \ge 1$, 
in either case $S_N \le \bigl( \ell(1-\de)-\de \bigr) \be_N^a$.
But because
$(\ell+1)\be^a_N 
< \left| \be_{N+1}^a \right| =
q$, we can estimate $B-(m+1)q$ as follows:
\begin{eqnarray*}
B-(m+1)q 
&=& - q(1-\de) + \ts \sum_{j<N} \ell_j \de_j\be^{a}_j + S_N\\
&\le& - \be^a_N\Bigl((\ell+1)(1-\de) - \sqrt{E}-\ell(1-\de) + \de\Bigr)  
< 0. 
\end{eqnarray*}
This completes the proof when $N$ is odd. 
The case when $N$ is even is similar, and is left to the reader.
\QED

%%%%%%%%%%%%%%%%%%%%%%%%%%%%%%%%%%%%%%%%%%%%%%%%%%%%%%%%%%
\subsection{Connection to the lattice counting problem} 
\labell{ss:latt}
%%%%%%%%%%%%%%%%%%%%%%%%%%%%%%%%%%%%%%%%%%%%%%%%%%%%%%%%%%

In this section we prove Theorem~\ref{thm:ECH}, 
stating that $c_{ECH}(a) \ge c(a)$ for all $a \ge 1$.
Recall that for $a \ge 1$,
$$
c_{ECH}(a) \,:=\, \inf \left\{ \mu >0 \mid N(1,a) \preccurlyeq N(\mu,\mu) \right\} .
$$
The first step is to describe $c_{ECH}$ in another way.
As Hutchings pointed out\footnote{Private communication.},
the inequalities $N(1,a) \preccurlyeq N(\mu,\mu)$ can be understood in terms of counting 
lattice  points in triangles, as follows. 
Let $a\ge 1$ be irrational.
For each pair of integers $A,B\ge 0$, consider the closed triangle 
$$
T^a_{A,B} \,:=\, \bigl\{ (x,y) \in \R^2 \mid x,y\ge 0,\, x+ay \le A+aB \bigr\}.
$$
Thus the slant edge of $T^a_{A,B}$ has slope $-\frac 1a$ and passes through the integral point~$(A,B)$. 
Then the number $\#\bigl(T^a_{A,B} \cap \Z^2 \bigr)$ of integer points in the triangle 
$T^a_{A,B}$ is just the number of elements in $N(1,a)$ that are $\le A+Ba$.
We define 
\begin{equation} \labell{eq:kAB}
k_{A,B}(a) \,:=\, \tfrac{A+Ba}d ,
\end{equation}
where $d$ is the smallest positive integer such that
$$
\#\bigl( T^a_{A,B} \cap \Z^2\bigr) \le \tfrac 12 (d+1)(d+2).
$$
(Note that $N(1,1) = (0,1,1,2,2,2,3,3,3,3,4,\dots)$ has precisely
$\tfrac 12 (d+1)(d+2)$ entries that are $\le d$.)
Further, set
$$
K(a) \,:=\, \sup_{A,B \ge 0} \bigl\{k_{A,B}(a) \bigr\} .
$$ 
We extend the function $K$ to rational~$a$ by defining
\begin{equation} \label{e:irrat}
K(a) \,:=\, \sup_{ z<a, \,z \,\mbox{\scriptsize irrat}}\; K(z).
\end{equation}

\begin{lemma}\labell{lem:ECH} 
$K(a) = c_{ECH}(a)$ for all $a \ge 1$.
\end{lemma}

\proof
For each $\lambda >1$ we have $N(1,\la a) \preccurlyeq \la N(1,a)$.
Therefore, the conclusions of Lemma~\ref{le:1} hold for $c_{ECH}$ as well as for~$c$.
In particular, $c_{ECH}$ is continuous and  
nondecreasing.
Therefore, \eqref{e:irrat} also holds for $c_{ECH}$. It hence suffices to prove the lemma 
for irrational~$a$.

Fix an irrational~$a$.
If $c_{ECH}(a) < K(a)$ then one can find a rational number $\mu>c_{ECH}(a)$ 
and non-negative integers $A,B$ 
with $\mu< k_{A,B}(a)$.
Since $\mu>c_{ECH}(a)$ we have $N(1,a)\preccurlyeq N(\mu,\mu)$.
This inequality implies that for all non-negative integers~$A,B$ we have
$$
\# \left\{ p \in N(1,a) \mid p \le A+Ba \right\}
\,\ge\, 
\# \left\{ p \in N(\mu,\mu) \mid p \le A+Ba \right\} .
$$ 
The number on the left is $\# \bigl( T_{A,B}^a \cap \ZZ^2 \bigr)$,
while the number on the right is $\frac 12 (D+1)(D+2)$, 
where $D := \lfloor \frac{A+Ba}{\mu} \rfloor$.
This must be a strict inequality for some $A,B$.
To see this, let $u = (u_1, u_2, u_3, \dots)$ be the sequence of natural numbers obtained by arranging in increasing order
all the numbers on the LHS obtained by running through all pairs of integers $A,B \ge 0$.
Since $a$ is irrational, each number in $N(1,a)$ occurs with multiplicity~$1$.
The definition of $N(1,a)$ therefore shows that $u=(1,2,3,\dots)$.
On the other hand, the numbers in $N(\mu,\mu)$ appear with larger and larger multiplicity.
The sequence obtained in this way from the RHS therefore jumps by larger and larger amounts.

Consider $A,B$ such that this is a strict inequality.  
Then $k_{A,B}(a) = \frac{A+Ba}d$ where $d>D$. 
On the other hand because~$a$ is irrational and $\mu$ is rational, 
$D+1 > \frac{A+Ba}{\mu}> D$, so that 
$$
k_{A,B}(a) \,=\, \tfrac{A+Ba}{d} \,\le\, \tfrac{A+Ba}{D+1} \,<\, \mu .
$$  
Since this contradicts our assumptions, we conclude that $c_{ECH}(a) \ge K(a)$.  
 
To complete the proof, it suffices to show that
$K(a) \ge \mu$ for all $\mu<c_{ECH}(a)$.
For such~$\mu$ we have
$N(1,a) \not\preccurlyeq N(\mu,\mu)$.  Therefore there is~$A,B$ such that
$$
\# \left\{ p \in N(1,a) \mid p \le A+Ba \right\}
\,<\, 
\# \left\{ p \in N(\mu,\mu) \mid p \le A+Ba \right\} .
$$ 
With $D$ as before, this implies that $d \le D$, so that
$k_{A,B}(a) = \frac{A+Ba}{d} \ge \tfrac{A+Ba}{D} \ge \mu$.
Hence $K(a) = \sup k_{A,B}(a) \ge \mu$ as required. 
\proofend

We are now going to prove 
Theorem~\ref{thm:ECH} by direct calculation,
showing that for each of the constraints $(d;\mm)$ that contributes to~$c(a)$ 
there is a triangle that contributes to $K(a)$ in exactly the same way.
Therefore we will assume the results of Theorems~\ref{thm:main} and~\ref{thm:78}.

The key to understanding the relation between the functions $k_{A,B}$ of 
equation~\eqref{eq:kAB} and the number of lattice points in the triangles $T^a_{A,B}$ 
is the following lemma, that was explained to us by Hutchings.

\begin{lemma}\labell{le:ECH1} 
Suppose that $a$ is rational, abbreviate $T := T^{a}_{A,B}$, and suppose that
$$
\# (T\cap \Z^2) \,\le\, {\ts \frac 12(d+1)(d+2) + s-1 \,=\, \frac 12 (d^2+3d)} + s,
$$
where $s\ge 1$ is the number of integral points on the slant edge of~$T$. 
Assume that $(A,B)$ (resp.~$(A',B')$) is the integral point on the slant edge with smallest 
(resp.\ largest) $x$-coordinate.
Then there is $\eps>0$ such that 
$$
K(z) \ge {\ts \frac {A+zB}d}   \;\mbox{ if }\, z \in (a-\eps,a], \qquad
K(z) \ge {\ts \frac {A'+zB'}d} \;\mbox{ if }\, z \in [a,a+\eps).
$$
\end{lemma}

\begin{proof}  
Recall that $c_{ECH}$ and hence $K$ is continuous.
To prove the statement for $z<a$ it therefore suffices to consider irrational~$z$ of the form $z=a-\eps$.  
Then, for small enough $\eps>0$, the triangle $T^z_{A,B}$ 
contains $s-1$ fewer integral points than~$T$.  
Therefore
$k_{A,B}(z)\ge \frac {A+zB}d$, which proves the first statement.
Similarly,  the second statement holds because if $z=a+\eps$ is irrational and $\eps>0$ is 
sufficiently small, the triangle $T^z_{A',B'}$ contains $s-1$ fewer integral points than~$T$.
\end{proof}

\begin{lemma}\labell{le:ECH2} 
$K(b_n) \ge \sqrt{a_{n+1}}$ for all $n\ge 1$.
\end{lemma}

\begin{proof} 
Consider the triangle $T_n \subset \R^2$ with vertices $(0,0)$, $(g_{n+2},0)$ and $(0,g_n)$, 
where $g_n$ is the $n$\;\!th odd  Fibonacci number.  
Because $g_n, g_{n+2}$ are mutually prime and satisfy the identities
$$
g_n+g_{n+2} = 3g_{n+1},\qquad g_n g_{n+2} = g_{n+1}^2 + 1
$$
(see Section~\ref{ss:JL}),
we find that
\begin{eqnarray} \label{e:Tn}
\# (T_n\cap \Z^2) &=& {\ts\frac 12} (g_n+1)(g_{n+2}+1) + 1\\
&=&  {\ts\frac 12} (g_{n+1}^2 + 3g_{n+1}) + 2. \notag
\end{eqnarray}
Since $b_n = \frac {g_{n+2}}{g_{n}}$, we have $T_n=T_{0,g_n}^{b_n}$.
In view of~\eqref{e:Tn} we can apply 
Lemma~\ref{le:ECH1} with $s=2$ and $d=g_{n+1}$:
For some $\eps>0$ we have
$$
K(z) \ge \tfrac {z g_{n}}{g_{n+1}} \;\mbox { when }\, z\in (b_n-\eps, b_n] ,
$$
and
$$
K(z) \ge \tfrac {g_{n+2}}{g_{n+1}} \;\mbox { when }\, z\in [b_n, b_n+\eps) .
$$
In particular, $K(b_n) = \frac {g_{n+2}}{g_{n+1}} = \sqrt{a_{n+1}}$.
\end{proof}

\begin{cor}\labell{cor:ECH2} 
$K(a)\ge c(a)$ for all $a\in [1,\tau^4]$.
\end{cor}

\begin{proof}  
First observe that $c$ is the smallest continuous and nondecreasing function on $[1,\tau^4]$ 
that is $\ge \sqrt a$, has the scaling property of Lemma~\ref{le:1}, 
and also satisfies $c(b_n)= \sqrt{a_{n+1}}$.
On the other hand, we already remarked that $c_{ECH}$ and hence $K$
is continuous, nondecreasing and has the scaling property.
It is also easy to see that $K(a) \ge \sqrt a$, because the number of integer points 
in a large triangle approximates its area.  
Therefore 
$K(b_n) \ge c(b_n)$ implies that 
$K(a) \ge c(a)$ over the whole interval.
\end{proof}

\MS \NI 
{\bf Proof of Theorem \ref{thm:ECH}.}  
By Corollary~\ref{cor:ECH2} we only need to show $K \ge c$ on the interval 
$[\tau^4,\infty)$.
Since, as remarked there, $K(a) \ge \sqrt a$ for all~$a$, 
we just need to check that $K(a) \ge \mu(d;\mm)(a)$ for all $(d;\mm)$ that contribute to~$c$.
Recall from the proof of Proposition~\ref{prop:7easy} that the class $(3;2,1^{\times 6})$ 
gives the constraint $\frac{a+1}3$ on $[\tau^4,7]$.
Together with Theorems~\ref{thm:main}~(ii) and~\ref{thm:78}, 
we see that it suffices to check $K(a) \ge \mu(d;\mm)(a)$ for the nine classes 
in Table~\ref{t:0} below.
Each of these classes contributes on both sides of its center point.
It suffices to show that in each case there is a triangle that gives an equal constraint.  
Proposition~\ref{prop:mainc} shows which triangles to take:
if the constraint $(d;\mm)$ is centered at~$a$, then one should consider
$T := T_{A,B}^a = T_{A',B'}^a$ 
where $\mu(d;\mm)(z)$ equals $\frac 1d (A+Bz)$ to the left of~$a$ and 
$\frac 1d (A'+B'z)$ to the right.  
Because $c=\mu(d;\mm)$ in a neighborhood of the center point, 
this proposition together with 
Corollary~\ref{cor:fin} implies 
that $0 \le A < p$ and $mq\le B<(m+1)q$, so that the integral points $(A,B)$ and
$(A',B') = (A+mp, B-mq)$ are the first and last on the slant edge of~$T$, 
as required by Lemma~\ref{le:ECH1}. Therefore, it suffices to check that in each case 
the coefficient~$d$ occurring in $(d;\mm)$ satisfies the condition in Lemma~\ref{le:ECH1}.  
Thus the number $N(A,B)$ of integer points in~$T$ must be $\le N(d) := \frac 12(d+1)(d+2) + s-1$, 
where $s=m+1$ is the number of points on the slant edge of~$T$.  
In fact, as the following table shows we find that $N(A,B) = N(d)$ in each case.
\begin{equation}\labell{t:0}
\begin{array}{|c|l|c|c|r|c|r|}
 \hline
 a&(d;\mm)& (A,B)& (A',B')& N(A,B)&s& N(d)\\
 \hline
 7&(3;2,1^{\times 6})&(1,1)&(8,0)&11&2&11\\ 
 \hline
 7\frac 18&(48;18^{\times 7},3,2^{\times 7})&(7,17)&(121,1)& 1227&3&1227\\ 
 \hline
 7\frac 2{15}&(64;24^{\times 7},3^{\times 7},1^{\times 2})&(14,22)&(121,7)& 2146&2&2146\\
 \hline
 7\frac 17&(24;9^{\times 7},2,1^{\times 6})&(7,8)&(57,1)& 326&2&326\\
 \hline
 7\frac 2{13}&(40;15^{\times 7},2^{\times6},1^{\times 2})&(14,13)&(107,0)& 862&2&862\\
 \hline
 7\frac 15&(16;6^{\times 7},1^{\times 5})&(7,5)&(43,0)&154&2&154\\
 \hline
7\frac 14&(35;13^{\times 7},4,3^{\times 3})&(0,13)&(87,1)&669&4&669\\
 \hline
7\frac 12&(8;3^{\times 7},1^{\times 2})&(7,2)&(22,0)&46&2&46\\
 \hline
 8&(6;3,2^{\times 7})&(1,2)&(17,0)&30&3&30\\\hline 
\end{array}
\end{equation}

\begin{figure}[ht]
 \begin{center}
  \psfrag{AB}{$(A,B)$}
  \psfrag{ABs}{$(A',B')$}
  \psfrag{a}{$\alpha$}
  \psfrag{b}{$\beta$}
  \psfrag{c}{$\gamma$}
  \psfrag{d}{$\delta$}
  \psfrag{e}{$\vareps$}
 \leavevmode\epsfbox{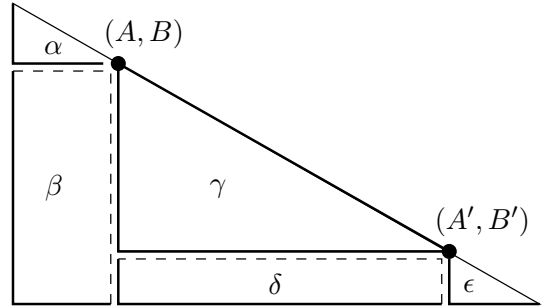}
 \end{center}
 \caption{The subdivision of the triangle $T^a_{A,B}$.}
 \label{fig.triang}
\end{figure}
%magnification: 100
%
%

Here we calculate $N(A,B)$ by
subdividing~$T$ into five parts labeled $\al,\dots,\eps$
as in Figure~\ref{fig.triang}.  
Each part
besides~$\gamma$ 
is half open, and includes the integer points on the heavy boundary edges 
but not those on the  
dashed boundary edges. 
For example, the rectangle~$\be$ includes the integer points on the~$x$ and $y$-axes, but not those on the 
(dashed) edges shared by $\al$, $\gamma$ or $\de$.
Thus $\#(\be\cap \Z^2) = AB$.  Further, because $A,B$ (resp.\ $(A',B')$) is the integer point on the slant edge with smallest (resp.\ largest) $x$ coordinate, we put all integer points on the slant edge  into $\ga$. Thus, we find that
$$
\#(\ga\cap \Z^2) \,=\, \tfrac 12 \Bigl((A'-A+1)(B'-B+1) -s\Bigr) + s.
$$
For example, in the case of the triangle $T^a_{7,17}$ with $a=7\frac 18$ and $s=3$, the numbers of integer points in $\al,\dots,\eps$ are
$7, 119, 979, 114$ and $8$, giving a total of $1227$.
 
This completes the proof of Theorem~\ref{thm:ECH}.\QED

\begin{rmk}\label{rmk:hid}
\rm  
(i)
On the interval $[7,8]$,
there are four other classes (described in Table~\ref{t:5}) 
with the property that $\mu(d;\mm)(a) = c(a)$ at their center points $a=\frac pq$, 
but that do not contribute otherwise to~$c(a)$. 
Let us look at their contribution to~$K$.
In each case $(A,B)=(-1,mq)$, where~$m$ is the last nonzero entry in~$\mm$,
so that $\mu(d;\mm)$ does not satisfy the scaling condition to the left.  
In the corresponding triangles the first point on the slant edge is 
$(A_1,B_1) = (-1+p, (m-1)q)$ and one can check as before that in each case $s=m$ and
$N(A,B) = N(d)$.
Therefore to the left of each center point we obtain the inequality
$$
K(a) \,\ge\, k_{A_1,B_1}(a) \,=\, \frac{p-1 + (m-1)q a}{d}.
$$
In each case, one can check that  $k_{A_1,B_1}(a)$ is precisely $c(a)$.
For example, at $7 \frac 18 = \frac{57}8 =: \frac pq$, we get 
$\frac{56 + 17 \cdot 8 a}{384} = \frac {7 + 17a}{48}$ 
which agrees with the first line in Table~\ref{t:1}.

\m
(ii)
Recall from the introduction that
the expected functorial properties of embedded contact homology 
should establish that $c_{ECH}(a) \le c(a)$ for all~$a$.
For $a \le 6 \frac 14$, $a=9$ and for $a \ge 11$ one can prove this 
by directly showing that for each $d$ the closed triangle 
with vertices $(0,0)$, $(d\,c(a),0)$, $(0, d\, \frac{c(a)}{a})$
contains at least $\frac 12 (d+1)(d+2)$ lattice points.  
For other values of~$a$ the lattice point formula
in Theorem~2.10 of~\cite{BR} should be useful. 
This is work in progress.
\diam
\end{rmk}

%%%%%%%%%%%%%%%%%%%%%%%%%%%%%%%%%%%%%%%%%%%%%%%%%%%%
\section{The Fibonacci stairs.}\labell{s:JL}
%%%%%%%%%%%%%%%%%%%%%%%%%%%%%%%%%%%%%%%%%%%%%%%%%%%%

In this section we establish the behavior of $c(a)$ for $a\le \tau^4$.

%%%%%%%%%%%%%%%%%%%%%%%%%%%%%%%%%%%%%%%%%%%%%%%%%%%%
\subsection{Main results.}\labell{ss:JL}
%%%%%%%%%%%%%%%%%%%%%%%%%%%%%%%%%%%%%%%%%%%%%%%%%%%%

Recall that the  Fibonacci numbers $f_n$ for $n \ge 0$ are recursively defined by 
\begin{equation}\label{eq:defFib}
f_0 = 0, \; f_1 = 1 
\quad \text{ and } \quad 
f_{n+1} = f_n+f_{n-1}, \;\, n \ge 1 .
\end{equation}
Denote by $g_n = f_{2n-1}$, $n \ge 1$, the sequence of {\it odd}\, Fibonacci numbers.  
The sequence $g_n$ starts with 
$$
1,\; 2,\; 5,\; 13,\; 34,\; 89,\; 233,\; 610,\; 1597,\; 4181,\; 10946, \;\dots
$$
The recursion formula $f_{n+1} = f_n + f_{n-1}$ implies the recursion formula
\begin{equation}  \labell{re:g1}
g_{n+1} \,=\, 3\:\!g_n - g_{n-1} .
\end{equation}
Using this and induction we find that
\begin{equation}  \labell{re:g2}
g_n^2 + 1 \,=\, g_{n-1} g_{n+1} .
\end{equation}
Set 
$$
a_n \,=\, \left( \tfrac{g_{n+1}}{g_n} \right)^2
\quad \text{ and } \quad
b_n  \,=\, \tfrac{g_{n+2}}{g_n} .
$$
Then $\dots < a_n < b_n < a_{n+1} < b_{n+1} < \dots$.
Since $\displaystyle{ 
\lim_{n \to \infty}} \tfrac{f_{n+1}}{f_n} = \tfrac{1+\sqrt 5}{2}=:\tau$, we have
$$
\lim_{n \to \infty} a_n \,=\, \lim_{n \to \infty} b_n \,=\, 
\tau^4 \,\approx\, 6.8541 .
%6.854101966249682043
$$

They key to establishing the Fibonacci stairs is the following result that states that there are elements in~$\Ee$ corresponding to the points $a_n, b_n$.

\begin{thm} \labell{thm:ladder}
\begin{itemize}
\item[(i)] 
Let $W(b_n) = g_n\,\ww(b_n)$. 
Then $E(b_n) := \bigl( g_{n+1}; W(b_n) \bigr) \in \Ee$.

\SSS
\item[(ii)]  
Let $W'(a_n)$ be the tuple obtained from $W(a_n) := g_n^2\,\ww(a_n)$ by adding an extra~$1$ at the end. 
Then $E(a_n) := \bigl( g_n g_{n+1}; W'(a_n) \bigr) \in \Ee$.
\end{itemize}
\end{thm}

\begin{cor}   \labell{cor:ladder}
Part (i) of Theorem~\ref{thm:main} holds.
\end{cor}

\begin{proof}[Proof of Corollary]
Since $E(b_n)$ is a perfect element, Lemma~\ref{le:perf}~(i) shows that
$$
c(b_n) \,=\, \mu \bigl(g_{n+1};W(b_n) \bigr) (b_n)
       \,=\, \tfrac{g_n}{g_{n+1}} \ww(b_n) \cdot \ww(b_n)
       \,=\, \tfrac{g_n}{g_{n+1}} b_n \,=\, \tfrac{g_{n+2}}{g_{n+1}} 
       \,=\, \sqrt {a_{n+1}} .
$$
Suppose that $c(a_{n}) > \sqrt{a_{n}}$ for some~$n$.
Then Corollary~\ref{cor:wgt} implies that there is $(d;\mm)\in \Ee$ such that
$$
\mm \cdot \ww(a_{n}) \,>\, d\,\sqrt{a_{n}}.
$$
Note that $(d;\mm) \ne \bigl( g_ng_{n+1}; W'(a_n) \bigr)$ since
$W'(a_n)\cdot \ww(a_{n}) = g_n^2 a_n =  g_ng_{n+1}\,\sqrt{a_{n}}.$
Therefore by positivity of intersections 
(part~(ii) of Proposition~\ref{prop:eek}) 
we must have
$$
d g_ng_{n+1} \,\ge\, \mm \cdot W'(a_n) \,\ge\, g_n^2\,\mm \cdot \ww(a_n),
\quad\mbox{ i.e. } \;\; d \,\sqrt{a_{n}} \,\ge\, \mm\cdot \ww(a_n).
$$
It follows that $c(a_n) = \sqrt{a_n}$ for all $n$. 
Thus $c(b_n) = \sqrt {a_{n+1}} = c(a_{n+1})$. 
Moreover,
$$
\tfrac{c(b_n)}{b_n} \,=\, \tfrac{\sqrt{a_{n+1}}}{b_n}
                    \,=\, \tfrac 1{\sqrt {a_n}} \,=\, \tfrac{c(a_n)}{a_n} .
$$
Hence $c$ is linear on the interval $[a_n,b_n]$ by the scaling property.
\end{proof}

\begin{corollary} \labell{cor:perf}
The classes $E(b_n)$ are the only perfect elements.  
\end{corollary}

\proof
On $[1,\tau^4]$, $c(a)$ is given by the Fibonacci stairs.
By Lemma~\ref{le:perf}~(i),
the perfect element~$E(b_n)$ is the only class giving the constraint $c(b_n)$ at~$b_n$.
This and 
Proposition~\ref{prop:mainc} show that 
the step of the stairs over $[a_n,a_{n+1}]$ centered at $b_n$ 
is the constraint~$\mu$ given by the perfect element~$E(b_n)$.
Lemma~\ref{le:perf}~(i) now shows that there cannot be another perfect element on $[1,\tau^4]$. 
By (ii)~of Lemma~\ref{le:perf} there is no perfect element on $[\tau^4,\infty)$. 
\proofend

We now turn to the proof of Theorem~\ref{thm:ladder}.
The proof of part~(i) is relatively easy; it is deferred to Corollary~\ref{c:41} 
since it is a special case of Proposition~\ref{prop:bkiE}.
To prove part~(ii) we first need to show that the elements~$E(a_n)$ satisfy 
the appropriate Diophantine equations, which is accomplished in~Lemma~\ref{le:dioph}.   Second, we must check that~$E(a_n)$ reduces correctly under Cremona moves.  
As we see~in Section~\ref{ss:fib.red},
the reduction process is quite complicated 
(and in fact is much more complicated than for the~$E(b_n)$), 
basically because the weight expansions~$\ww(a_n)$ involve
quadratic rather than linear functions in the Fibonacci numbers.
The intermediate Section~\ref{ss:fib.id} collects basic identities on Fibonacci numbers
and explains an inductive procedure useful for checking identities on them.

\begin{lemma}\labell{le:dioph}  
The tuples  
$E(a_n) := \bigl( g_ng_{n+1}; W'(a_n) \bigr)$ have integer entries and
satisfy equations~\eqref{eq:ee}.
\end{lemma}

\begin{proof} 
Consider $\bigl( g_ng_{n+1}; W'(a_n) \bigr)$.  
Since $a_n = \frac{g_{n+1}^2}{g_n^2}$, 
it follows from Lemma~\ref{le:ww} that the last entry $w_M$ of $\ww(a_n)$
is $\frac 1{g_n^2}$.
Therefore, the terms in $W(a_n) = g_n^2 \ww(a_n)$ and hence in $W'(a_n)$ are all integers.
Next, 
\begin{eqnarray*}
\sum_i W_i'(a_n) &=& g_n^2 \left(\sum_i w_i\right) + 1  \;=\; 
g_n^2 \left( a_n+1 - \tfrac1{g_n^2} \right) + 1 \\
&=& 
g_{n+1}^2 + g_n^2 \,=\, g_{n+1}^2 + g_{n+1}g_{n-1} -1 \\
&=& g_{n+1} \left( g_{n+1}+g_{n-1} \right) -1 \;=\; 3 g_{n+1}g_{n}-1 .
\end{eqnarray*}
Finally, $W'(a_n) \cdot W'(a_n) = g_n^4\, a_n + 1 = (g_ng_{n+1})^2 + 1$.
This completes the proof.
\end{proof}

%%%%%%%%%%%%%%%%%%%%%%%%%%%%%%%%%%%%%%%%%%%%%%%%%%%%%%%%%%%%%%%%
\subsection{Identities for Fibonacci numbers} \label{ss:fib.id}
%%%%%%%%%%%%%%%%%%%%%%%%%%%%%%%%%%%%%%%%%%%%%%%%%%%%%%%%%%%%%%%%

The proof that the classes $E(a_n)$ reduce correctly involves many small calculations. 
To avoid having to do them explicitly, we first explain a general inductive procedure 
whose conclusions are summarized in Proposition~\ref{prop:proof}. 
It is based on the following elementary result.
Recall that $f_k$ denotes the $k$th Fibonacci number defined in~\eqref{eq:defFib}.

\begin{lemma}\labell{sl:basic}  
Given any three distinct numbers
$s_0,s_1,s_2\ge 0$,  
there are rational constants $\la, \mu$ 
such that $f_{s_2+j} = \la f_{s_0+j} + \mu f_{s_1+j}$ 
for all $j\ge 0$. 
\end{lemma}

\begin{proof}  
The equations
$$
\la f_{s_0} + \mu f_{s_1} = f_{s_2},\qquad 
\la f_{s_0+1} + \mu f_{s_1+1} = f_{s_2+1}
$$
have a unique solution because $\frac{f_{s_0}}{f_{s_0+1}} \ne
\frac{f_{s_1}}{f_{s_1+1}}$ when $s_0\ne s_1$.  
Now apply the defining relation~\eqref{eq:defFib}.
\end{proof}

We will frequently use the following relations between Fibonacci 
numbers.
\begin{eqnarray}
f_k^2   &=& f_{k+1}f_{k-1} - (-1)^k  \label{eq:cassini} \\
f_{2k-1}  &=& f_k^2 + f_{k-1}^2      \label{eq:doubling.odd} \\
f_{2k}  &=& f_{k+1}^2 - f_{k-1}^2    \label{eq:doubling.even} 
\end{eqnarray}

\begin{lemma}\labell{le:Q}
For each $i\ge 0$ and $s\ge 0$, 
there is an identity of the form
$$ 
f_{s+i} f_{s} \,=\, \sum_{j \ge 0} a_{ij} f_{2s+j} + (-1)^s c_i,
$$
with a finite number of coefficients $c_i,a_{ij} \in \Q$ 
that do not depend on~$s$.   
Further, $c_i = -\sum_{j\ge 0} a_{ij}f_j$.
\end{lemma}

\begin{proof} 
By~\eqref{eq:defFib} it suffices to prove this for $i=0$ and $i=2$.   
We claim that for all $s \ge 0$, 
\begin{eqnarray}
5 f_sf_s     &=& -f_{2s}+2f_{2s+1}-2(-1)^s  ,       \label{e:ff0} \\
5 f_{s+2}f_s &=& f_{2s+1}+f_{2s+3}-3(-1)^s . \label{e:ff2}
\end{eqnarray}
For $s=0$, equation~\eqref{e:ff0} is true.
For $s\ge 1$, equation~\eqref{e:ff0} can be rewritten as
$$
5 f_s^2  \,=\, f_{2s+1}+f_{2s-1}-2(-1)^s .
$$
By~\eqref{eq:doubling.odd}, the RHS is $f_{s+1}^2+2f_s^2+f_{s-1}^2-2(-1)^s$, 
whence we need to check
\begin{equation} \labell{e:ffind}
3f_s^2 \,=\, f_{s+1}^2+f_{s-1}^2-2(-1)^s .
\end{equation}
Replacing $f_{s+1}$ by $f_s+f_{s-1}$, this becomes
$2 f_s^2 = 2 f_sf_{s-1} +2f_{s-1}^2-2(-1)^s$,
which is true since by~\eqref{eq:cassini}, 
$f_s^2 = f_{s+1}f_{s-1}-(-1)^s = f_sf_{s-1} + f_{s-1}^2 - (-1)^s$.

By~\eqref{eq:cassini} and~\eqref{eq:doubling.odd},
equation~\eqref{e:ff2} becomes 
$$
5 f_{s+1}^2+5(-1)^{s+1} \,=\, f_{s+2}^2+2f_{s+1}^2+f_s^2-3 (-1)^s,
$$
i.e.,
\begin{equation*}
3 f_{s+1}^2 \,=\, f_{s+2}^2 +f_s^2-2 (-1)^{s+1} 
\end{equation*} 
which is true by~\eqref{e:ffind}.
The formula for $c_i$ holds because $f_0=0$.
\end{proof}

\begin{prop}\labell{prop:proof}  
A quadratic identity of the form 
$$
Q(s) \,:=\, \sum_{i,j \ge 0} a_{ij} f_{s+i}f_{s+j} + \sum_{j \ge 0} b_jf_{2s+j} + (-1)^s c \,=\, 0
$$
holds for all $s\ge 0$ if it holds for any three distinct values of~$s$.  
Moreover, if the relation is homogeneous and linear 
(that is, if $a_{ij} = c =0$ for all $i,j$), 
then it suffices to check two values of~$s$.
\end{prop}

\begin{proof} 
Suppose that $Q(s)=0$ for $s=s_0,s_1,s_2$ where $0\le s_0<s_1<s_2$.
By Lemma~\ref{le:Q} one can convert
$Q(s)$ to an equivalent identity of the form 
$$
Q'(k) \,:=\, \sum_{j \ge 0} a_{j}f_{k+j} + (-1)^s c'=0 \quad \text{ where } k=2s .
$$

We first claim that $c'=0$.
By Lemma~\ref{sl:basic} there are constants $\mu, \la$ such that 
\begin{equation} \label{e:mula}
f_{2s_2+j} \,=\, \mu f_{2s_1+j} + \la f_{2s_0+j} 
\quad \text{ for all }\, j \ge 0.
\end{equation}
Since $Q'(k)=0$ for $k=2s_0,2s_1,2s_2$, and by~\eqref{e:mula},
$$
c' \,=\, \left( (-1)^{s_0+s_2}\mu+(-1)^{s_1+s_2}\la \right) c'.
$$
If $c' \neq 0$, we thus have $1 \in \left\{ \pm\mu \pm \la \right\}$.
This is impossible:
We use the recurrence relation~\eqref{eq:defFib} to extend
the sequence $(f_n), n\ge0$, to negative index~$n$. 
Note that
$f_{-n}=(-1)^{n+1}f_n$.
Then~\eqref{e:mula} holds for all~$j \in \ZZ$.
With $j=-2s_0$ and $j=-2s_1$ we get
$$
\mu = \frac{f_{2(s_2-s_0)}}{f_{2(s_1-s_0)}},
\quad 
\la = -\frac{f_{2(s_2-s_1)}}{f_{2(s_1-s_0)}}.
$$
Therefore, $\pm \mu \pm \la = 1$ exactly if 
$$
\pm f_{2(s_2-s_0)} \,=\, \pm f_{2(s_2-s_1)} + f_{2(s_1-s_0)} .
$$
The signs $++$ are impossible because $f_{m+n} = f_{m+n-1}+f_{m+n-2} > f_m+f_n$ for all even $m,n>0$. 
Further, $+-$, $-+$, and $--$ are impossible because $s_2>s_1>s_0>0$.

We have shown that $Q'(k) = \sum_{j \ge 0} a_{j}f_{k+j}$.
Recall that $Q'(k)=0$ for $k=2s_0,2s_1,2s_2$.
By Lemma~\ref{sl:basic}, the expression for $Q'(0)$ can be written as a linear
combination of the expressions for $Q'(2s_0)$ and $Q'(2s_1)$,
and the same is true for $Q'(1)$.
Therefore, $Q'(0)=0$ and $Q'(1)=0$.
This and the defining relation~\eqref{eq:defFib} show that $Q'(k)=0$ for all~$k$. 
In particular, $Q'(k)=0$ for all even~$k$, and so $Q(s)=0$ for all~$s$.
This proves the first statement. The second holds similarly. 
\end{proof}

In the subsequent sections, the following abbreviations will be useful.

\begin{defn}\labell{def:Lucas}
The $k$th Lucas number is defined to be 
$\ell_k = f_{k-1}+f_{k+1}$,\,  $k \ge 1$.
We set $F_k := \frac 13 f_{4k}$ and $L_k := \frac 13 \ell_{4k+2}$. 
\end{defn}

Then
\begin{gather*}
F_1 = 1,\, 
F_2 = 7,\, 
F_3 = 48,\, 
F_4 = 329,
\\
L_0 = 1,\,
L_1 = 6,\, 
L_2 = 41,\, 
L_3 = 281,\, 
L_4 = 1926 .
\end{gather*}
Further for $k \ge 0$ define the sequence~$H_k$ by
\begin{equation}\labell{eq:defH}
H_k \,=\, \tfrac 13\, f_{2k} f_{2k+2}.
\end{equation}
Then $H_0 = 0, \; H_1=1,\; H_2 = 8, \; H_3 = 56, \; H_4 = 385, \; H_5 = 2640, \dots$.

\begin{lemma} \labell{le:FLH} 
The following identities hold for all $k\ge 0$:
\begin{itemize}
\item[(i)] 
$F_{k+1} = L_k + F_k$;    

\vspace{0.2em}
\item[(ii)] $
L_{k+1} = 5\,F_{k+1}+L_k$;

\vspace{0.2em}
\item[(iii)] 
$H_{k+1} = H_k + F_{k+1} = \sum_{i=1}^{k+1} F_i$;

\vspace{0.2em}
\item[(iv)]  
$L_k = 5H_k + 1$.

\vspace{0.2em}
\item[(v)]  
$F_{k+1}^2-F_{k}F_{k+2} =1$.
\end{itemize}
\end{lemma}

\begin{proof}  
The second identity in (iii) follows from the first identity in~(iii) by induction.
All other identities have the form considered in Proposition~\ref{prop:proof}, and so it is enough to check each of them for at most three low values of~$k$.
\end{proof}

%%%%%%%%%%%%%%%%%%%%%%%%%%%%%%%%%%%%%%%%%%%%%%%%%%
\subsection{Reducing $E(a_n)$}  \label{ss:fib.red}
%%%%%%%%%%%%%%%%%%%%%%%%%%%%%%%%%%%%%%%%%%%%%%%%%%

We begin with a general remark about the reduction process.

\begin{rmk}\labell{rmk:order}
\rm  
Consider a tuple
$(d;\mm)$ that satisfies the Diophantine identities~\eqref{eq:ee}.
Proposition~\ref{prop:eek} states that $(d;\mm)\in\Ee$ exactly if 
it reduces to $(0;-1, 0,\dots,0)$ under standard Cremona moves, 
as defined in Definition~\ref{def:Crem}.
In fact, it clearly suffices to reduce~$(d;\mm)$ 
to a known element of~$\Ee$ by any sequence of Cremona moves. 
Each such move consists of an application of the Cremona transformation 
$$
(d; \mm)\mapsto (2d-m_1-m_2-m_3; d-m_2-m_3, d-m_1-m_3, d-m_1-m_2, \dots)
$$ 
followed by a choice of reordering.
It does not matter whether this reordering restores the natural order; 
all that is important is that in the end, after doing many such moves, 
we arrive at a known element of~$\Ee$. 
In fact, the reorderings chosen below all do restore the natural order.   
The point of this remark is that there is no need to {\it prove} this.
\diam
\end{rmk}

\begin{example} \labell{ex:an}
{\rm  
The first few elements $E(a_n)$ are
\begin{eqnarray*}
\bigl(g_2g_3;W'(a_2)\bigr) &=&\bigl(10; 4^{\times 6}, 1^{\times 5}\bigr),\\
\bigl(g_3g_4;W'(a_3)\bigr) &=&\bigl(65; 25^{\times 6}, 19^{\times1}, 6^{\times 3},1^{\times 7}\bigr),\\
\bigl(g_4g_5;W'(a_4)\bigr) &=&\bigl(442; 169^{\times 6}, 142^{\times1},
27^{\times 5}, 7^{\times 3},6^{\times1}, 1^{\times 7}\bigr).
\end{eqnarray*}
These values of $n$ are too low for our general arguments in \S\ref{ss:fib.even} and \ref{ss:fib.odd} to apply.  
Hence one proves that they reduce correctly by direct calculation.
}
\end{example}

The following list of Fibonacci numbers will be useful in the subsequent proofs.
\begin{center}
  \begin{tabular}{ | c || c | c | c | c | c | c | c | c | c | c | c |}
    \hline
    $n$     & 0 & 1 &  2  &  3  &  4  &  5 & 6  & 7 & 8  & 10 & 12  \\  
    \hline 
    $f_n$   & 0 & 1 &  1  &  2  &  3  &  5 & 8  & 13& 21 & 55 & 144  \\ 
    \hline 
\end{tabular}
\end{center}

%%%%%%%%%%%%%%%%%%%%%%%%%%%%%%%%%%%%%%%%%%%%%%%%%%%%%%%%%%%%%%%%%%%
\subsubsection{Reducing $E(a_n)$ for even $n$.}\labell{ss:fib.even}
%%%%%%%%%%%%%%%%%%%%%%%%%%%%%%%%%%%%%%%%%%%%%%%%%%%%%%%%%%%%%%%%%%%

Throughout this subsection we will consider $n=2m\ge 2$ to be a fixed even number.  
We will obtain an explicit expression for $W'(a_n)$ and then
examine its reduction by Cremona moves.
By Example~\ref{ex:an} it suffices to consider the case $n\ge 6$. 
Hence this case of Theorem~\ref{thm:ladder} 
follows from Propositions~\ref{prop:st1}, \ref{prop:st2}
and~\ref{prop:st3}.

For each fixed $n$, denote $k' := n-k$ and define
\begin{eqnarray}\labell{eq:uvdef}
u_k &:=& f_{2n-2k-1}^2 + 2\, H_k \,=\, f_{2k'-1}^2 + 2\,H_k 
\quad k = 0, \dots, m-1 , \\\notag
v_k &:=& 3\,F_{n-k} - 2\,F_k = 3\, F_{k'} - 2 \,F_k,\phantom{H_n}
\quad k = 1, \dots, m. 
\end{eqnarray}
Note that $u_k, v_k$ depend on $n$, though for simplicity the notation does not make this explicit. Also,
$v_k > 0$ for all $k\le m$ and $v_m=F_m$.  
However, the above formula for $v_k$ gives a negative number when $k>m$.  
This is why the expansion in Proposition~\ref{prop:a2m} below changes its form at the term~$v_m$.   
Note also that by equations~\eqref{eq:cassini} and \eqref{eq:defH}
\begin{equation}\labell{eq:uk}
f_{2n+1}^2 \,=\, f_{2n} f_{2n+2} + 1 \,=\, 3 H_{n} + 1.
\end{equation}
Hence we also have
\begin{equation}\labell{eq:uk2}
u_k \,=\, 3\,H_{k'-1} + 2\, H_k + 1.
\end{equation}

\begin{proposition}\labell{prop:a2m}
If $n=2m$ is even, then the continued fraction expansion of~$a_n$ is 
$$
[ 6; \underbrace{1,\, 5}_{m-1}, \, 3, \, 1, \underbrace{5,\, 1}_{m-1}]  \,=:\,
[ 6; \{1,5\}^{\times(m-1)},3, \,1,\{5,1\}^{\times(m-1)}],
$$
and the (renormalized) weight expansion~$W(a_n)$ is $(A_n, B_n)$, 
where the vectors $A_n,B_n$ are 
\begin{eqnarray*}
A_n &=& \bigl({u_0}^{\times6}, {v_1}^{\times1}, {u_1}^{\times5}, {v_2}^{\times1}, \dots, {u_{m-2}}^{\times5}, {v_{m-1}}^{\times1}, {u_{m-1}}^{\times5} \bigr) , \\
B_n &=& \bigl({F_m}^{\times3}, {L_{m-1}}^{\times1}, {F_{m-1}}^{\times5}, {L_{m-2}}^{\times1}, \dots,{F_2}^{\times5}, {L_1}^{\times1}, {F_1}^{\times5},
{L_0}^{\times1} \bigr).
\end{eqnarray*}
\end{proposition}

\begin{rmk}\rm 
Although our usual convention is that the expression 
$a=[\ell_0;\ell_1,\dots,\ell_N]$ always has $\ell_N\ge 2$, 
we relax this condition here in order to simplify the formulas.
We allow ourselves another liberty at the end of these expansions:
in the formula for $B_n$ in Proposition~\ref{prop:a2m}
the last two weights $F_1,L_0$ are equal, so the ending multiplicity is in fact $6$ rather than~$5,1$.
\diam
\end{rmk}

\begin{proof}  Recall that 
$a_n = \left(\frac{g_{n+1}}{g_n}\right)\,\!^2 = \frac{f_{2n+1}^2}{f_{2n-1}^2}$. 
Also, because  $b_{n-1}< a_n < b_n$, the weight expansion 
$\ww(a_n)$ begins as $(1^{\times6}, (a_n-6)^{\times 1}, \dots)$.  
Hence 
$$
W(a_n) = f_{2n-1}^2 \ww(a_n) = \bigl({f_{2n-1}^2}^{\times6}, f_{2n+1}^2 - 6\,f_{2n-1}^2,\dots\bigr).
$$
Therefore, because $v_m=F_m$ as noted above, the expansion of  $W(a_n)$ up to
and including the term $L_{m-1}$ follows from Lemma~\ref{le:expanan} below. 
The rest of the expansion holds by Lemma~\ref{le:FLH}.
\end{proof}

\begin{lemma} \labell{le:expanan}   
Let $n = 2m \ge 4$. 
The following identities hold.

\begin{itemize}
\item[(i)]
$f_{2n+1}^2 = 6\,u_0+v_1$ and $v_1 < u_0$;
\item[(ii)]
$u_k = v_{k+1} + u_{k+1}$ and $u_{k+1} < v_{k+1}$ for $k = 0, \dots, m-2$;
\item[(iii)]
$v_k = 5\,u_k + v_{k+1}$ and $v_{k+1} < u_k$ for $k = 1, \dots, m-1$;
\item[(iv)]
$u_{m-1} = 3\, F_m + L_{m-1}$ and $L_{m-1} < F_m$.
\end{itemize}
\end{lemma}

\begin{proof}
Statement (i) is equivalent to
\begin{equation}  \label{e:=62}
f_{2n+1}^2  = 6\,f_{2n-1}^2  + 3F_{n-1} - 2.
\end{equation}
Since $f_{s+1}^2=6f_{s-1}^2+f_{2s-4}-2(-1)^s$ holds true for $s=2,3,4$, this
equation holds true for all $s \ge 2$ by Proposition~\ref{prop:proof};
in particular it holds true for all even $s \ge 2$, and so~\eqref{e:=62} holds true.

The equality in (ii) follows from
the definitions of $u_k,v_k$ by using~\eqref{eq:uk2} and $H_k = \sum_{i=1}^k F_i$ and by dividing the equations into two equations, one for $k$ and one for $k':=n-k$.
To prove the equation in (iii) we again divide it into two equations, one for $k$ and one for $k':=n-k$, namely
$$
-2 F_k = 10 H_k +2 -2 F_{k+1}, \qquad
3 F_{k'} = 5f_{2k'-1}^2 -2 +  3F_{k'-1} .
$$
The first equation is equivalent to $L_k=5H_k+1$, which holds by Lemma~\eqref{le:FLH}~(iv),
and the second holds because $f_{2s} = 5f_{s-1}^2 -2(-1)^s + f_{2s-4}$ is true for
$s=2,3,4$ and hence for all~$s\ge 2$ by Proposition~\ref{prop:proof}.

Equation~\eqref{eq:uk2} and Lemma~\ref{le:FLH} imply that
$$
u_{m-1} = 3H_m + 2H_{m-1} + 1 = 3F_m + 5H_{m-1} + 1= 3F_m + L_{m-1}.
$$
This proves (iv).

Since $u_k, v_k$ are positive in the given ranges,
the equalities in~(ii)  and (iv) imply the inequalities in~(i) and~(iii).
Similarly, the equalities in~(iii) imply the inequalities in~(ii).  
This completes the proof.
\end{proof}

The reduction process has three steps that are described in Propositions~\ref{prop:st1}, \ref{prop:st2} and~\ref{prop:st3}.  Notice, before we begin, that the weights of $a_n$ divide into three groups, namely
$(m-1)$ pairs $\{1,5\}$, two central  blocks with multiplicities $3, 1$, and finally $(m-1)$ pairs $\{5,1\}$.  
In the first step we show that a set of $5$~Cremona moves has the effect of moving the first $\{1,5\}$ pair from the left to a $\{5,1\}$ pair on the right. 
Moreover, doing this introduces no new weights to the right 
while slightly modifying the first block on the left.

Denote $V_n = E(a_n) = \left(g_n g_{n+1}; W(a_n), 1 \right)$.
Note that 
$$
g_n g_{n+1} = f_{2n-1}f_{2n+1} = f_{2n}^2+1 = f_{2n}^2+f_1^2
$$
by~\eqref{eq:cassini}.

\begin{prop}\labell{prop:st1}
For $n =2m\ge 6$, 
the vector $V_n$ is reduced by $5(m-1)$ Cremona moves to the vector
$V_n^1 = \left( f_{2(m+1)}^2 + f_{2m-1}^2;\, A_n^1, B_n^1 \right)$,
where 
\begin{eqnarray*}
A_n^1: &=&
\bigl(\, (u_{m-1}+F_{m-1})^{\times 1}, \, {u_{m-1}}^{\times 5} \bigr),
\\
B_n^1: &=&
\bigl(\,
{F_m}^{\times 3}, \,{L_{m-1}}^{\times 1}, 
\,{F_{m-1}}^{\times 10}, \, {L_{m-2}}^{\times 2},
\dots, 
\,{F_2}^{\times 10}, \,{L_1}^{\times 2}, 
\,{F_1}^{\times 10}, \, {L_0}^{\times 2}, \,1^{\times 1} 
\,\bigr).
\end{eqnarray*}
\end{prop}

Let $$
V(n,1) := \left( f_{2n}^2+1; A(n,1), B(n,1) \right) := 
\left( f_{2n}^2+1; A_n, B_n \right) = V_n,
$$
and for $k = 2, \dots, m$ define the vector $V(n,k)$ by 
$\left( f_{2(n-k+1)}^2+f_{2k-1}^2; A(n,k), B(n,k) \right)$,
where 
\begin{eqnarray*}
A(n,k) &=& \bigl(\, (u_{k-1} + F_{k-1})^{\times 1},
{u_{k-1}}^{\times 5}, {v_k}^{\times 1},  {u_{k}}^{\times 5}, \dots,,{v_{m-1}}^{\times 1}, {u_{m-1}}^{\times 5}\bigr) , \\
B(n,k) &=& \bigl(\, {F_m}^{\times 3}, {L_{m-1}}^{\times 1}, \dots, 
{F_{k}}^{\times 5}, {L_{k-1}}^{\times 1}, 
{F_{k-1}}^{\times 10}, {L_{k-2}}^{\times 2}, 
\dots,\\
&&\hspace{2.4in}
{F_1}^{\times 10}, {L_0}^{\times 2}, 1^{\times 1}\, \bigr) .
\end{eqnarray*}
Then $V(n,1) = V_n$ and $V(n,m)=V_n^1$.
Moreover, $A(n,k+1)$ is obtained from $A(n,k)$ by replacing its first seven entries by
the single entry $u_k+F_k = f_{2(n-k)-1}^2 +2 \,H_k + F_k$;
and $B(n,k+1)$ is obtained from $B(n,k)$ by inserting 
$\bigl(\, {F_{k}}^{\times5}, L_{k-1} \,\bigr)$.
Proposition~\ref{prop:st1} will follow if we prove:

\begin{lemma}  \labell{le:nk} For $1 \le k \le m-1$,
$V(n,k)$ reduces to $V(n,k+1)$
by 5~Cremona moves.
\end{lemma}

We prove Lemma~\ref{le:nk} in five steps.
Throughout, we abbreviate $n-k$ to $k'$.  Thus $k'>k$.

\begin{rmk}\labell{rmk:proof} 
\rm (i)
Each step of the reduction involves many small calculations 
that can be done directly using identities such as~\eqref{eq:cassini}, 
\eqref{eq:doubling.odd} and \eqref{eq:doubling.even}.  
However, in all cases the required identity is quadratic in the sense of Proposition~\ref{prop:proof}.
Hence they can be all proved by verifying them for just 
three low values of~$s$,
as we have already illustrated in the proof of Lemma~\ref{le:expanan}.
The only condition on the choice of~$s$ is that all subscripts
of the~$f_i$ should be $\ge 0$.  

\s
(ii)  
In each step of the reduction process there is no interaction 
between the terms in~$k$ and those in~$k'$; 
we simplify each set of terms separately.
\diam
\end{rmk}

\NI {\bf Step 1:}  
{\it There is a Cremona move that takes $V(n,k)$ to
\begin{eqnarray*}  
V_1(n,k) &:=& \bigl(\, f_{4k'+2} - F_{k-1}; \,
       {u_{k-1}}^{\times 3},\, v_k,\, 
       f_{2k'+1}f_{2k'} - H_{k-1},\\ \notag
       &&\hspace{1in}
       (f_{2k'+1}f_{2k'} - H_{k-1} - F_{k-1})^{\times 2},\,
       {u_k}^{\times 5}, \dots
       \bigr).
\end{eqnarray*}

\proof 
The first component of the Cremona transform of $V(n,k)$ is
$2 \left( f_{2(k'+1)}^2+f_{2k-1}^2 \right) - 3\, u_{k-1}$, and we must show that it equals $f_{4k'+2}$.  
Since $u_{k-1} := f_{2k'+1}^2 + 2H_{k-1}$,
we need to see that 
$$
2 \, f_{2(k'+1)}^2 - 3\, f_{2k'+1}^2- f_{4k'+2}  \,=\,   6\,H_{k-1}  - 2\, f_{2k-1}^2
\,=\, 2 \left(f_{2k-2}f_{2k}-f_{2k-1}^2 \right).
$$
But both sides are equal to $-2$. This is clear for the RHS by~\eqref{eq:cassini}.
For the LHS, note that $2 f_{s+2}^2 - 3 f_{s+1}^2- f_{2s+2} = -2(-1)^s$,
since this is true for $s=-1,0,1$ and hence for all $s \ge -1$ by 
Proposition~\ref{prop:proof}. 

The second three terms of the Cremona transform of $V(n,k)$ equal
$\left(f_{2(k'+1)}^2+f_{2k-1}^2 \right) - 2\, u_{k-1}$, and one can check as above that this is $f_{2k'+1}f_{2k'} - H_{k-1}$.
Therefore the given vector $V_1(n,k)$ is a reordering of the 
Cremona transform of $V(n,k)$.
\proofend

\NI {\bf Step 2:} 
{\it There is a Cremona move that takes $V_1(n,k)$ to
\begin{eqnarray*} 
 V_2(n,k)  &:=& \Bigl( f_{4k'} + f_{2k'+1}f_{2k'} -H_k-F_{k-1}; \,
       v_k,\, f_{2k'+1}f_{2k'} - H_{k-1},    \,\Bigr. \\
    &&   \hspace{.5in} \Bigl. 
     (f_{2k'+1}f_{2k'} - H_{k-1} - F_{k-1})^{\times 2}, (f_{2k'}f_{2k'-1} - f_{2k}f_{2k-1})^{\times 3}, \,
       {u_k}^{\times 5}, \dots
       \Bigr) . \notag
\end{eqnarray*}

\proof
For the first component we need to see that
$$
2 \left( f_{4k'+2}-F_{k-1} \right) - 3 \, u_{k-1}
\,=\, 
f_{4k'} + f_{2k'+1}f_{2k'} -H_k-F_{k-1}.
$$
But this is equivalent to the identity
\begin{equation}  \label{e:23}
2\, f_{4k'+2} - 3\, f_{2k'+1}^2 - f_{4k'} - f_{2k'+1}f_{2k'}
\,=\,
6\, H_{k-1} + F_{k-1} - H_k,
\end{equation}
and one can check that both sides here equal~$-1$.

The Cremona transform also contains three terms of the form
$f_{4k'+2} - F_{k-1} - 2\,u_{k-1}$, and we need to check that this is
$f_{2k'}f_{2k'-1} - f_{2k}f_{2k-1}$. But this holds because
$$
f_{4k'+2} - 2\, f_{2k'+1}^2 - f_{2k'}f_{2k'-1} 
\,=\, 
\tfrac 43 f_{2k-2}f_{2k} - f_{2k}f_{2k-1} + \tfrac 13 f_{4k-4} \,=\, -1.
$$
Thus $V_2(N,k)$ is a reordering of the Cremona transform of $V_1(n,k)$.
\proofend

\NI {\bf Step 3:}  
{\it There is a Cremona move that takes $V_2(n,k)$ to
\begin{equation*} 
V_3(n,k) \,:=\, \bigl(\, f_{4k'} - F_{k-1};\,
       f_{2k'+1}f_{2k'} - H_{k-1} - F_{k-1}, \,
       (f_{2k'}f_{2k'-1} - f_{2k}f_{2k-1})^{\times 4}, \,
       {u_k}^{\times 5}, \dots \bigr) 
\end{equation*} 
where the multiplicities of~$F_{k}$ and $L_{k-1}$ are each increased by one to
$ {F_k}^{\times 6}, \, {L_{k-1}}^{\times 2}$.}

\proof
Since $H_k = H_{k-1}+F_k$, the first term
$$
2 \left( f_{4k'} + f_{2k'+1}f_{2k'} -H_k-F_{k-1} \right) - v_k 
- 2 \left( f_{2k'+1}f_{2k'} - H_{k-1} \right) + F_{k-1} 
$$
of the Cremona transform of $V_2(n,k)$ is equal to
$$
2\,f_{4k'} - 2\, F_k - F_{k-1} - (f_{4k'}-2\,F_k) 
\,=\,
f_{4k'}-F_{k-1} .
$$
Its second term
$$
\left( f_{4k'} + f_{2k'+1}f_{2k'} -H_k-F_{k-1}\right) - 
2 \left( f_{2k'+1}f_{2k'} - H_{k-1} \right) + F_{k-1}
$$ 
simplifies to
$$
f_{4k'} - f_{2k'+1}f_{2k'} -F_k + H_{k-1} 
\,=\,
f_{2k'}f_{2k'-1} - f_{2k}f_{2k-1}, 
$$
where the last equality follows from the identities
\begin{equation} \label{e:red4}
f_{4k'} - f_{2k'+1}f_{2k'} - f_{2k'}f_{2k'-1} =0, \quad
F_k - H_{k-1} - f_{2k}f_{2k-1} =0.
\end{equation}

The third term of the Cremona transform of $V_2(n,k)$ is
$$
\left( f_{4k'} + f_{2k'+1}f_{2k'} -H_k-F_{k-1} \right) 
- v_k 
- \left( f_{2k'+1}f_{2k'} - H_{k-1} - F_{k-1}  \right) .
$$
This is equal to $F_k$.
In the same way, we find that its fourth term  is 
$F_k-F_{k-1} = L_{k-1}$.
The result follows immediately.
\proofend

\NI {\bf Step 4:}  
{\it There is a Cremona move that takes $V_3(n,k)$ to
\begin{eqnarray*} 
V_4(n,k) &:=&\bigl(\,
f_{2k'+1}f_{2k'} - H_{k-1} - F_{k-1} + 2\, F_k; \,
       f_{2k'}^2 + f_{2k-1}^2 + F_k, \\  \notag&&\hspace{1.5in}
       (f_{2k'}f_{2k'-1} - f_{2k}f_{2k-1})^{\times 2}, \,
       {u_k}^{\times 5}, \dots
\bigr) 
\end{eqnarray*}
where the multiplicities of $F_{k}$ and $L_{k-1}$ are now
$ {F_k}^{\times 8}, \, {L_{k-1}}^{\times 2}$.}

\proof
By~\eqref{e:red4},
the first term of the Cremona transform of $V_3(n,k)$ is
$$
f_{2k'+1}f_{2k'} - H_{k-1} - F_{k-1} + 2\, F_k.
$$
We claim that its second term is
$$
f_{4k'} - F_{k-1} -2 \left( f_{2k'}f_{2k'-1} - f_{2k}f_{2k-1}\right)
\,=\,
f_{2k'}^2 + f_{2k-1}^2 + F_k .
$$
This follows from
\begin{equation} \label{e:red5}
f_{4k'} -2 \,f_{2k'}f_{2k'-1} - f_{2k'}^2 = 0, 
\quad
F_{k-1} - 2 \,f_{2k}f_{2k-1} + f_{2k-1}^2 + F_k = 0 .
\end{equation}

Finally, the third and fourth term of the Cremona transform of
$V_3(n,k)$ are
$$
f_{4k'} -F_{k-1} - \left( f_{2k'+1}f_{2k'} - H_{k-1} - F_{k-1} \right) - 
\left( f_{2k'}f_{2k'-1} - f_{2k}f_{2k-1}  \right) \,=\, F_k
$$
where the equality holds by \eqref{e:red4}. 
Hence $V_4(n,k)$ is a reordering of the Cremona transform, as claimed.
\proofend

\NI {\bf Step 5:}  
{\it There is a Cremona move that takes $V_4(n,k)$ to
\begin{equation} \label{v:red6o}
V(n,k+1) \,:=\, \bigl(\, f_{2k'}^2 + f_{2k+1}^2;\,
       u_k + F_k, \,
       {F_k}^{\times 2}, \,
       {u_k}^{\times 5}, \dots
\bigr) 
\end{equation}
where  the multiplicities of $F_{k}$ and $L_{k-1}$ are
${F_k}^{\times 10}, \, {L_{k-1}}^{\times 2}$.
}

\proof  
The above expression for the first term of the Cremona transform of $V_4(n,k)$
follows from the identities
\begin{gather*}
2\, f_{2k'+1}f_{2k'} - 2\, f_{2k'}^2 - 2\, f_{2k'}f_{2k'-1} = 0,\\
2\,H_{k-1} + 2\, F_{k-1} - 3\, F_k + f_{2k-1}^2 - 2\, f_{2k}f_{2k-1} + f_{2k+1}^2 = 0. 
\end{gather*}
(The second identity can be simplified by subtracting the second identity of~\eqref{e:red5}.)

We next claim that the second term of this transform  is
$u_k + F_k$.  Since $u_k = f_{2k'-1}^2 +2\,H_k$ and $H_{k-1} + F_k = H_k$,
this is equivalent to the identity 
$$
f_{2k'+1}f_{2k'} - 2\, f_{2k'}f_{2k'-1} - f_{2k'-1}
\,=\,
3\,H_{k-1} + F_{k-1} + F_k - 2\, f_{2k}f_{2k-1} .
$$
But both sides equal $-1$.

Finally, its third and fourth terms are $F_k$ because
$$
f_{2k'+1}f_{2k'} - f_{2k'}^2 - f_{2k'}f_{2k'-1}
\,=\,
H_{k-1} + F_{k-1} + f_{2k-1}^2 - f_{2k}f_{2k-1} .
$$
Here both sides vanish.
\proofend

This completes the proof of Lemma~\ref{le:nk} and hence of Proposition~\ref{prop:st1}.
The next stage of the reduction process
results in a vector $V_n^2$ whose components are linear (rather than quadratic) 
functions of the~$f_k$ 
and do not involve the index~$k'$.

\begin{prop}\labell{prop:st2} 
When $n=2m\ge 6$, the vector $V_n^1$ may be reduced by four Cremona moves to 
$$
V_n^2 \,=\, \bigl(\, {F_m}^{\times 1}; {L_{m-1}}^{\times 1}, {F_{m-1}}^{\times 11}, {L_{m-2}}^{\times 2}, \dots, {F_2}^{\times 10}, {L_1}^{\times 2}, {F_1}^{\times 13} \bigr).
$$
\end{prop}

\proof
Note that $V_n^2$ is obtained from $B_n^1$ by removing two copies of~$F_m$ and adding one~$F_{m-1}$.

We first claim that the Cremona transform of $V_n^1$ is
$$
\bigl(\,
f_{4m+2}-F_{m-1};\, f_{4m-1}+F_{m-1},\, {f_{4m-1}}^{\times 2},\, {u_{m-1}}^{\times 3},\, B_n^1 \,\bigr),
$$
which we reorder as
\begin{equation} \label{v:2.1}
\bigl(\, 
f_{4m+2}-F_{m-1};\, {u_{m-1}}^{\times 3}, \, f_{4m-1}+F_{m-1},\, {f_{4m-1}}^{\times 2},\, B_n^1 \,\bigr) .
\end{equation}
Here one obtains the first term of the transform from the identity
$$
2 \left( f_{2(m+1)}^2+f_{2m-1}^2\right) -3\, u_{m-1} \,=\, f_{4m+2},
$$
and the second term from
$$
f_{2m+2}^2 + f_{2m-1}^2 -2\,u_{m-1}\,=\, f_{4m-1} + F_{m-1}.
$$

Next, observe that Lemma~\ref{le:expanan}~(iv) implies that
$2\,u_{m-1} = 6 F_m+ 2L_{m-1}$,
which, as one can easily check, is just $f_{4m+2}-F_{m-1}$.
%	$$
%	f_{4m+2}-F_{m-1} = 2\,u_{m-1}.
%	$$
Therefore, the second Cremona transform moves the vector~\eqref{v:2.1} to
\begin{equation}  \label{v:000}
\bigl(\,
u_{m-1};\, f_{4m-1}+F_{m-1},\, {f_{4m-1}}^{\times 2},\, B_n^1 
\,\bigr) .
\end{equation}

We next claim that
$u_{m-1} - 2\,f_{4m-1} \,=\, F_{m-1}$.
Hence the third Cremona transform moves the vector~\eqref{v:000} to
$$
\bigl(\, f_{4m-1}+F_{m-1};\, F_{m-1}, \,{0}^{\times 2},\, B_n^1 \,\bigr),
$$ 
which we reorder as
\begin{equation*}
\bigl(\, 2F_m;
{F_m}^{\times 3}, \, {L_{m-1}}^{\times 1},
\,{F_{m-1}}^{\times 11}, \,{L_{m-2}}^{\times 2},
\,{F_{m-2}}^{\times 10}, \,{L_{m-2}}^{\times 2},
\dots, \,
\,{F_2}^{\times 10}, \,{L_1}^{\times 2}, 
\,{F_1}^{\times 13} 
\,\bigr) .
\end{equation*}
Therefore, another Cremona move takes the above vector to
$$
V_n^2 \,=\, \bigl(\,F_m;\, {L_{m-1}}^{\times 1}, 
\,{F_{m-1}}^{\times 11}, \,{L_{m-2}}^{\times 2},
\,{F_{m-2}}^{\times 10}, \,{L_{m-2}}^{\times 2},
\dots, \,
\,{F_2}^{\times 10}, \,{L_1}^{\times 2}, 
\,{F_1}^{\times 13} \bigr).
$$
This completes the proof of Proposition~\ref{prop:st2}.
\proofend

\begin{prop}\labell{prop:st3}  
For $n=2m \ge 6$, 
the vector $V_n^2$ may be reduced to~$(2; 1^{\times 5})$ by Cremona moves.
\end{prop}

\proof  
One shows by direct calculation that this holds when $n=6$.  
Therefore, by induction it suffices to show that the vector~$V_n^2$ is reduced to 
$V_{n-2}^2$ by 6~Cremona moves. 
By using the identities in Lemma~\ref{le:FLH}, it is not hard to prove directly that 
this reduction may be achieved by six standard Cremona moves.
Alternatively,
one can check numerically that this holds when $n=8$ and $10$, 
checking also that the reordering required is the same in both cases at each stage. 
Then it holds for all~$n$ by Proposition~\ref{prop:proof}. 
(Note that we only need to check two values since all identities are homogeneous and linear.)
\endproof

%%%%%%%%%%%%%%%%%%%%%%%%%%%%%%%%%%%%%%%%%%%%%%%%%%%%%%%%%%%
\subsubsection{Reducing $E(a_n)$ for odd $n$.}\labell{ss:fib.odd}
%%%%%%%%%%%%%%%%%%%%%%%%%%%%%%%%%%%%%%%%%%%%%%%%%%%%%%%%%%%

Throughout this section we denote $n = 2m+1$, where $m\ge 1$.
By Example~\ref{ex:an} it suffices to consider the case $n\ge 5$. Hence  this case of Theorem~\ref{thm:ladder} 
follows from Propositions~\ref{prop:st1o}, \ref{prop:st2o} and \ref{prop:st3o}.
 
We consider the numbers 
$$
u_k = f_{2n-2k-1}^2 + 2\, H_k,\;\; k = 0, \dots, m, \quad
v_k = f_{4(n-k)} - 2\,F_k,\;\;k = 1, \dots, m,
$$
as before. Again we have $v_m>0>v_{m+1}$ but now $u_m=L_m$.
Therefore
Lemma~\ref{le:expanan} takes the following form.
 
\begin{lemma}\labell{le:expanano}   
If $n=2m+1$, the following identities hold.

\begin{itemize}
\item[(i)]
$f_{2n+1}^2 = 6\,u_0+v_1$ and $v_1 < u_0$;
\item[(ii)]
$u_k = v_{k+1} + u_{k+1}$ and $u_{k+1} < v_{k+1}$ for $k = 0, \dots, m-1$;
\item[(iii)]
$v_k = 5\,u_k + v_{k+1}$ and $v_{k+1} < u_k$ for $k = 1, \dots, m-1$;
\item[(iv)]
$v_{m} = 3\,u_m+ F_m = 3\, L_m + F_{m}$ and $F_m<L_m$.
\end{itemize}
\end{lemma}
 
\begin{proof}    
The proofs of the equalities in~(i), (ii) and (iii) go through as before, 
since these are based on equalities that do not mention~$m$. 
(Note that the proof of the equality in~(ii) works when $k=m-1$, 
though the corresponding inequality failed for even~$n$.)   
One then checks~(iv). Then the inequality in~(ii) follows from the equalities in~(iii) and (iv), 
while the inequality in~(iii) holds by~(ii).
\end{proof}
 
As before, this lemma immediately gives the following result. 

\begin{proposition}\labell{prop:anodd}
If $n=2m+1$ is odd, the continued fraction expansion of~$a_n = \frac{f_{2n+1}^2}{f_{2n-1}^2}$ is 
$$
[ 6;\, \{1,\,5\}^{\times (m-1)},\, 1,
\, 3, \,
\{5, \, 1\}^{\times m} ],
$$
and the (renormalized) weight expansion~$W(a_n)$ is $(\HA_n, \HB_n)$, 
where  
\begin{eqnarray*}
\HA_n &:=& \bigl({u_0}^{\times6},\,{v_1}^{\times1}, {u_1}^{\times5},  \dots,{v_{m-1}}^{\times1},\, {u_{m-1}}^{\times5}, {v_{m}}^{\times1}  \bigr) , \\
\HB_n &:=& \bigl((u_m=L_m)^{\times3},  {F_{m}}^{\times5}, {L_{m-1}}^{\times1}, \dots, {F_1}^{\times5},
{L_0}^{\times1} \bigr).
\end{eqnarray*}
\end{proposition}

The proof of Proposition~\ref{prop:st1} also goes through. 
In other words, each set of five Cremona moves takes
one of the $m-1$ pairs $\{1,5\}$ from the left to the right of the central blocks 
(which now have multiplicities $1,3$), while introducing no new weights on the right.
Since we start with~$m$ blocks on the right but only $(m-1)$ on the left, this means that one pair $\{5,1\}$ still remains on the right, though all the others become $\{10,2\}$.
The only other difference is in the interpretation of the first term of $V(n,m)$: 
when $n=2m+1$ and $k=m$ we have $2(n-k+1) = 2(m+2)$.
Thus we obtain:

\begin{prop}\labell{prop:st1o}
For $n =2m+1\ge 5$, 
the vector $\HV_n=\left(g_n g_{n+1}; W(a_n), 1 \right)$ is reduced by $5(m-1)$ Cremona moves to the vector
${\HV}\,\!_n^1 = \left( f_{2m+4}^2 + f_{2m-1}^2;\, \Hat A_n^1, \Hat B_n^1 \right)$,
where 
\begin{eqnarray*}
\Hat A_n^1 &:=&
\bigl(\, (u_{m-1}+F_{m-1})^{\times 1}, \, {u_{m-1}}^{\times 5},\,
{v_{m}}^{\times 1} \bigr),
\\
\Hat B_n^1 &:=&
\bigl(\,
{L_m}^{\times 3}, \,{F_{m}}^{\times 5}, 
\,{L_{m-1}}^{\times 1}, \, {F_{m-1}}^{\times 10},
\dots, 
\,{F_2}^{\times 10}, \,{L_1}^{\times 2}, 
\,{F_1}^{\times 10}, \, {L_0}^{\times 2}, \,1^{\times 1} 
\,\bigr).
\end{eqnarray*}
\end{prop}

Notice that the entries in $\HV_n^1$ still depend explicitly both on~$k$ and on $k'=n-k$
since $u_{m-1}$ and $v_{m-1}$ have this structure.
This means that there are entries in $\HV_n^1$ that do not occur anywhere in the reduction of $\HV_{n+2}$.
The next stage takes us to a vector that occurs in the reduction of all $\HV_{n+2i}$.  
For even~$n$, this stage consisted of four moves, but now it takes six moves.

\begin{prop}\labell{prop:st2o} 
When $n=2m+1\ge 5$, the vector $\HV_n^1$ is reduced by six Cremona moves to 
\begin{eqnarray*}
\HV_n^2 &:=& \bigl(\,L_m-F_m; 
({L_m-2F_m})^{\times 1}, {F_m}^{\times 7},
{L_{m-1}}^{\times 2}, {F_{m-1}}^{\times 10}, {L_{m-2}}^{\times 2}, \dots \bigr),
\end{eqnarray*}
where the terms including and after ${F_{m-1}}^{\times 10}$ in $\HV_n^2$ are the same 
as those in $\HB_n^1$.
\end{prop}

\begin{proof}
Since the proof is much the same as that of Proposition~\ref{prop:st2}, 
we simply list the results of each Cremona move.  
Here are the results of the first four moves:
\begin{eqnarray*}
 &&\bigl(\,f_{4m+6} -F_{m-1} ; {u_{m-1}}^{\times 3}, v_m, f_{4m+3} + f_{4m-1}+F_{m-1}, (f_{4m+3} + f_{4m-1})^{\times 2}, {L_m}^{\times 3},\dots \bigr)\\
  &&\bigl(\,f_{4m+5} +L_{m-1} ;   v_m, f_{4m+3} + f_{4m-1}+F_{m-1},
   (f_{4m+3} + f_{4m-1})^{\times 2},{L_m}^{\times 3}, \\
  &&\hspace{4in}{f_{4m+1}}^{\times 3}, \dots \bigr)\\
  &&\bigl(\,f_{4m+4} -F_{m-1} ;   {f_{4m+3} + f_{4m-1}}, {L_m}^{\times 3}, {f_{4m+1}}^{\times 4}, F_m, L_{m-1}\dots \bigr)\\
  &&\bigl(\,f_{4m+3} + f_{4m-1};
{f_{4m+3} + f_{4m-1} - L_m},
{L_{m}}, {f_{4m+1}}^{\times 4}, F_m, L_{m-1}, \dots\bigr).
\end{eqnarray*} 
None of these moves uses up any of the terms 
of~$\HB_n^1$ after ${L_m}^{\times 3}$, 
though the multiplicity of $F_m, L_{m-1}$ is increased.
Hence after the entry for~$L_m$ we have simply listed the extra terms that get added to~$\HB_n^1$.  
Note also there is just one new number that appears after $L_m$, namely $f_{4m+1}=L_m-F_m$ 
which appears with multiplicity~$4$.   
Using the same conventions, the next move gives:
$$
\bigl(\,2 f_{4m+1} ;   {f_{4m+1}}^{\times 3}, f_{4m+1}- F_m,
   {F_m}^{\times 2}, L_{m-1}, \dots \bigr).
$$
The last move changes the first four terms to the single term  $f_{4m+1} = L_m-F_m$.
\end{proof}

\begin{prop}\labell{prop:st3o}  
For $n=2m+1 \ge 5$, 
the vector $\HV_n^2$ may be reduced to~$(2; 1^{\times 5})$ by Cremona moves.
\end{prop}

\begin{proof} 
This is just the same as the proof of Proposition~\ref{prop:st3}.
%One can check directly that the result holds when $n=5$.
%Hence it suffices to prove that for all odd $n \ge 7$ six Cremona moves 
%reduce $\HV_n^2$ to $\HV_{n-2}^2$.  
%This holds for $n=7,9$ 
%by direct calculation, and hence holds for all~$n$.
%(Note that by Proposition~\ref{prop:proof} we only need to check two values because the relevant identities are linear with no constant terms.)
\end{proof}

%%%%%%%%%%%%%%%%%%%%%%%%%%%%%%%%%%%%%%%%%%%%%%%%%%%%%%%%%%%%
\section{The interval $[\tau^4,7]$} \label{s:tau7}
%%%%%%%%%%%%%%%%%%%%%%%%%%%%%%%%%%%%%%%%%%%%%%%%%%%%%%%%%%%%

This section is devoted to the 
calculation of $c(a) $ on the interval $[\tau^4,7]$, 
thus completing the proof of part (ii) of Theorem~\ref{thm:main}.
Proposition~\ref{prop:7easy} gives an easy argument that 
$c(a) = \frac{a+1}3$ when 
$a \in [6\frac{11}{12}, 7\frac 19]$.   
To prove that this holds on the whole interval $[\tau^4,7]$, 
we shall adapt the arithmetic approach that works for $a<\tau^4$ 
rather than using more analytical arguments as in the case $a>7$.     
We show in Proposition~\ref{p:fewpoints} that it suffices to restrict attention to 
some special points with relatively short continued fraction expansions that are related to 
the convergents of $\tau^4$, and then deal with these special points by largely arithmetic means.
The proof that $c(a) = \frac{a+1}3$ on $[\tau^4,7]$ is given at the end 
of \S\ref{ss:specpt}.

%%%%%%%%%%%%%%%%%%%%%%%%%%%%%%%
\subsection{Reduction to special points} \labell{ss:specpt}
%%%%%%%%%%%%%%%%%%%%%%%%%%%%%%%

As in Section~\ref{ss:basob}, 
given $a$ with weight expansion $\ww(a)$ 
and  $(d;\mm)\in \Ee$, we define $\eps$ by 
$\mm = \tfrac d{\sqrt a}\ww(a) + \eps$, and
denote $E := \sum \eps_i^2$. Also set $\la^2 := 1-E$.
Denote
\begin{equation*} 
y(a) := a+1-3\sqrt a .
\end{equation*}
Since $y(\tau^4)=0$ we have $y(a)>0 $ for $a>\tau^4$.
The first result extends Proposition~\ref{prop:obs}.

\begin{prop} \labell{prop:obs1} 
Let $a\in (\tau^4,7)$ and suppose that $(d;\mm) \in \Ee$ is such that
$\mu(d;\mm)(a) > \frac{a+1}3$. 
Then 
\begin{itemize}
\item[(i)]
$\displaystyle d < \frac{3 \sqrt a}{\sqrt{a^2-7a+1}}$;

\b
\item[(ii)]
$\displaystyle \gl^2 > \frac{2d^2y(a)}{3\sqrt a}$.
\end{itemize}
\end{prop}

\begin{proof} 
By Proposition~\ref{prop:obs}~(i) we have
$\frac{a+1}3 < \sqrt{a} \sqrt{1+1/d^2}$, proving~(i).

Since $\frac {a+1}3 <\mm\cdot\ww/d = \sqrt a + \eps\cdot\ww/d$, we have
$y(a)\:\! d/3 < \eps \cdot \ww$.
Further, 
$$
d^2+1 \,=\, \mm \cdot \mm \,=\, 
\tfrac{d^2}{a} a + \tfrac{2d}{\sqrt a} \ww(a) \cdot \eps + \eps \cdot \eps ,
$$
i.e., $\eps \cdot \ww (a) = \frac{\la^2\sqrt a}{2d}$, proving~(ii).
\end{proof}

The continued fraction expansion of $\tau^4=\frac{7+3\sqrt 5}2$ is $[6;1,5,1,5,\dots]$. 
For $k \ge 1$ define its $k$th convergent $c_k$ by
\begin{eqnarray*}
c_{2k-1} &:=& \bigl[ \,6;\{1,5\}^{\times (k-1)},1 \,\bigr] 
         ,=\, \bigl[ \,6;\{1,5\}^{\times (k-2)},1,6 \,\bigr] , \\
c_{2k}   &:=& \bigl[ \,6;\{1,5\}^{\times k} \,\bigr].
\end{eqnarray*}
Thus
\begin{gather} \labell{eq:ck}
c_1 = 
7,\;\; 
c_2 = 
6 \tfrac 56 = \tfrac{41}{6}, \;\;
c_3 = [6;1,5,1] = [6;1,6] = 6 \tfrac 67 = \tfrac {48}7 , \\ \notag
c_4 = [6;1,5,1,5] = 6 \tfrac{35}{41} = \tfrac{281}{41}, \;\; 
c_5 = [6;1,5,1,6] = 6 \tfrac{41}{48} = \tfrac {329}{48} ,
\end{gather}  
and more generally
$$
c_{2k} < c_{2k+2} < \tau^4 < c_{2k+1}< c_{2k-1}
\quad \mbox{ for all }\, k \ge 1 .
$$
Moreover, 
for $k \ge 1$ and $j \ge 1$ define the numbers $u_k(j), v_k(j) \in (c_{2k+1},c_{2k-1})$ by
\begin{eqnarray*}
u_k(j) &:=& \bigl[ \,6;\{1,5\}^{\times (k-1)},1,6,j \,\bigr], \\
v_k(j) &:=& \bigl[ \,6;\{1,5\}^{\times (k-1)},1,j \,\bigr] .
\end{eqnarray*}
As in Definition~\ref{def:Lucas}, 
we define for $k \ge 1$ the $k$th Lucas number 
$\ell_k = f_{k-1}+f_{k+1}$,
and set $F_k = \frac 13 f_{4k}$ and $L_k = \frac 13 \ell_{4k+2}$. 
Recall from Lemma~\ref{le:FLH} that for all $k\ge 0$,
\begin{equation} \label{FL:repeat}
F_{k+1} = L_k + F_k \quad \mbox{ and } \quad 
L_{k+1} = 5\,F_{k+1}+L_k 
\end{equation}
as well as
\begin{equation} \label{e:Fkkk}
F_{k+1}^2-F_{k}F_{k+2} \,=\,1 .
\end{equation}

\begin{lemma} \label{le:Fpq}
For all $k\ge 1$,

\m
\begin{itemize}
\item[(i)]
$\displaystyle c_{2k-1} = \frac{F_{k+1}}{F_k};  \qquad c_{2k} = \frac{L_{k+1}}{L_k}$;

\m
\item[(ii)]
$\displaystyle u_k(j) \,=\, \frac{F_{k+1} + jF_{k+2}}{F_k + jF_{k+1}} 
\;\mbox{ for all } \, j \ge 1$;

\m
\item[(iii)]
$\displaystyle v_k(j) \,=\, 
\frac{L_k + jF_{k+1} }{L_{k-1} + jF_k} \,=\,
\frac{(j-6)F_{k+1} + F_{k+2} }{(j-6)F_k + F_{k+1}}
\;\mbox{ for all } \,j \ge 1$.
\end{itemize}
\end{lemma}

\begin{proof} 
Recall that if $p_n/q_n$ is the $n$th convergent to
$[\ell_0; \ell_1, \dots, \ell_{N}]$ then, for any $n<N$ and 
any positive $x\in\R$ we have
\begin{equation} \labell{eq:pqn}
[\ell_0; \ell_1,\dots,\ell_{n-1},x] \,=\, 
\frac { p_{n-2} + x p_{n-1}}{q_{n-2} +x q_{n-1}}.
\end{equation}

(i) follows from induction on~$k$: 
The statement is true for~$k=1$. Assume it holds for~$k$.
By~\eqref{eq:pqn} with $x=1$, and by~\eqref{FL:repeat},
$$
c_{2k+1} \,=\, \frac{F_{k+1}+L_{k+1}}{F_k+L_k} \,=\, \frac{F_{k+2}}{F_{k+1}} .
$$
Then, by~\eqref{eq:pqn} with $x=5$, and by~\eqref{FL:repeat},
$$
c_{2k+2} \,=\, \frac{L_{k+1}+5F_{k+2}}{L_k+5F_{k+1}} \,=\, \frac{L_{k+2}}{L_{k+1}} .
$$

(ii) follows from (i)
by using $c_{2k+1} = [6;\{1,5\}^{\times (k-1)},1,6]$.

(iii) follows from (i) and \eqref{FL:repeat}.
\end{proof}

\begin{corollary}  \label{c:order}
For all $k \ge 1$, $s \ge 3$ and $t \ge 8$ we have
$$
c_{2k+1} < u_k(s+1) < u_k(s) < \dots < u_k(2) < u_k(1) = v_k(7) < v_k(t) < v_k(t+1) < c_{2k-1} .
$$
\end{corollary}

\proof  
This follows from Lemma~\ref{le:Fpq} and identity~\eqref{e:Fkkk}.
\proofend

The following corollary will be very useful.

\begin{corollary} \labell{co:hh} 
(i) 
Let $u = u_k(j) =: \frac pq$, where $j \ge 1$.
Then  
$$
q^2 \left( u^2-7u+1 \right) \,=\, j^2 + 7j + 1.
$$

\m
(ii)
Let $v = v_k(j) =: \frac pq$, where $j \ge 1$.
Then  
$$
q^2 \left( v^2-7v+1 \right) \,=\, j^2 - 5j- 5.
$$
\end{corollary}

\begin{proof} 
The proofs of (i) and (ii) are similar.
We prove~(ii).
In view of Lemma~\ref{le:Fpq}~(iii) we need to show
\begin{equation}  \label{e:LhF}
\left( L_k+j F_{k+1} \right)^2 - 7 \left( L_{k-1}+jF_k \right) \left( L_k+jF_{k+1} \right)
+ \left( L_{k-1}+j F_k\right)^2 \,=\,
j^2-5j-5 .
\end{equation}
Fix $j$. Identity~\eqref{e:LhF} is true for $k=1,2,3$.
It therefore holds for all~$k$ by Proposition~\ref{prop:proof}.
\end{proof}

\begin{defn}
We say that a point $a \in [\tau^4,7]$ is {\bf regular}
if for all $(d;\mm) \in \Ee$ with $\ell(\mm) = \ell(a)$
we have $\mu(d;\mm)(a) \le \frac{a+1}3$.
\end{defn}

\begin{proposition} \label{p:fewpoints}
Assume that all the points $c_{2k-1}$ and all the points
$$
u_k(j) \,\text{ with $k \ge 1$ and $j \ge 2$}
\quad \text{ and } \quad
v_k(j) \,\text{ with $k \ge 1$ and $j \ge 7$}
$$
are regular.
Then $c(a) = \frac{a+1}{3}$ on $[\tau^4,7]$.
\end{proposition}

A main ingredient in the proof will be the following

\begin{lemma} \label{le:zigzag} 
Consider the functions $\varphi(a) := \frac{a+1}{3}$ and $\psi(a) := \sqrt{a}$.
Fix $k \ge 1$. Then
\begin{itemize}
\item[(i)]
$\varphi \bigl(u_k(j+1)\bigr) > \psi \bigl(u_k(j) \bigr)$ for all $j \ge 1$;

\s
\item[(ii)]
$\varphi \bigl(v_k(j) \bigr) > \psi \bigl(v_k(j+1) \bigr)$ for all $j \ge 7$.
\end{itemize}
\end{lemma}

\proof
(i)
Abbreviate $u := u_k(j+1)$, $u'= u_k(j)$.
We need to show that $\frac{u+1}3 > \sqrt{u'}$, i.e.,
\begin{equation} \label{e:u9.2}
u^2+2u+1 \,>\, 9u' .
\end{equation}
Recall from Lemma~\ref{le:Fpq}~(iii) that
\begin{equation}  \label{e:uvh}
u = \frac{(j+1)F_{k+2}+F_{k+1}}{(j+1)F_{k+1}+F_{k}}, 
\qquad
u' = \frac{jF_{k+2}+F_{k+1}}{jF_{k+1}+F_{k}} .
\end{equation}
In particular, the denominator of $u$ is $q := (j+1)F_{k+1}+F_{k}$.
Applying Corollary~\ref{co:hh}~(i) to~$u$ we therefore find
$$
u^2 \,=\, \frac{(j+1)^2+7(j+1)+1}{q^2} +7u-1 
$$
and so \eqref{e:u9.2} is equivalent to
\begin{equation} \label{e:9vu1}
9(u'-u)q^2 \,<\, (j+1)^2+7(j+1)+1 .
\end{equation}
Using \eqref{e:uvh} and $F_{k+1}^2-F_kF_{k+2} = 1$ we compute
$$
u'-u \,=\, \frac{1}{\bigl((j+1)F_{k+1}+F_k \bigr) \bigl( jF_{k+1}+F_k \bigr)} . 
$$
Inequality~\eqref{e:9vu1} therefore becomes
\begin{equation} \label{e:9h1}
9 \,\frac{(j+1)F_{k+1}+F_k}{jF_{k+1}+F_k} \,<\, (j+1)^2+7(j+1)+1 .
\end{equation}
For all $j \ge 1$ the second factor on the LHS is $<2$ and the RHS is $\ge 19$,
and so~\eqref{e:9h1} holds for all $j \ge 1$. 

The proof of (ii) is similar (but slightly easier).
\proofend

\m
\ni
{\it Proof of Proposition~\ref{p:fewpoints}:}
Assume that $c(a) \le \frac{a+1}3$ does not hold on $[\tau^4,7]$.
Since $c(a) \ge \frac{a+1}3 > \sqrt{a}$ on
$(\tau^4,7]$,
Corollary~\ref{cor:fin} shows that $c(a)$ is piecewise linear 
on $(\tau^4,7]$.
Let~$S \subset (\tau^4,7)$ be the set of non-smooth points of~$c$ 
on $(\tau^4,7)$.
This set decomposes as $S = S_+ \cup S_-$,
where $S_+$ (resp.~$S_-$) consists of those $s \in S$ near which~$c$ is convex (resp.\ concave).
Note that for $s \in S_+$ we have $c(s) > \frac{s+1}3$.
By Proposition~\ref{prop:7easy}, $c(a)= \frac{a+1}3$ for 
$a \in [6 \frac{11}{12}, 7]$, and so 
the biggest point of~$S$ is in $S_-$.
This and $c(\tau^4) = \frac{\tau^4+1}3$ imply that the set~$S_+$ is non-empty.
Let $a_0= \max S_+$. Then $a_0 \in ( \tau^4,7 )$.
By Corollary~\ref{cor:fin}~(i) there exists $(d;\mm) \in \Ee$ and $\eps >0$
such that 
\begin{equation} \label{e:mua3}
c(z) = \mu (d;\mm)(z) \quad \mbox{ on }\; [ a_0,a_0+\eps ] .
\end{equation}
Abbreviate $\mu (z) := \mu(d;\mm)(z)$.
By~\eqref{e:mua3}, $\mu(a_0) = c(a_0) > \frac{a_0+1}3 > \sqrt{a_0}$.
Let $I$ be the maximal open interval containing $a_0$ such that 
$\mu(z) > \sqrt z$ for all $z \in I$.
By Lemma~\ref{le:I}, there exists a unique $a' \in I$ with $\ell (\mm)=\ell(a')$,
and $\ell(\mm) < \ell(z)$ for all other $z \in I$.
Further, by Proposition~\ref{prop:mainc},
the constraint $\mu(z)$ is given by two linear functions on~$I$:
$$
\mu(z) \,=\,
\left\{
\begin{array}{ll}
\alpha + \beta z & \mbox{ if }\;  z<a',\; z\in I, \\
\alpha'+ \beta'z & \mbox{ if }\;  z>a',\; z\in I,
\end{array}
\right.
$$
that is, $a'$ is the only non-smooth point of $\mu$ on~$I$.
By~\eqref{e:mua3}, and since $a_0 \in S_+$ and $\mu \le c$,
the point $a_0 \in I$ is also a non-smooth point of $\mu$, 
and hence $a'=a_0$.
Now~\eqref{e:mua3} and the fact that~$c$ is 
nondecreasing 
show that $\beta' \ge 0$.

Let $k \ge 1$ be such that $a_0 \in [ c_{2k+1},c_{2k-1} ]$.
Since $c_{2k+1}$ and $c_{2k-1}$ are regular by assumption, 
we have $a_0 \in (c_{2k+1},c_{2k-1} )$.
Note that
$u_k(j) \to c_{2k+1}$ and $v_k(j) \to c_{2k-1}$ as $j \to \infty$.
Let $u_-,u_+$ be the two neighboring points from the sequence 
$$
\cdots < u_k(s+1) < u_k(s) < \cdots < u_k(2) < u_k(1) = v_k(7) < v_k(t) < v_k(t+1) < \cdots  
$$ 
from Corollary~\ref{c:order}~(ii) with $a_0 \in [u_-,u_+]$.
Since $u_-$ and $u_+$ are regular by assumption, we have $a_0 \in ( u_-,u_+ )$.
Then $\mu(a_0) > \frac{a_0+1}3 > \frac{u_-+1}3 = \varphi(u_-) > \psi(u_+) = \sqrt{u_+}$ 
by Lemma~\ref{le:zigzag}.
Since $\beta' \ge 0$, it follows that $\mu(u_+) \ge \mu(a_0) > \sqrt{u_+}$, and hence $u_+ \in I$,
and hence $\ell(u_+) > \ell(a_0)$.
However, $\ell(z) > \ell(u_-)$ and $\ell(z) > \ell(u_+)$ for all $z \in (u_-,u_+)$;
in particular, $\ell(a_0) > \ell(u_+)$, a contradiction.
\proofend

As we see in the next two lemmas, one can prove that most of the points 
$u_k(j)$ and $v_k(j)$ are regular by direct arguments.

\begin{lemma} \label{le:ah}
The points $u_k(j)$ with $k \ge 1$ and $j \ge 2$ are regular.
\end{lemma}

\proof
Abbreviate $u := u_k(j)$ and $\frac pq := u$ in lowest terms.
Assume that $(d;\mm) \in \Ee$ is such that $\mu(d;\mm)(u) > \frac{u+1}3$ and $\ell(u)=\ell(\mm)$.
By Proposition~\ref{prop:obs1}~(i) and Corollary~\ref{co:hh}~(i) we can estimate
$$
\frac{d}{q \sqrt u} \,<\, \frac{3}{q \sqrt{u^2-7u+1}} \,=\, \frac{3}{\sqrt{j^2+7j+1}},
$$
an estimate independent of~$k$.
Note that $\frac{3}{\sqrt{j^2+7j+1}} < 1$ for $j \ge 2$.
Since $\ell(u)=\ell(\mm)$, we have $m_i \ge 1$ for all~$i$.
Therefore, 
$$
E \,\ge\, j \left( 1-\frac d{q\sqrt u} \right)^2 \,>\, j \left( 1- \frac{3}{\sqrt{j^2+7j+1}} \right) \,=:\, s(j) . 
$$ 
The function $s(j)$ is increasing in~$j$, and $s(3)>1$,
proving the lemma for $j \ge 3$ and all $k \ge 1$.

Assume now that $j=2$. In this case, $s(j) < 1$. 
We therefore need to use the better estimate 
$E = 1-\gl^2 < 1- \frac{2d^2y(u)}{3 \sqrt u}$
from Proposition~\ref{prop:obs1}~(ii).
With this estimate we have
\begin{eqnarray*}
0 \,=\, 
E+\gl^2-1 &>& 
2 \left( 1-\frac{d}{q\sqrt u} \right)^2 + \frac 23 \frac{d^2y(u)}{\sqrt u}-1 \\
&=& 
\left( \frac{2}{q^2u} + \frac 23 \frac{y(u)}{\sqrt u} \right) d^2 + \left( -\frac{4}{q \sqrt u}\right) d +1
\,=:\, f(d) .
\end{eqnarray*}
We need to show that $f(d) \ge 0$.
Since $f$ is a quadratic polynomial in~$d$, this holds if its discriminant is negative,
$$
\frac{16}{q^2u} \,<\, 4 \left( \frac{2}{q^2u}+\frac 23 \frac{y(u)}{\sqrt u} \right) .
$$
Multiplying by $\frac u8$ and using $y(u)=u+1-3\sqrt u$, this is equivalent to
$$
\frac{1}{q^2}+u \,<\, \frac{u+1}3 \sqrt u .
$$
Taking squares, replacing $u$ by $\frac pq$, and multiplying by $9q^4$,
this becomes
\begin{equation} \label{ine:-9}
0 \,<\, -9 + p^3 q - 7 p^2 q^2 + p\;\! q \left(-18 + q^2\right) .
\end{equation}
Recall now from Lemma~\ref{le:Fpq}~(ii) that $p = 2 F_{k+2}+ F_{k+1}$ and $q = 2 F_{k+1}+ F_k$.
Since $p\;\!q > 9$ for all~$k \ge 1$, \eqref{ine:-9} follows from the identities
\begin{equation*} 
1 \,=\,  p^2 - 7 p \;\!q + \left(-18 + q^2\right) ,
\end{equation*}
which hold true for $k=1,2,3$ and hence for all~$k$ by Proposition~\ref{prop:proof}.
\proofend

\begin{lemma} \label{le:bh}
The points $v_k(j)$ with $k \ge 1$ and $j \ge 8$ are regular.
\end{lemma}

\proof
Abbreviate $v := v_k(j)$ and $\frac pq := v$ in lowest terms.
Assume that $(d;\mm) \in \Ee$ is such that $\mu(d;\mm)(v) > \frac{v+1}3$ and $\ell(\mm)=\ell(v)$.
Again, by Proposition~\ref{prop:obs1}~(i) and Corollary~\ref{co:hh}~(ii),
$$
\frac{d}{q \sqrt v} \,<\, \frac{3}{q \sqrt{v^2-7v+1}} \,=\, \frac{3}{\sqrt{j^2-5j-5}},
$$
independent of~$k$.
Note that $\frac{3}{\sqrt{j^2-5j-5}} < 1$ for $j \ge 8$.
Since $\ell(v)=\ell(\mm)$, we have $m_i \ge 1$ for all~$i$.
Therefore, 
$$
E \,\ge\, j \left( 1-\frac d{q\sqrt v} \right)^2 \,>\, j \left( 1- \frac{3}{\sqrt{j^2-5j-5}} \right) \,=:\, t(j) . 
$$ 
The function $t(j)$ is increasing on $\{ j \ge 8 \}$, and $t(8) > 1$,
whence the lemma follows.
\proofend

For $k \ge 1$ and $i \ge 0$ we set 
\begin{equation} \label{def:bki}
b_k(i) \,:=\, v_k(1+3i) \,=\, 
\bigl[\, 6;\{1,5\}^{\times (k-1)}, 1, 1+3i \,\bigr] .
\end{equation}
Hence $b_k(2)=v_k(7),\, b_k(3)=v_k(10), \, \dots$. 
\begin{lemma} \label{le:ek}
$c(a) = \frac{a+1}3$ for all $a = b_k(i) $, $k \ge 1$, $i \ge 2$.
In particular, the points $v_k(7)$ are regular for all $k \ge 1$.
\end{lemma}

The proof is postponed to Corollary~\ref{c:42} in the next subsection.  
It uses the existence of special (nearly perfect) elements of~$\Ee$, 
rather than the estimates of Proposition~\ref{prop:obs1}.

\b
\NI
{\bf Proof of Theorem~\ref{thm:main} part (ii).}\,\,
By Proposition~\ref{p:fewpoints} it suffices to show that 
for all $k\ge 1$ the points $c_{2k-1}$, 
$u_k(j)$, $j \ge 2$, and $v_k(j)$, $j \ge 7$, are regular.  
Regularity holds for $u_k(j)$, $j \ge 2$, by Lemma~\ref{le:ah}
and for $v_k(j)$, $j \ge 8$, by Lemma~\ref{le:bh}. 
Moreover, when~$a$ belongs to the subsequence $b_k(i)$, $i \ge 2$, of the $v_k(j)$,
then Lemma~\ref{le:ek} makes the stronger statement that 
$c(a) = \frac{a+1}3$.  This holds in particular when $a = v_k(7) = b_k(2)$.  
Hence these points are regular.
Further, because the sequence $\bigl( b_k(i) \bigr)_{i \ge 2}$ converges to $c_{2k-1}$ 
as $i \to \infty$,
the continuity of~$c$ implies that $c(a) = \frac{a+1}3$ also at $a = c_{2k-1}$.  
Hence these points are also regular, which completes the proof.
\proofend

%%%%%%%%%%%%%%%%%%%%%%%%%%%%
\subsection{The classes $E \bigl( b_k(i) \bigr)$}
%%%%%%%%%%%%%%%%%%%%%%%%%%%%

Recall from Section~\ref{ss:JL} that for $n \ge 0$ the points $b_n = \frac{g_{n+2}}{g_n} < \tau^4$
are the break points of the Fibonacci stairs.
The next result shows their relation to the numbers $b_k(i)$, $k\ge 1$, $i\ge 0$,  
defined by~\eqref{def:bki}.

\begin{lemma} \label{le:b=b}
$b_{2k} = b_k(0)$ and $b_{2k+1} = b_k(1)$ for all $k \ge 1$.
\end{lemma} 

\proof
Using Proposition~\ref{prop:proof} 
and Definition~\ref{def:Lucas} we see that
$L_k+F_{k+1} = g_{2k+2}$ and $L_k+4F_{k+1} = g_{2k+3}$ for all $k \ge 0$.
Together with Lemma~\ref{le:Fpq}~(iii) we conclude that
\begin{eqnarray*}
b_k(0) &=& v_k(1) \;=\; \frac{L_k+ F_{k+1}\phantom{4}}{ L_{k-1}+ F_k } \;=\; \frac{g_{2k+2}}{g_{2k}} \;=\; b_{2k}, \\
b_k(1) &=& v_k(4) \;=\; \frac{L_k+4F_{k+1}}{ L_{k-1}+4F_k } \;=\; \frac{g_{2k+3}}{g_{2k+1}} \;=\; b_{2k+1},
\end{eqnarray*}
as required.
\proofend

The lemma says that the sequence $\bigl( b_n \bigr)$, $n \ge 2$,
extends to a double sequence $\bigl( b_k(i) \bigr)$, $k \ge 1$, $i \ge 0$,
where for each $k\ge 1$ the sequence $\bigl( b_k(i) \bigr)$, $i \ge 0$,
emanates from the pair $\bigl( b_{2k}, b_{2k+1} \bigr) = \bigl( b_k(0), b_k(1) \bigr)$.

\b
Recall from Section~\ref{ss:JL} that  $E(b_n) := \bigl( g_{n+1}; W(b_n) \bigr)$,
where $W(b_n) := g_n\, \ww(b_n)$.
In order to prove Lemma~\ref{le:ek}, 
we associate classes $E \bigl( b_k(i) \bigr)$
to all $b_k(i)$ as follows.
Let $b_k(i) =: \frac pq$.  
Let $\mm_k(i)$ be the tuple obtained from $q\, \ww \bigl(b_k(i) \bigr)$ 
by replacing its last block~$\left( 1^{\times (1+3i)} \right)$ 
by $\left( i,1^{\times (1+2i)} \right)$, 
and set $d_k(i) := q \bigl(1+b_k(i) \bigr)/3$.
Then define 
$E \bigl( b_k(i) \bigr) := \bigl( d_k(i); \mm_k(i) \bigr)$.

\s
Note that for $i=0,1$ we have (with $n=2k+i$) 
that $\mm_k(i) = g_n \,\ww \bigl( b_k(i) \bigr) = W(b_n)$ and,
by~\eqref{re:g1},
$d_k(i) = g_n \bigl( 1+ \frac{g_{n+2}}{g_n} \bigr)/3 = \left( g_n+g_{n+2} \right) = g_{n+1}$.
Therefore, 
\begin{equation}  \label{e:E=E}
E \bigl( b_k(0) \bigr) = E \bigl( b_{2k} \bigr),
\qquad
E \bigl( b_k(1) \bigr) = E \bigl( b_{2k+1} \bigr) \quad 
\mbox{ for all }\, k \ge 1.
\end{equation}
\begin{proposition} \label{prop:bkiE}
$E \bigl( b_k(i) \bigr) \in \Ee$ \;for all\, $k \ge 1$ and $i \ge 0$.
\end{proposition}

Before proving Proposition~\ref{prop:bkiE}, 
we show that it is the key to completing the calculation of $c$ on the interval $[1,7]$.

\begin{cor} \labell{c:41}  
Part (i) of Theorem~\ref{thm:ladder} holds.  
\end{cor}  

\begin{proof}  
This is immediate from equation \eqref{e:E=E}.
\end{proof}

\begin{cor} \labell{c:42} 
Lemma~\ref{le:ek} holds.
\end{cor}

\proof
We argue as in the proof of Corollary~\ref{cor:ladder}.
Let $(d;\mm) \in \Ee$. 
Write $b' := b_k(i)$, $d' := d_k(i)$, $\mm' := \mm_k(i)$,
so that $(d';\mm') = E(b_k(i))$.
If $(d;\mm) = \bigl( d';\mm' \bigr)$, then 
$$
\mu(d;\mm) \bigl(b'\bigr) \,=\, \frac{\mm' \cdot \ww (b')}{d'} \,=\,
\frac{3 b'}{b'+1} \,<\, \sqrt{b'}
$$
because $b' > \tau^4$.
If $(d;\mm) \neq (d';\mm')$, then $d \,d' \ge \mm \cdot \mm'$ by positivity of intersections.
By the definition of $d'$ and $\mm'$, and since $\mm$ is ordered,
this spells out to
$$
q\, d\, \tfrac{1+b'}{3} \,\ge\, \mm \cdot \mm' \,\ge\, q\, \mm \cdot \ww (b'), 
$$
i.e., $\tfrac{1+b'}{3} \ge \frac 1d \mm \cdot \ww (b') =: \mu(d;\mm) (b')$.
\proofend

\begin{remark}
\rm
(i)
Note that the classes $E \bigl( b_k(i) \bigr) = \bigl( d_k(i); \mm_k(i) \bigr)$
are perfect for $i=0,1$, but are not perfect for $i \ge 2$.
Since, nevertheless, at $b_k(i) =: \frac pq$ we have
\begin{equation} \label{e:mkw}
\mm_k(i) \cdot \ww \bigl( b_k(i) \bigr) \,=\, q\, \ww \bigl(b_k(i) \bigr) \cdot \ww \bigl( b_k(i) \bigr) \,=\, q\, b_k(i) \quad \mbox{ for all }\, i ,
\end{equation}
these classes are useful also for $i \ge 2$.

\m
(ii)
For $i \ge 3$ there are other choices for $\mm_k(i)$ that can be used
to prove Lemma~\ref{le:ek}.
All one needs is that~\eqref{e:mkw} holds.
For this, one just needs to alter the last block $1^{\times j}$ so that 
the sum of entries stays intact, while the sum of the squares goes up by~$i^2-i$.

If $i=2$, one has no choice: In order to make the sum of squares go up by~$2$ 
one must replace $1^{\times 7}$ by $2,1^{\times 5}$. 
If $i=3$, however, instead of replacing $1^{\times 10}$ by $3,1^{\times 7}$,
one can replace it by $2^3, 1^{\times 4}$.
The resulting class $\bigl( d_k(3); \mm_k'(3) \bigr)$ lies in $\Ee$ for all~$k$.
If $i=4$, instead of replacing $1^{\times 13}$ by $4,1^{\times 9}$,
one can make the sum of squares go up by~$12$ also by replacing 
$\mm_k(4) = \bigl( \dots, 1^{\times 13} \bigr)$ by one of
$$
\mm_k^{(1)}(4)   := \bigl( \dots, 3^{\times 2}, 1^{\times 7} \bigr), \quad
\mm_k^{(2)}(4)   := \bigl( \dots, 3, 2^{\times 3}, 1^{\times 4} \bigr), \quad
\mm_k^{(3)}(4)   := \bigl( \dots, 2^{\times 6}, 1 \bigr).
$$
However, while the classes 
$\bigl( d_k(4);\mm_k^{(2)}(4) \bigr)$ and $\bigl( d_k(4);\mm_k^{(3)}(4) \bigr)$
reduce to $(1;1,1)$, and thus lie in $\Ee$, 
the classes $\bigl( d_k(4);\mm_k^{(1)}(4) \bigr)$ do not reduce correctly.
\diam
\end{remark}

We now turn to the proof of Proposition~\ref{prop:bkiE}.  
It follows from Lemma~\ref{le:hh} and Proposition~\ref{p:Ekred} below.

\begin{lemma} \labell{le:hh} 
The class $\bigl( d_k(i); \mm_k(i) \bigr)$ satisfies the Diophantine conditions~\eqref{eq:ee} 
for the elements of~$\Ee$. 
\end{lemma}

\begin{proof}  
Write $b := b_k(i)$ and $(d;\mm) = \bigl( d_k(i); \mm_k(i) \bigr)$.
By Lemma~\ref{le:ww}, $\sum m_s = \sum q\, w_s(b) = q(b+1)-1 = 3d-1$.  
Further, $\sum m_s^2 = q^2 b + i^2-i$, and by Corollary~\ref{co:hh}~(ii), 
$$
q^2 \left( b^2-7b+1 \right) \,=\, \left(1+3i\right)^2-5 \left(1+3i\right)-5
\,=\, 9i^2-9i-9 .
$$
Therefore,
\begin{eqnarray*}
d^2+1 \,=\, 
\tfrac 19 q^2 \left( 1+b \right)^2 +1 &=&
\tfrac 19 q^2 \left( b^2-7b+1 \right) +q^2b + 1 \\
&=&
i^2-i+q^2b \,=\, \sum m_s^2 ,
\end{eqnarray*}
as required. 
\end{proof}

\begin{proposition}  \label{p:Ekred}
The classes $E \bigl( b_k(i) \bigr)$ reduce to $(1;1,1)$ for all $k \ge 1$ and $i \ge 0$.
\end{proposition}

\proof  
We fix $i \ge 0$, and argue by induction on $k \ge 1$.

\m
\ni
{\bf Step 1.}
Assume that $k=1$.
Set $j = 1+3i$.
The weight expansion of $b_1(i) = \left[ \;\!6;1,j \,\right] =: \frac pq$ is
\begin{equation} \label{e:1qh}
\tfrac 1q \bigl( (j+1)^{\times 6}, j, 1^{\times j} \bigr) .
\end{equation}
Together with~\eqref{eq:rsum} we find
$d_1(i) = \frac q3 \bigl( 1+b_1(i) \bigr) = \frac 13 \bigl( 6(j+1)+j+j+1 \bigr) = 
\frac 13 \left( 8j+7 \right) = 8i+5$, and so
$$
\bigl( d_1(i); \mm_1(i) \bigr) \,=\,
\bigl( 8i+5; (3i+2)^{\times 6}, 3i+1, i, 1^{\times (2i+1)} \bigr) .
$$
Applying five standard Cremona moves yields, successively,
\begin{eqnarray*}
\bigl( 7i+4; (3i+2)^{\times 3}, 3i+1, (2i+1)^{\times 3}, i, 1^{\times (2i+1)} \bigr) ; \\
\bigl( 5i+2; 3i+1, (2i+1)^{\times 3}, i^{\times 4}, 1^{\times (2i+1)} \bigr) ; \\
\bigl( 3i+1; 2i+1, i^{\times 5}, 1^{\times (2i+1)} \bigr) ; \\
\bigl( 2i+1; i+1, i^{\times 3}, 1^{\times (2i+1)} \bigr) ; \\
\bigl(  i+1; i, 1^{\times (2i+2)} \bigr) .
\end{eqnarray*}
The standard Cremona move maps
$(s+1; s, 1^{\times t})$ to $(s; s-1, 1^{\times (t-2)})$ for any $s \ge 1$ and $t \ge 2$.
Applying $i$~more standard Cremona moves therefore moves $\bigl(  i+1; i, 1^{\times (2i+2)} \bigr)$
to $(1;1,1)$.

\b
\ni
{\bf Step 2.}
Assume by induction that $\bigl( d_k(i);\mm_k(i) \bigr)$ reduces to $(1;1,1)$.
We shall show that $\bigl( d_{k+1}(i);\mm_{k+1}(i) \bigr)$ reduces to 
$\bigl( d_k(i);\mm_k(i) \bigr)$
by five standard Cremona moves. 

The end of the weight expansion $q\:\! \ww \bigl( b_k(i) \bigr) = q\:\! \ww \bigl( v_k(j) \bigr)$ is
$$
\bigl( \dots,\; (48 j+41)^{\times 5},\; 41 j + 35,\; (7j+6)^{\times 5},\; 6j+5,\; (j+1)^{\times 5},\; j,\; 1^{\times j} \bigr) .
$$
Using $F_{k+1} =L_k+F_k$ and $L_{k+1} = 5 F_{k+1} +L_k$ from~\eqref{FL:repeat}, and 
$$
L_{k+1}-L_k \,=\, 5 L_k + (L_k-L_{k-1}) ,
$$ 
which follows from these two formulae, we see that in general,
\begin{eqnarray} \label{e:mkj} 
\mm_k(i)  &=&
\bigl(\,
(jF_k +L_{k-1})^{\times 6},\; (j+1)L_{k-1} - L_{k-2},\;
 \\  \notag&&\hspace{.1in}
  (jF_{k-1}+L_{k-2})^{\times 5},\; (j+1)L_{k-2}-L_{k-3},\; \dots,\;  
  \\  \notag&&\hspace{.1in}
  (jF_2+L_1)^{\times 5},\;
(j+1)L_1-L_0,\; (j+1)^{\times 5},\; j,\; i,\;1^{\times (j-i)} 
\,\bigr) .
\end{eqnarray}

It will be convenient to express the numbers $d_k(i)$ in terms of the
{\it even} Fibonacci numbers $h_k := f_{2k}$.    
Thus
\begin{gather}\labell{eq:evenF}
h_1=1,\; h_2=3,\; h_3=8,\; h_4=21,\; h_5 = 55,\; h_6 = 144, \,\dots
%\\ \notag
%\; h_7 = 377,\; h_8 = 987,\; h_9= 2584, \dots
\end{gather}

\begin{lemma} \label{le:dkj}
$d_k(i) = h_{2k+2} + (i-2)h_{2k+1}$
\end{lemma}

\proof
We fix $i$ and $j=1+3i$, and write $d_k = d_k(i)$, $\mm_k = \mm_k(i)$, $b_k = b_k(i)$, etc.
Set $b_k = \frac{p_k}{q_k}$.
By Lemma~\ref{le:Fpq}~(iii), 
$$
p_k = F_{k+2} + (j-6)F_{k+1}, \qquad 
q_k = F_{k+1} + (j-6)F_k .
$$
Therefore, 
$3 d_k = q_k (b_k+1) = p_k+q_k = F_{k+2} + (j-5)F_{k+1} + (j-6)F_k$.
We thus need to show that
\begin{equation*} 
F_{k+2} + (j-5) F_{k+1} + (j-6)F_k \,=\, 3 h_{2k+2} + 3(i-2) h_{2k+1} .
\end{equation*}
This holds true for $k=1$ and $k=2$, 
and so it holds true for all $k \ge 1$ by Proposition~\ref{prop:proof}.
\proofend

\begin{lemma} \label{le:red}
The class $\bigl( d_{k+1}(i); \mm_{k+1}(i) \bigr)$ reduces to 
$\bigl( d_{k}(i); \mm_{k}(i) \bigr)$ by five Cremona transforms.
\end{lemma}

\begin{proof}
Fix $i$.
By~\eqref{e:mkj} and Lemma~\ref{le:dkj}, the entries of $\bigl( d_{k+1}(i); \mm_{k+1}(i) \bigr)$
are given by linear formulas in Fibonacci numbers, that depend only on~$k$.
Using \eqref{e:mkj} and Lemma~\ref{le:dkj}
one checks for $k=1$ and $k=2$ that $\bigl( d_{k+1}(i); \mm_{k+1}(i) \bigr)$ reduces to 
$\bigl( d_{k}(i); \mm_{k}(i) \bigr)$ by five Cremona transforms
with equal reordering at each stage. 
The lemma thus follows from Proposition~\ref{prop:proof}.  
\end{proof}

The proof of Proposition~\ref{p:Ekred} is complete.

%%%%%%%%%%%%%%%%%%%%%%%%%%%%%%%%%%%%%%%%%%%%%%
\subsection{The ghost stairs} \label{ss:ghost}
%%%%%%%%%%%%%%%%%%%%%%%%%%%%%%%%%%%%%%%%%%%%%%

In this section we compute the contribution of the classes 
$E \bigl( b_k(i) \bigr) = \bigl( d_k(i); \mm_k(i) \bigr)$, $i \ge 2$, to the graph of 
$c(a) = \frac{a+1}3$ on $\left[ \tau^4,7 \right]$.
The lemma and the proposition below are not needed for the results of this paper,
but they illuminate the role of these classes.

\begin{lemma} \label{le:notobs}
Assume that $i \ge 3$. 
Then $\mu \bigl( d_k(i);\mm_k(i) \bigr)(a) \le \sqrt a$ for all $a>1$ 
and all $k \ge 1$.
\end{lemma}

\proof
Write $d=d_k(i)$, $\mm = \mm_k(i)$.
Assume that $\mu (d;\mm)(a) > \sqrt a$ for some~$a \ge 1$.
Let~$I$ be the open interval such that $\mu (d;\mm)(z) > \sqrt z$ and $a \in I$.
Let $a_0$ be the unique point in~$I$ with $\ell (a_0) = \ell (\mm)$.
Recall that
$$
\mm \,=\, \bigl( \dots,\, i,\, 1^{\times (2i+1)} \bigr) .
$$
Since $i \ge 3$, the last block of $\ww (a_0)$ must have length $2i+1$
according to Lemma~\ref{le:atmost1}~(i).
But 
$$
\left| i-(2i+1) \right| \,=\, i+1 \,\ge\, \sqrt{2i+1} \,,
$$
in contradiction to Lemma~\ref{le:irrel}~(i).
\proofend

Thus, surprisingly, for $i \ge 3$ the classes $E \bigl( b_k(i) \bigr)$ give no embedding constraints,
but nevertheless are most useful to find $c(a)$ on $[\tau^4,7]$.

\m
We now look at the case $i=2$.
For $k \ge 1$ write $e_k := b_k(2) = v_k(7)$ and $E(e_k) = (d_k;\mm_k) := \bigl( d_k(2);\mm_k(2) \bigr)$.
Recall from Lemma~\ref{le:dkj} that $d_k = h_{2k+2}$,
where $h_{2k+2} = f_{4k+4}$ is an even Fibonacci number.  
We now show that the corresponding constraint functions $\mu(d_k;\mm_k)$   
form a staircase whose properties echo
that of the Fibonacci stairs on the other side of~$\tau^4$.  
However this staircase does not add anything new to the graph of~$c(a)$ because
it never rises above the line $y=\frac{a+1}3$.
Thus we call it the {\it ghost stairs}.
Note also that, although 
$E(e_k)$ is made from~$\ww(e_k)$ and so influences~$c(a)$ at~$a=e_k$,
it gives a constraint that is centered at the convergent $c_{2k+1} < e_k$.

\begin{prop}[The ghost stairs] \labell{prop:newmd} 
$$
\mu(d_k;\mm_k)(z) \,=\, 
\left\{
\begin{array}{ll}
\frac {z+1}3                &\mbox{ for }\, z \in \left[ c_{2k}, c_{2k+1} \right] , \\
\frac {h_{2k+3}}{h_{2k+2}}  &\mbox{ for }\, z \in \left[c_{2k+1},e_k \right] .
\end{array}
\right.
$$
\end{prop}

Since for each $k$ we have $c_{2k} < \tau^4$, the proposition shows that 
$\mu(d_k;\mm_k)(z) = c(z) = \frac {z+1}3$ on $[\tau^4, c_{2k+1}]$.

\proof
Fix $k$, and 
recall from Lemma~\ref{le:Fpq} and Corollary~\ref{c:order} that
\begin{eqnarray} 
c_{2k}   &=&  \bigl[ 6; \{1,5\}^{(k-1)},1,5 \bigr], \label{e:c2k} \\
c_{2k+1} &=&  \bigl[ 6; \{1,5\}^{(k-1)},1,6 \bigr], \label{e:c2k1}\\
e_{k}    &=&  \bigl[ 6; \{1,5\}^{(k-1)},1,7 \bigr] \,=\, \tfrac{L_k+7F_{k+1} }{L_{k-1}+7F_k }, \label{e:ek}
\end{eqnarray}
and that $c_{2k} < c_{2k+1} < e_k$.
If  $z \in (c_{2k}, c_{2k+1})$ then $z = \bigl[ 6; \{1,5\}^{(k-1)},1,g,h, \dots \bigr]$ with $g=5$ and $h\ge 1$, 
while if $z \in (c_{2k+1}, e_{k})$ then $z = \bigl[ 6; \{1,5\}^{(k-1)},1,g,h, \dots \bigr]$ 
with $g=6$ and $h\ge 1$.
In both cases, the weight expansion has the form
\begin{eqnarray}\labell{eq:zzz}
\ww(z) &=& \bigl( 1^{\times 6}, z-6, (7-z)^{\times 5}, 6z-41, \dots,
x_{2k-1}(z) = z L_{k-1}-L_{k}, \\ \notag
&&\qquad  
\bigl( x_{2k}(z) = F_{k+1}-z F_k \bigr)^{\times g}, \bigl( x_{2k+1}(z) \bigr)^{\times h}, \dots \bigr) .
\end{eqnarray}
Further $\ww(e_k)$ begins the same way, but ends at  the block of terms
$$
\bigl( x_{2k}(z) \,=\, F_{k+1}-z F_k \bigr)^{\times 7} \,=\, 
\bigl(\alpha_{2k}^{e} +z \beta_{2k}^{e}\bigr)^{\times 7},
$$
where the last expression uses the elements $\alpha^{e}_j$ and $\beta^{e}_j$ of 
equation~\eqref{eq:itere} with $a = e := e_k$.   

Since $e_k$ has $N+1$ blocks where $N$ is even, 
Corollary~\ref{cor:mirror2} implies that
$$
\sum_j \ell_j \,x_j(e_k) \bigl(\alpha_j^{e} +z\beta_j^{e}\bigr) = e_k = 
\tfrac{L_k+7F_{k+1} }{L_{k-1}+7F_k }.
$$
Further, $\mm_k = q \,\ww(e_k)$  where $q := L_{k-1}+7F_k$,
except for the last block where we have $2,1^{\times 5}$ instead of  $1^{\times 7}$. 
When $z$ has $g=6$  it follows that 
$$
\mm_k\cdot \ww(z) = \mm_k\cdot \ww(e_k)  = q\,\ww(e_k) \cdot \ww(e_k) = q\,e_k = L_{k}+7F_{k+1} .
$$
Thus $\mu(d_k;\mm_k)(z)$ is constant on this interval.  On the other hand, if $g=5$ then $x_{2k+1}(z)= zL_{k} - L_{k+1}$ and we find 
\begin{eqnarray*}
\mm_k \cdot \ww(z) &=& 
\sum_j q\,\ell_j\, x_j(e_k) \bigl(\alpha_j^e +z\beta_j^e\bigr) - x_{2k}(z) + x_{2k+1}(z)\\
&=&  q\,e_k -F_{k+1}+z F_k + z L_{k}-L_{k+1}\\
&=& 6F_{k+1} + L_k-L_{k+1}+ z(L_{k}+F_k)\\
&=& (1+z)F_{k+1}.
\end{eqnarray*}
But 
$$
3d_k = q(1+e_k) = L_{k-1}+7F_k + L_k+7F_{k+1} = 9F_{k+1}.
$$
Therefore $\mu(d_k;\mm_k)(z)=(z+1)/3$  for
$z \in [c_{2k}, c_{2k+1}]$ and it remains to check that
$c_{2k+1} + 1 = 3 h_{2k+3}/h_{2k+2}$.
Since $c_{2k+1} = F_{k+2}/F_{k+1}$ and $F_{k+1} = 3h_{2k+2}$
this  reduces to the identity
$$
f_{4k+8} + f_{4k+4} = 3  f_{4k+6},
$$
which is readily checked using Proposition~\ref{prop:proof}.
\proofend

\begin{remark}\labell{rmk:evenF}
{\rm
(i) 
At the points $v_k(j)$ with $j \ge 7$, 
Theorem~\ref{thm:main}~(ii) implies that
$c\bigl( v_k(j)\bigr) = \frac{v_k(j)+1}3$, 
which by Lemma~\ref{le:Fpq}~(iii) can be written as
$$
c \left( \frac{\ell_{4k+2}+jf_{4k+4}}{\ell_{4k-2}+jf_{4k}} \right)
\,=\,
\frac{\ell_{4k}+jf_{4k+2}}{\ell_{4k-2}+jf_{4k}} ,
$$
where $\ell_{4k+2} = 3L_k$ as in Definition~\ref{def:Lucas}.
In particular, at $b_k(2)=v_k(7)$ and $b_k(3)=v_k(10)$,
$$
c \left( \frac{h_{2k+3}}{h_{2k+1}} \right)  \;=\; \frac{h_{2k+2}}{h_{2k+1}} 
\qquad \mbox{ and } \qquad
c \left( \frac{\ell_{4k+5}}{\ell_{4k+1}} \right)  \;=\; \frac{\ell_{4k+3}}{\ell_{4k+1}},
$$
where the $h_k$ are the even Fibonacci numbers of~\eqref{eq:evenF}.
On the other hand, on the left of $\tau^4$, where $\frac{a+1}3 < \sqrt{a}$, 
we have by Theorem~\ref{thm:main} that
$$
c(b_n) \,=\, c \left( \frac{g_{n+2}}{g_n} \right) \,=\, \frac{g_{n+2}}{g_{n+1}} \,=\,
\frac{b_{n+1}+1}3 .
$$
In other words, the function $c$ attains the value $\frac{b_n+1}3$ already at $b_{n-1}$.\MS

\NI
(ii)
Recall from Section~\ref{ss:JL} that $a_n = \bigl( \frac{g_{n+1}}{g_n} \bigr)^2$,
and that on the left of $\tau^4$, the classes $W'(a_n)$ obtained from 
$W(a_n) = g_n^2\;\! \ww(a_n)$
by adding one~$1$ were very useful to establish the Fibonacci stairs.
One may try to define similar classes at $a_n' := \bigl( \frac {h_{n+1}}{h_n} \bigr)^2$.
Denote by $W''(a_n')$ the sequence obtained from $h_n^2\;\! \ww(a_n')$
by removing three of the~$1$s at the end, and adding one~$2$.  
Thus when $n=3$ we get 
$$
a_3' = \bigl({\ts \frac{21}8\bigr)^2,} \qquad  
W''(a_3') = \left( 64^{\times 6}, 57, 7^{\times 8}, 2,1^{\times 4} \right) .
$$
It is easy to check that the tuple $\bigl( h_n h_{n+1}; W''(a_n') \bigr)$ 
satisfies the Diophantine equations~\eqref{eq:ee}.
However, when $n\ge 3$ this is {\it not}\/ an element of~$\Ee$ 
because it has negative intersection with the class $(3;2,1^{\times 6})\in \Ee$.
On the other hand, when $n=2$ this gives $(24;9^{\times 7},2,1^{\times 6}) \in \Ee$
which as we will see in Theorem~\ref{thm:78} does give an obstruction near $a=7\frac 17$,
and when $n=1$ we get $(3;2,1^{\times 6})$ itself.  
\diam
}
\end{remark}

%%%%%%%%%%%%%%%%%%%%%%%%%%%%%%%%%%%%%%%%%%%%%%%%%%%%%%%%%%
\section{The interval $[7,9]$} \labell{s:78}
%%%%%%%%%%%%%%%%%%%%%%%%%%%%%%%%%%%%%%%%%%%%%%%%%%%%%%%%%%

This section calculates $c$ on the interval~$[7,9]$.  
The main arguments are contained in~\S\ref{ss:78} 
and~\S\ref{ss:89}. 
We begin in~\S\ref{ss:greater7} by establishing some estimates
that are most useful on $[8,9]$  
but are also needed for some of the arguments concerning $[7,8]$ 
such as Lemma~\ref{le:7less}.

%%%%%%%%%%%%%%%%%%%%%%%%%%%%%%%%%%%%%%%%%%%%%%%%%%%
\subsection{Preliminaries} \labell{ss:greater7}
%%%%%%%%%%%%%%%%%%%%%%%%%%%%%%%%%%%%%%%%%%%%%%%%%%%

We begin with a simple result about continued fractions.
Let $q_n(a)$ be the denominator of the $n$th convergent 
$[\ell_0;\ell_1,\dots,\ell_n]$ to the continued fraction
$$
a := [\ell_0;\ell_1,\dots,\ell_N] = \ell_0+ \frac{1}{\ell_1 + \frac{1}{\ell_2+\dots}}.
$$
Thus $q_1(a)=\ell_1,q_2(a)=1+\ell_1\ell_2$ and, in general,
$q_n (a) = \ell_n q_{n-1}(a) + q_{n-2}(a)$. Then an easy induction argument shows that:

\begin{sublemma} \labell{sl:L}  
Let $L := \sum_{j=1}^N\ell_j$.  
Then $q_N(a) \ge L$. 
\end{sublemma}

In the sequel, we abbreviate $\si := \sum_{i>\ell_0}\eps_i^2 < 1$
and $\si' := \sum_{\ell_0 < i \le M-\ell_N} \eps_i^2 \le \si'$.

\begin{lemma}\labell{le:7eps} 
Assume that $(d;\mm)\in \Ee$ is such that 
$\mu (d;\mm)(a) > \sqrt{a}$ for some $a\in (\tau^4,9)$ with $\ell(a)=\ell(\mm)$.  
Assume further that $y(a)>\frac 1q$ where $q := q_N(a)$, 
and denote $v_M := \frac d{q\sqrt a}$.
Then

\vspace{0.2em}
\begin{itemize}
\item[(i)] 
$|\sum_{i\ge 1}\eps_i| \le  \sqrt {\si L}$

\vspace{0.2em}
\item[(ii)] 
If $v_M<1$ then 
$|\sum_{i\ge 1}\eps_i|\le \sqrt {\si' L}$.

\vspace{0.2em}
\item[(iii)] 
If $v_M \le \frac 12$, then $v_M > \frac 13$ and $\si' \le \frac 12$.
If $v_M \le \frac 34$ , then $\si' \le \frac 78$.

\vspace{0.2em}
\item[(iv)]  
Define $\delta := y(a)-\frac 1q >0$.
Then
$$
d \,\le\, 
\tfrac{\sqrt{a}}{\delta} \left( \sqrt{\si L}-1\right) 
\,\le\,
\tfrac{\sqrt{a}}{\delta} \left( \sqrt{\si q}-1\right) 
\,<\, \tfrac{\sqrt{a}}{\delta} \left( \tfrac{\si}
{\de v_M}-1\right) .
$$
Further, if $v_M <1$, then $\si$ can be replaced by $\gs'$.
In particular, always
$$
d \,<\, \tfrac{\sqrt{a}}{\delta}  \left( \tfrac 2\delta -1 \right)
  \,<\, \tfrac{2\sqrt{a}}{\delta^2}.
$$
\end{itemize}
\end{lemma}

\begin{proof} 
{\bf Step 1:}   $\sum_{i\ge 1}\eps_i<0$.

Proposition~\ref{prop:obs} (iv) states that
\begin{equation}\labell{eq:obsiv}
 {\ts -\sum\eps_i= 1+\frac d{\sqrt a}\bigl(y(a)-\frac 1q\bigr).}
\end{equation}
Since we assume that $y(a) > \frac 1q$, Step 1 is immediate.

\MS
\NI  
{\bf Step 2:}  
$\sum _{i>\ell_0}|\eps_i| \;\ge \;|\sum_{i\ge 1}\eps_i|$.
\SSS

If $\sum_{i\le \ell_0}\eps_i\ge 0$, then by Step 1  we have
$$
|\sum_{i\ge 1}\eps_i|\le |\sum_{i>\ell_0}\eps_i|\le \sum_{i>\ell_0}|\eps_i|,
$$
as required.    
Therefore, suppose that $\sum_{i\le \ell_0}\eps_i<0$.
Let  $P= \{ i>\ell_0 \mid \eps_i >0 \}$ and
$Q= \{ i>\ell_0 \mid \eps_i \le 0 \}$.  
Because $w_i=1$ for $i\le \ell_0$, we have
\begin{eqnarray*}
0 < \eps\cdot\ww &=& \sum_{i\le \ell_0}\eps_i + \sum_{i\in P}\eps_i w_i -\sum_{i\in Q}|\eps_i|w_i\\
&<& \sum_{i\le \ell_0}\eps_i +\sum_{i\in P}\eps_i.
\end{eqnarray*}
Therefore
$$
0 \,>\, \sum_{i\ge 1} \eps_i \ge \sum_{i\in Q} \eps_i
\,=\, -\sum_{i\in Q} |\eps_i| \,\ge\, -\sum_{i>\ell_0} |\eps_i| .
$$

\NI  {\bf Step 3:}  {\it Proof of (i)}.
Let $a=[\ell_0;\ell_1,\dots,\ell_N]$ as above, and 
write $\eps$ as $N+1$ blocks each of length $\ell_j$.   
Assume first that $\eps_i$ is constant on each block
with absolute value~$\de_j$.
Let $\nu_j = \ell_j\de_j^2$ so that $|\de_j| = \sqrt{\frac{\nu_j}{\ell_j}}$.
Then
$$
\sum_{i>\ell_0} \eps_i^2 = \sum_{j\ge 1} \ell_j \de_j^2 = \sum \nu_j = \si.
$$
Hence, by Step 2,
\begin{eqnarray*}
\Bigl| \sum_{i\ge 1} \eps_i \Bigr|
&\le & 
\sum_{i>\ell_0} |\eps_i| = \sum \ell_j\,|\de_j| \\
&=& \sum \ell_j \sqrt{\tfrac{\nu_j}{\ell_j}} \\
&=&  \sum \sqrt{{\nu_j}{\ell_j}}\\
&\le& \sqrt{\sum \ell_j}\sqrt{\sum\nu_j} \le  \sqrt {\si L}.
\end{eqnarray*}
This proves (i) in the case when the $\eps_i$ are constant 
on the $j$th block for all $j\ge 1$.
But by Lemma~\ref{le:atmost1} the only other possibility is that
there is precisely one block, say the $J$th, on which $\eps_i$ is not constant.  
In that case we subdivide this block into two subblocks of lengths $\ell_J -1$ and $1$.
Since the upper bound $\sqrt{\si L}$ depends only on the sum of the $\ell_j$, the argument goes through as before. 
\MS

\NI  {\bf Step 4:}  
{\it Proof of (ii)}.
We abbreviate $M'=M-\ell_N$, 
and write $v_i :=  \frac d{\sqrt a}w_i$.
If the $v_i$ are constant on the last block and if
$v_M := \frac{d}{q\sqrt a} < 1$,
then $m_{M'+1} = \dots = m_M =1$, and so
$\eps_{M'+1} = \dots = \eps_M = 1- v_M>0$. %\frac{d}{q\sqrt a} >0$.
Since also $\sum_i\eps_i<0$, we have
$\left| \sum_i \eps_i \right| \le \left| \sum_{i=1}^{M'} \eps_i \right|$. 
Hence, the  argument in Step~3 adapts to show that
$$
\Bigl| \sum_{i \ge 1} \eps_i \Bigr| 
\,\le\, \left| \sum_{i=1}^{M'} \eps_i \right|
\,\le\,
\sum_{i>\ell_0}^{M'} |\eps_i|
\,\le\, \sqrt {\si'L} .
$$

\ni
{\it Proof of (iii)}.
Assume that $v_M \le \frac 13$.
If $\ell_N \ge 3$, then 
$$
1 > \eps_{M'+1}^2 +\dots + \eps_M^2\ge 3 \left(\tfrac 23\right)^2 >1,
$$ 
a contradiction.
If $\ell_N =2$, then $v_{M-2} = v_{M-1}+v_M = 2v_M \le \frac 23$, and so
$$
1 > \eps_{M-2}^2 +2\eps_M^2 \ge \left(\tfrac 13\right)^2+ 2 \left(\tfrac 23\right)^2 =1,
$$ 
a contradiction. Further,
$\eps_M \ge \frac 12$ implies that
$$
\si'\,\le\, \sum_{\ell_0<i \le M-2} \eps_i^2 \,\le\, \si-\tfrac12 \,\le\, \tfrac 12.
$$
The second claim in (iii) is proved similarly.

\medskip
\ni
{\it Proof of (iv)}.
We use Sublemma~\ref{sl:L} and equation~\eqref{eq:obsiv}
to estimate
\begin{equation}  \label{est:d}
\sqrt{\sigma q} \,\ge\, \sqrt{\sigma L} \,\ge\, 1+ \tfrac{d}{\sqrt a} \left( y(a)-\tfrac 1q \right)
\,=\, 1+\tfrac{d}{\sqrt a} \delta \,=\, 1+\delta q v_M \,>\,  \delta q v_M.
\end{equation}
Therefore, $\sqrt q < \frac{\sqrt \sigma}{\de v_M }$, and so, using again~\eqref{est:d},
$$
d 
\,\le\, \tfrac{\sqrt a}{\delta} \left( \sqrt{\sigma L}-1\right) 
\,\le\, \tfrac{\sqrt a}{\delta} \left( \sqrt{\sigma q}-1\right) 
\,<\,   \tfrac{\sqrt a}{\delta}  \left( \tfrac{\sigma}{\de v_M} -1 \right) .
$$
If $v_M < 1$, we repeat this argument with $\gs$ replaced by $\gs'$.
This completes the proof.
\end{proof}

%%%%%%%%%%%%%%%%%%%%%%%%%%%%%%%%%%%%%%%%%%%%%%%%%%%%%%%%%%%%
\subsection{The interval $[7,8]$} \label{ss:78}
 %%%%%%%%%%%%%%%%%%%%%%%%%%%%%%%%%%%%%%%%%%%%%%%%%%%%%%%%%%%%

In this section we calculate $c(a)$ on the interval $[7,8]$.
At some places, we will use the computer. 
We will therefore first prove a weaker result that does not use the computer.
 
\begin{prop}\labell{prop:7finite} 
There are only finitely many $(d;\mm)\in \Ee$ for which
there is $ a\ge 7$ with $c(a) = \mu(d;\mm)(a)>\sqrt {a}$.
\end{prop}

\begin{proof} 
Suppose that $\mu(d;\mm)(a)>\sqrt {a}$ for some $a\ge 7$.
Let~$I$ be the maximal open interval containing~$a$ on which $\mu(d;\mm)(z)>\sqrt z$, 
and let $a_0\in I$ be the unique element with $\ell(a_0) = \ell(\mm)$.   
(This exists by Lemma~\ref{le:I}.)  
If $7 \in I$, then clearly $a_0=7$ so that $(d;\mm)$ belongs to the finite set~$\Ee_7$.
Otherwise, $a_0>7$. 
In particular, $y(a_0) > y(7) = 8-3\sqrt 7 > \frac 1{20}$.
Moreover $a_0<9$ by Corollary~\ref{cor:2}.

Now write $a_0=p/q$.  
There are only finitely many $a=\frac pq \in [7,9]$ with $q \le 40$,
and for each of them Corollary~\ref{cor:fin} shows that there are only finitely many  
obstructive~$(d;\mm)$. 
We can therefore assume that $q := q(a_0) \ge 40$ so that $y(a_0)-\frac 1q \ge \frac 1{40} >0$.
Since $\ell(a_0)= \ell(\mm)$  we can apply the last statement of Lemma~\ref{le:7eps} to conclude that   
$$
d \,\le\, 2 (40)^2 \sqrt{a_0} \,<\, 6 (40)^2.
$$
Since for each $D$ there are only finitely many $(d;\mm) \in \Ee$ with $d\le D$, 
this completes the proof.
\end{proof}

\begin{rmk} \labell{rmk:7finite}
\rm  
The result in Proposition~\ref{prop:7finite} clearly extends to any interval of 
the form~$[a,b]$ provided that $a>\tau^4$.
\diam
\end{rmk}

We already know that $c(a) = \frac 83$ on $[7, 7 \frac 19]$ by Proposition~\ref{prop:7easy}.
We can therefore assume that $a \in [7\frac 19,8]$.

In order to explain our notation in Theorem~\ref{thm:78} below, 
we work out the constraint given by the class 
$$
 (d;\mm) \,=\, \left( 48; 18^{\times 7},3,
2^{\times 7}\right) \,\in\, \Ee.
$$
Note that $\ell(\mm) = 7+8 = \ell (7 \frac 18)$.
It gives the constraint 
$\mu (d;\mm) (7 \frac 18) = \frac{1025}{384} > \sqrt{7 \frac 18}$ at $7 \frac 18$.
For $a=7+x$ with $x \in [\frac 19, \frac 18]$ we have
$\ww (a) = \left( 1^{\times 7}, x^{\times 8}, \dots \right)$.
Therefore,
$$
\mm \cdot \ww (a) = 7 \cdot 18 + 3x +14x = 126+17x = 7 + 17a,
$$
and so $\mu(d;\mm)(a) = \frac 1{48}(7+17a)$.
Note that $\frac 1{48}(7+17a) = \sqrt a$ at $u_{\frac 18} := 7.12499$ 
(where the last decimal is rounded).
Similarly, for $a=7+x$ with $x \in [\frac 18, \frac 17]$ we have
$\ww (a) = \left( 1^{\times 7}, x^{\times 7},1-7x, \dots \right)$.
Therefore,
$$
\mm \cdot \ww (a) = 7 \cdot 18 + 3x +12x +2-14x = 128+x = 121 + a,
$$
and so $\mu(d;\mm)(a) = \frac 1{48}(121+1a)$.
Note that $\frac 1{48}(121+a) = \sqrt a$ at $v_{\frac 18} := 7.12501$ (where the last decimal is rounded).
The interval containing $a=7 \frac 18$ on which this class gives a constraint is therefore
$I_{\frac 18} := [u_{\frac 18}, v_{\frac 18}]$, and
\begin{equation*} 
\mu(d;\mm)(z) \,=\,
\left\{
\begin{array} {rl}
 \frac 1{48}(7+17z) &\text{ if } z \in \bigl[u_{\frac 18}, 7 \frac 18\bigr] \\ [0.6em]
 \frac 1{48}(121+z) &\text{ if } z \in \bigl[7 \frac 18, v_{\frac 18}\bigr].
 \end{array}\right.
\end{equation*}
All this is expressed in the first row of the table below.
In the same way we compute $(A,B)$, $(A',B')$, $u_x$, $v_x$ and $\mu(a) := \mu(d;\mm)(a)$ at $a=7+x$ 
for the other seven classes in the table below,
where we write 
$\mu(z) = \frac 1d(A+Bz)$ for $z$ just less than~$a$ and 
$\mu(z) = \frac 1d(A'+B'z)$ for $z$ just greater than~$a$.
Note that the eight intervals $[u_x,v_x]$ are all disjoint.

\begin{thm} \label{thm:78}
For $a \in [7 \frac 19,8]$ we have $c(a) = \sqrt a$ except for the eight intervals $[u_x,v_x]$
where $c(a)$ is as described in the following table.
\end{thm}
\begin{equation}\labell{t:1}
\begin{array}{|l||l|c|c|l|l|c|r|} 
 \hline
 a&(d;\mm)& (A,B)& (A',B')&u_x&v_x&\mu(a) & \mu(a)-\sqrt a \\ 
 \hline 
 7\frac 18&(48;18^{\times 7},3,2^{\times 7})&(7,17)&(121,1)& 7.12499 & 7.12501 
      &\tfrac {1025}{384} & 1.27 \,10^{-6} \\ 
 \hline
 7\frac 2{15}&(64;24^{\times 7},3^{\times 7},1^{\times 2})&(14,22)&(121,7)&7.1333&7.1334
      &\tfrac{641}{240} & 3.25 \,10^{-6}\\
 \hline
 7\frac 17&(24;9^{\times 7},2,1^{\times 6})&(7,8)&(57,1)&7.1428&7.1429
      &\tfrac{449}{168} & 6.63 \,10^{-6}\\
 \hline
 7\frac 2{13}&(40;15^{\times 7},2^{\times6},1^{\times 2})&(14,13)&(107,0)&7.151&7.156
      &\tfrac{107}{40} & 332.5 \,10^{-6}\\
 \hline
 7\frac 15&(16;6^{\times 7},1^{\times 5})&(7,5)&(43,0)&7.1665&7.22
      &\tfrac{43}{16} & 4218.4 \,10^{-6}\\
 \hline
 7\frac 14&(35;13^{\times 7},4,3^{\times 3})&(0,13)&(87,1)&7.2485&7.252
     &\tfrac{377}{140} & 274.7 \,10^{-6}\\
 \hline
 7\frac 12&(8;3^{\times 7},1^{\times 2})&(7,2)&(22,0)&7.328&7.56
       &\tfrac{11}4 & 11387.2 \,10^{-6}\\
 \hline
 8&(6;3,2^{\times 7})&(1,2)&(17,0)&7.97&8.03
      &\tfrac{17}6 & 4906.2 \,10^{-6}\\
 \hline
\end{array}
\end{equation}

\NI 
\begin{rmk}
\rm  
(i) 
The above table gives just enough decimal places of the (irrational) numbers $u_x, v_x$ to describe their important features. For example
$u_{\frac 12} = \frac 12 \left( 9+4 \sqrt 2 \right) \approx 7.328 < 7\frac 13$.

\MS
\NI (ii) 
In the above table there is one constraint centered at each point of the form 
$7\frac 1k$ for $2\le k\le 8$, except for $k=3$ and $k=6$. 
In fact, there are classes $(d;\mm)$ giving constraints centered at 
$7 \frac 16$ and $7 \frac 13$,
namely 
$$
\left( 96; 36^{\times 6}, 35, 6^{\times 6} \right) \; \text{ at }\, 7\tfrac 16
\qquad \text{ and } \qquad
\left( 24; 9^{\times 6}, 8, 3^{\times 3} \right) \; \text{ at }\, 7\tfrac 13 .
$$
But these $(d;\mm)$ have the property that $\mu(d;\mm)(a)= c(a)$
only at their center points.
(See the proof of Theorem~\ref{thm:78} at the end of this section for details).
\MS

\NI(iii)
The four steps
at the points $7 \frac 18$, $7 \frac 2{15}$, $7 \frac 17$, 
$7 \frac 14$ are the only ones in the graph of $c(a)$ that 
are not flat to the right.
\diam
\end{rmk}

To prove Theorem~\ref{thm:78} we will proceed as follows.
Assume that $(d;\mm) \in \Ee$ is a class with $\ell(a)=\ell(\mm)$ and $\mu(d;\mm)(a) > \sqrt{a}$  
for some $a \in [7 \frac 19,8]$.
We first assume that $a=7\frac 1k$ for some $k \in \{1,\dots,8\}$, 
and find all such classes $(d;\mm)$.
We then assume that $a \in \;]7\frac{1}{k+1}, 7\frac 1k[$, 
and prove an upper bound $D(z_k)$ for $d$ if $a = z_k := 7\frac 2{2k+1}$
and an upper bound $D_k$ if $a \neq z_k$.
In both cases, we also show that $m_1=\dots=m_7$.
We then use a simple computer program to find all classes $(d;\mm)$ as above 
at $z_k$ with $d \le D(z_k)$.
Finally, we use another computer program to find all classes $(d;\mm)$ as above at some $a \neq z_k$ with $d \le D_k$.

We start by looking at the boundary points $7 \frac 1k$ 
of our subintervals $[7\frac{1}{k+1}, 7\frac 1k]$.
\begin{lemma} \labell{le:7k} 
The classes $(d;\mm) \in \Ee$ such that 
$\ell(7 \frac 1k) = \ell(\mm)$ and 
$\mu(d;\mm)(7 \frac 1k) > \sqrt{7 \frac 1k}$ are
\begin{equation} \labell{t:le:7k}
\begin{array}{|c|l|c|l|}\hline
k & (d;\mm)&k &(d;\mm)\\ \hline
8&
 \left(  48;  18^{\times 7},   3,  2^{\times 7} \right)  
           &8 &\left( 384; 144^{\times 6}, 143, 18^{\times 8} \right)\\ \hline
 7&          
 \left(  24;   9^{\times 7},   2,  1^{\times 6} \right)  
  & 7& \left( 168;  63^{\times 6},  62,  9^{\times 7} \right) \\ \hline
  6& 
 \left(  96;  36^{\times 6},  35,  6^{\times 6} \right)&5&
\left(  16;   6^{\times 7},  1^{\times 5} \right) \\
\hline
4& \left(  35;  13^{\times 7},   4, 3^{\times 3} \right)
&3&\left(  24;   9^{\times 6},   8, 3^{\times 3} \right)\\ \hline
2&
\left(   8;   3^{\times 7},   1^{\times 2} \right)&1& \left(   6;     3,   2^{\times 7} \right) \\\hline
\end{array}
\end{equation}
\end{lemma}

\begin{proof}  

We first look at the case $a = 7 \frac 11 = 8$.
Then $\ell(\mm) = \ell(8) = 8$.
By Lemma~\ref{le:atmost1} we need to consider 3~cases, namely
$\mm = (M^{\times 8})$,
$\mm = (M+1,M^{\times 7})$,
$\mm = (M^{\times 7},M-1)$.
Consider the case $\mm = (M^{\times 8})$.
From the Diophantine equations
\begin{equation*} 
\left\{
\begin{array} {rcl}
 3d  &=& 8M+1 \\
 d^2 &=& 8M^2-1 \\
 \end{array}\right.
\end{equation*}
we obtain $(8M+1)^2 = 9 \left( 8M^2-1\right)$, 
i.e.~$4M^2-8M-5 = 0$.
This equation has no solution in~$\N$.
In the case $\mm = (M+1,M^{\times 7})$, the Diophantine equations
give 
$$
(8M+1+1)^2 \,=\, 9 \left( 8M^2+2M+1-1\right)
$$
whose only solution in~$\N$ is $M=2$, giving the solution
$(d;\mm) = (6;3, 2^{\times 7})$.
In the case~$(M^{\times 7},M-1)$, 
the Diophantine equations
give $(8M-1+1)^2 = 9 \left( 8M^2-2M+1-1\right)$, which has no solution in~$\N$.

Assume now that $k \in \{2, \dots, 8\}$.
In view of Lemma~\ref{le:atmost1}, 
there are five possibilities for $\mm$, namely
\begin{gather*}
(M^{\times 7},m^{\times k}),\;\;  
(M+1,M^{\times 6},m^{\times k}), \;\;
 (M^{\times 6},M-1,m^{\times k}), \\
(M^{\times 7},m+1,m^{\times (k-1)}),\;\;
 (M^{\times 7},m^{\times (k-1)},m-1).
\end{gather*}
Since $\ell (\mm) = \ell (7 \frac 1k) =7+k$,
in the first four cases we can assume that $m \ge 1$ 
and in the last case we can assume that $m-1 \ge 1$.
We define $\eps_M$ and $\eps_m$ by
$$
M = \tfrac{d}{\sqrt{7 \frac 1k}}+\eps_M, 
\quad
m = \tfrac{d}{k \sqrt{7 \frac 1k}}+\eps_m .
$$  

\MS
\ni
{\bf Case 1. $\mm = (M^{\times 7},m^{\times k})$.}
Then $|M-km| = |\eps_M-k\eps_m| \le |\eps_M|+k|\eps_m|$.
Since $|\eps_M|^2 +k|\eps_m|^2<1$, we find $|\eps_M|+k|\eps_m| < \sqrt{k+1}$, and so
$|M-km| \le \lceil \sqrt{k+1}-1 \rceil \in \{0,1,2\}$.
Set
\begin{equation*} 
s \,=\, M-km \,\in\, 
\left\{
\begin{array} {ll}
 \{0, \pm 1 \}       &\mbox{ if }\;  k \in \{2,3\}, \\
 \{0, \pm 1, \pm 2\} &\mbox{ if }\;  k \in \{4, \dots, 8\}. 
 \end{array}\right.
\end{equation*}
From the Diophantine equations
\begin{equation*} 
\left\{
\begin{array} {rcl}
 3d  &=& 7M+km+1 \\
 d^2 &=& 7M^2+km^2-1 \\
 \end{array}\right.
\end{equation*}
we obtain $(7M+km+1)^2=9 \left( 7M^2+km^2-1\right)$.
Since $M=km+s$, this becomes
\begin{equation*} 
10+km \left( 16-9m+km\right) +14s \left( 1-km-s\right) \,=\, 0 . 
\end{equation*}
If $s=1$, this is
$$
10 + k m \left( 2- 9 m + k m\right) \,=\, 0,
$$
which has solutions in $\N$ only if $k=5$ or $2$, 
namely $m=1$, giving 
$$
\left( 16; 6^{\times 7}, 1^{\times 5}\right) \text{ at } 7 \tfrac 15, \qquad
\left( 8; 3^{\times 7}, 1^{\times 2}\right) \text{ at } 7 \tfrac 12.
$$
No other allowed values for $s$ and $k$ yield integer solutions~$m$.

\MS
\ni
{\bf Case 2. $\mm = (M+1,M^{\times 6},m^{\times k})$.}
Then $\si = k |\eps_m|^2 \le \frac 17$.
Therefore,
$|M-km| \le |\eps_M|+k|\eps_m| \le \frac 1{\sqrt 6} + \sqrt{\frac k7}$,
and so
\begin{equation} \label{e:sM} 
s \,:=\, M-km \,\in\,
\left\{
\begin{array} {ll}
  \{0\}              &\mbox{ if }\;  k = 2, \\
  \{0, \pm 1 \} &\mbox{ if }\;  k \in \{3, \dots, 8\}. 
 \end{array}\right.
\end{equation}
In this case, the Diophantine equations translate to
$$
(7M+km+2)^2 \,=\, 9 \left( 7M^2+2M+1+km^2-1\right) .
$$
With $M=km+s$ this becomes
\begin{equation} \label{e:kmsM+1}
-4 
-k m \left( 14 -9m +km\right)
+2s \left( -5 +7km +7 s\right)
\,=\, 0 . 
\end{equation}
If $s=0$, this becomes
$$
-4 
-k m \left( 14 -9m +km\right)
\,=\, 0 
$$
which has no solution in $\N$ for $k \in \{ 2, \dots, 8\}$.
For $s = \pm 1$ and $k \in \{3, \dots, 8\}$ equation~\eqref{e:kmsM+1} has no solution in~$\N$.

\MS
\ni
{\bf Case 3. $\mm = (M^{\times 6},M-1,m^{\times k})$.}
As in Case~2 we have~\eqref{e:sM}.
In this case, the Diophantine equations translate to
$$
(7M+km)^2 \,=\, 9 \left( 7M^2-2M+1+km^2-1\right).
$$
With $M=km+s$ this becomes
\begin{equation} \labell{e:kmsM-1}
-k m \left( 18 -9m +km\right)
+2s \left( -9 +7km +7 s\right)
\,=\, 0 . 
\end{equation}
If $s=0$, this becomes
\begin{equation*} 
18 -9m +km 
\,=\, 0 
\end{equation*}
which has a solution in $\N$ for four $k$, namely $k=8,7,6$, and $3$. 
This gives the first four of the five entries in the table with $m_1\ne m_7$.

If $s = 1$, equation~\eqref{e:kmsM-1} becomes
$$
-4 -km \left( 4 -9 m +km \right) \,=\, 0, 
$$
which has a solution in~$\N$ only for $k=4$.
We get the solution $(13; 5^{\times 6}, 4,1^{\times 4})$,
which is, however, not 
obstructive, since it gives
$\mu (d;\mm)(7 \frac 14) = \frac{35}{13} < \sqrt{7 \frac 14}$.

If $s = -1$, equation~\eqref{e:kmsM-1} becomes
$$
32 -km \left( 32 -9 m +km\right) 
\,=\, 0 . 
$$
It has a solution in~$\N$ only for $k=2$,
and gives $(19; 7^{\times 6}, 6,4^{\times 2})$. 
But again this class is not 
obstructive, since 
$\mu (d;\mm)(7 \frac 12) = \frac{52}{19} < \sqrt{7 \frac 12}$.

\MS
\ni
{\bf Case 4. $\mm = (M^{\times 7},m+1,m^{\times (k-1)})$.}
Note that for $\eps \in \RR$ and $k \in \N$ with $(k-1)\eps^2 +(\eps+1)^2 \le 1$
we have $\eps \in \left[ -\frac 2k, 0\right]$ and hence
$$
\left| (k-1)\eps + (\eps+1) \right| \,=\, |k\eps+1| \,\le\,  1 .
$$
Using this and $\si \ge \frac{k-1}{k}$ we estimate
\begin{eqnarray*}
\left| M-km-1\right| \,=\, \left|M-(m+1)-(k-1)m\right|
&=&   \left| \eps_M - (\eps_m+1) - (k-1)\eps_m \right| \\
&\le& |\eps_M| + \left|(k-1)\eps_m+\eps_m+1\right|     \\
&\le& \sqrt{\tfrac{1}{7k}} +1 \;<\;2.
\end{eqnarray*}
Therefore,
$$
M-1 \,=\, km+s \quad \text{ with }\, s \in \{0, \pm 1\} .
$$
In this case,
the Diophantine equations translate to
$$
(7M+km+1+1)^2 \,=\, 9 \left( 7M^2+km^2+2m+1-1\right) .
$$
With $M=km+1+s$ this becomes
\begin{equation} \label{e:kmsm+1}
-18 + 18 m 
- k m \left( 18 - 9 m +k m\right)
+ 14s  \left( k m + s \right)
\,=\, 0 . 
\end{equation}
If $s=1$, this is
$$
-4 + 18 m 
- k m \left(  4 - 9 m + k m \right)
\,=\, 0, 
$$
which has a solution in $\N$ only when $k=8, m=2$ and $k=7,m=1$, 
giving us two more entries in our table.
If $s=0$, equation~\eqref{e:kmsm+1} becomes
$$
-18 + 18 m - k m \left( 18 - 9 m + k m \right) \,=\, 0, 
$$
which has a solution in $\N$ only for $k=4, m=3$. This gives the entry in the table at $k=4$.
If $s=-1$, equation~\eqref{e:kmsm+1} has no solution in~$\N$ for $k \in \{2, \dots, 8\}$.

\MS
\ni
{\bf Case~5. $\mm = (M^{\times 7},m^{\times (k-1)},m-1)$.}
As in Case~4 we find
$$
M+1 \,=\, km+s \quad \text{ with }\, s \in \{0, \pm 1\} .
$$
In this case,
the Diophantine equations translate to
$$
(7M+km-1+1)^2 \,=\, 9 \left( 7M^2+km^2-2m+1-1\right) .
$$
With $M=km-1+s$ this becomes
\begin{equation} \label{e:kmsm-1}
-14 + 18 m 
+ k m \left( 14 - 9 m +k m \right)
+ 14 s  \left( 2 - k m - s \right)
\,=\, 0 .  
\end{equation}
If $s=1$, this becomes
$$
18 - 9 k m + k^2 m
\,=\, 0 . 
$$
It has a solution in~$\N$ only for $k=6$ and $k=3$, namely $m=1$.
Since we assumed that $m-1 \ge 1$, the corresponding classes
$(d;\mm)$ are not relevant.
If $s=0$ or if $s=-1$, equation~\eqref{e:kmsm-1} 
has no solution in~$\N$ for $k \in \{2, \dots, 8\}$.

\m
The above calculations show that the elements listed in Table~\ref{t:le:7k}  
are the only 
obstructive 
solutions to the Diophantine equations.   
One readily checks that these elements all reduce to~$(0;-1)$ under standard Cremona moves,
and therefore belong to~$\Ee$.
\end{proof}

\begin{rmk} \label{rmk:computer1}
\rm
Proceeding as in the proof of Lemma~\ref{le:7k},
one can find all classes $(d;\mm) \in \Ee$ with 
$\mu(d;\mm)(z_k) > \sqrt{z_k}$ and $\ell(z_k)=\ell(\mm)$
at the points $z_k := 7 \frac{2}{2k+1}$, $k \in \{1, \dots, 8\}$,
namely
$$
\left(  64;  24^{\times 7},  3^{\times 7},  1^{\times 2} \right)  \text{ at }\, 7 \tfrac 2{15} 
\quad \mbox{ and } \quad
\left(  40;  15^{\times 7},  2^{\times 6},  1^{\times 2} \right)  \text{ at }\, 7 \tfrac 2{13} .
$$
For convenience, we will find these classes by a different method,
that involves the first of the two computer programs of Appendix~\ref{app:comp}.
\diam
\end{rmk}

We next derive upper bounds for $d$ if 
$a \in \;]7 \frac{1}{k+1}, 7 \frac 1k[$.
There are various ways to do this.
We will give arguments that give rather low upper bounds,
so that our method of finding $c(a)$ depends as little as possible on computer computations
(compare Remark~\ref{rem:method} below). 
Note that $a \in \;]7 \frac{1}{k+1}, 7 \frac 1k[$ has $N+1$ blocks with $N \ge 2$, 
and that $L := \sum_{i\ge 1}\ell_i\ge 2+k$ with equality 
exactly if $a = [7;k,2] = 7\frac 2{2k+1} = z_k$.
  
\begin{lemma}\labell{le:7less} 
Suppose that $(d;\mm) \in \Ee$ is such that
$\mu(d;\mm)(a)=c(a) > \sqrt{a}$ for some $a$ with $\ell(a)=\ell(\mm)$.  
Suppose also that~$a$ has $N+1$ blocks for some $N \ge 2$
and that $a \in \;]7\frac 1{k+1}, 7\frac 1{k}[$ where $8\ge k \ge 1$.  
Then $m_1=m_7$.  
Further
\begin{itemize} \item[(i)]
When $L=2+k$, the following table gives the
maximum possible values~$D(z_k)$ of~$d$ for the different~$k$.  
\begin{equation}\labell{t:2}
\begin{array}{|c||c|c|c|c|c|c|c|c|c|}\hline
k      &   8&   7&   6&   5&   4&   3&   2&   1 \\ \hline
D(z_k) & 104 & 98 & 92 & 86 & 79 & 73 & 69 & 75 \\
\hline
\end{array}
\end{equation}

\item[(ii)] 
When $L>2+k$, the following table gives the
maximum possible values~$D_k$ of~$d$ for the different~$k$.  
\begin{equation}\labell{t:3}
\begin{array}{|c||c|c|c|c|c|c|c|c|c|} \hline
k &  8 &  7 &  6 &  5 &   4 &  3 &  2 &  1 \\ \hline
D_k & 88 & 81 & 74 & 67 &  61 & 56 & 64 & 66 \\
\hline
\end{array}
\end{equation}
\end{itemize}
\end{lemma}

\begin{rmk}
\rm 
By Lemma \ref{le:7k} one cannot conclude $m_1 = m_7$ 
without the assumption $a \neq 7 \frac 1k$.
\diam
\end{rmk}
         
\begin{proof}    
The proof of this lemma is based on an analysis of the equation~\eqref{eq:obsiv} 
using the estimates for $|\sum \eps_i|$ obtained in Lemma~\ref{le:7eps}. 
Recall from (iv) of that lemma that for $a = 7 \frac pq$ and with $v_M := \frac d{q\sqrt a}$ 
we have the estimates
\begin{equation} \label{ine:dsigma}
d \,\le\, \frac{\sqrt{a}}{\delta} \left( \sqrt{\si L}-1 \right) 
\,\le\, \frac{\sqrt{a}}{\delta} \left( \sqrt{\si q}-1 \right)
\,<\, \frac{\sqrt{a}}{\delta} \left( \frac{\si}{v_M\delta}-1 \right) 
\end{equation}
whenever $\delta := y(a) - \tfrac 1q > 0$. 

(i)
The only number in $]7\frac 1{k+1}, 7\frac 1{k}[$ with $L = k+2$
is $z_k := [7; k, 2] = 7 \frac 2{2k+1}$. 
Note that 
$$
y(z_k)-\tfrac 1q \,=\, 8 + \tfrac{1}{2k+1} - 3 \sqrt{7 + \tfrac 2{2k+1}} \,>\, 0
$$
for all $k$.
By~\eqref{ine:dsigma} we therefore have
$$
d \,\le\, \frac{\left(\sqrt{\si(k+2)}-1\right) \sqrt{z_k}}{y(z_k)- \frac{1}{2k+1}}
\,=\, \frac{\left(\sqrt{\si(k+2)}-1\right) \sqrt{7+\frac{2}{2k+1}}}{8+\frac{1}{2k+1}-3\sqrt{7+\frac{2}{2k+1}}}
$$
With $\si \le 1$
this yields Table~\ref{t:2}.
If $m_1 \ne m_7$ we may take $\si \le \frac 17$.  
The largest value of~$d$ is then $\le 6$ when $k\le 7$ and $\le 9$ when $k=8$.     
But there are clearly no suitable $(d;\mm)$ with such small $d$.  
Therefore this case does not occur. This proves~$m_1=m_7$ for $L=2+k$.
 
\medskip
We will prove (ii) and the claim that $m_1=m_7$ together.
We will give separate arguments for the three cases $k=1$, $k=2$ and $k \in \{3, \dots, 9\}$.
Denote $a_k = 7\frac 1{k+1}$ for some $1\le k\le 8$.  
We have the table (rounded down to $3$ decimal places)
$$
\begin{array}{|r||c|c|c|c|c|c|c|c|}\hline
k=         & 8             &7     & 6     &     5&     4&     3& 2    & 1    \\  \hline
y(a_k) \ge &\frac 19=0.111 & 0.117& 0.125 & 0.135& 0.150& 0.172& 0.209& 0.284\\
\hline
\end{array}
$$

\medskip
\ni
{\bf The case $k=1$.}
Assume that $(d;\mm) \in \Ee$ is a class with $\mu(d;\mm)(a) > \sqrt a$ and $\ell(a)=\ell(\mm)$ 
for some $a \in \;]7 \frac 12, 8[$ other than $7 \frac 23$.
We first prove that $m_1=m_7$.
If not, then $\si \le \frac 17$.
Therefore, $v_M = \frac{d}{q\sqrt a} \ge 1-\frac{1}{\sqrt{14}} > 0.73$,
since otherwise $\si \ge \eps_M^2+\eps_{M-1}^2 > 2 \frac{1}{14} = \frac 17$.
Also,
$q \ge L \ge 3+k = 4$, 
and so $y(a) - \frac 1q \ge y(7\frac 12)-\frac 14 \ge 0.28-\frac 14 >0$.
We can therefore apply~\eqref{ine:dsigma}:
$$
\sqrt{\tfrac q 7} \,\ge\, \sqrt{\si L} \,\ge\, 1+\tfrac{d}{\sqrt a} \left(y(a)-\tfrac 1q\right) \,>\, 1,
$$
showing that $q \ge 8$. 
Therefore, $y(a)-\frac 1q > y(7\frac 12)-\frac 18 > 0.28 -\frac 18 > \frac 17$.
Using again~\eqref{ine:dsigma} we finally find
$$
5 \sqrt a \,<\, 8 \cdot 0.73 \sqrt a \,\le\, 0.73 q \sqrt a \,<\, d \,<\, 
\frac{\sqrt a}{\frac 17} \left( \frac{\si}{v_M \frac 17}-1 \right) \,<\, 
7\sqrt a \left( \frac{1}{0.73}-1 \right) \,<\, 3 \sqrt a ,
$$
a contradiction.

We now prove that $d \le 66$.
For~$a$ as above, both numbers
\begin{eqnarray*}
f(a,q) &:=& \frac{\sqrt a}{a+1-3\sqrt a-\frac 1q}\bigl(\sqrt{q}-1 \bigr), \\
g(a,q) &:=& \frac{\sqrt a}{a+1-3\sqrt a-\frac 1q}\left(\frac{2}{a+1-3\sqrt{a}-\frac 1q}-1 \right),
\end{eqnarray*}
%
% by replacing, in g(a,q), 2 by 14/9 (as we did for [8,9]), we could bring down 66 to 54 
are positive.
Moreover, by~\eqref{ine:dsigma} we have $d \le f(a,q)$, 
and using also Lemma~\ref{le:7eps}~(iv) we see that $d < g(a,q)$.
We saw above that $q\ge 4$.
We first use the function~$f$ to see that for $q \in \{ 4,5,6,7, 8\}$ 
we have $d \le 26$.
Assume now that $q \ge 9$.
We then view $a$ and $q$ as independent variables of the functions~$f$ and $g$.
Both $f(a,q)$ and $g(a,q)$ are decreasing functions of $a$.
With $a_1 = 7 \frac 12$ we therefore have
$$
d \,\le\, \max_{q \ge 9} \min \left\{ f(a_1, q),\, g (a_1, q) \right\}.
$$
One readily checks that $f(a_1,q)$ is increasing on $\{ q \ge 9\}$ and that
$g(a_1,q)$ is decreasing in $q$.
Since 
$d \le f(a_1,56) < 67$ if $q \le 56$ and 
$d \le g(a_1,57) < 67$ if $q \ge 57$
we conclude that $d \le 66$, as claimed.
\diam  

\begin{remark} \label{rem:method}
{\rm
This method for estimating $d$ can be used for all~$k \le 8$.
However, the estimates get worse, e.g.~for $k=8$ 
(with the factor $\sqrt q-1$ of $f$ replaced by $\sqrt{8+\frac q8}-1$, 
see~\eqref{eq:ineqqq} below)
one finds $d \le 410$.
One could also omit checking that 
obstructive classes have $m_1=m_7$, and use a variant of 
our computer code {\tt SolLess} from Appendix~\ref{a:SolLess}
that does not use $m_1 = m_7$.
\diam
}
\end{remark}

\medskip
\ni
{\bf The case $k=2$.}
The class $(8;3^{\times 7}, 1^{\times 2})$ gives the constraint 
$c(a) \ge \mu_0(a) = \frac{7+2a}{8} > \sqrt a$ on $[7 \frac 13, 7 \frac 12]$.
Assume that $(d;\mm) \in \Ee$ is a class with $\mu (d;\mm)(a) = c(a) \ge \frac{7+2a}{8}$
for some $a \in [7 \frac 13, 7 \frac 12]$.
Proposition~\ref{prop:obs}~(i)  implies that
$$
\frac{7+2a}{8}\le \mu(d;\mm)(a) \le \sqrt a \sqrt{1+1/d^2}.
$$
When $a = 7\frac 13$ this gives the estimate $d\le 64$.
Since $\frac{7+2a}{\sqrt a}$ decreases on $[7 \frac 13, 7 \frac 12]$, 
we find $d\le 64$ everywhere.
We will check $m_1=m_7$ for $k=2$
and $L \ge 3+k = 5$
at the same time as for $k \ge 3$.
\diam

\medskip
\ni
{\bf The case $k \in \{3, \dots, 9\}$.}
Suppose that $a \in \;]7\frac 1{k+1}, 7\frac 1{k}[$ for some $k\ge 2$,
and that $L \ge 3+k$.
Then we may write 
$$
a = [7;k,\ell_2,\dots, \ell_N] = 7 + \tfrac 1{k + \tfrac {p'}{q'}}
$$
where
$\frac{p'}{q'} := a' := [0;\ell_2,\dots, \ell_N]$. 
Thus $q' = q_{N-1}(a') \ge \sum_{j\ge2} \ell_j =: L'$
by Sublemma~\ref{sl:L}, and so
$$
L \,:=\, \sum_{j\ge 1} \ell_j \,=\, k+L' \,\le\, k+q' .
$$
Since $q = kq'+p'$ we find $L\le k + \frac qk$.
Moreover, 
$q' \ge L' = L-k \ge 3$,
and so $q \ge 3k+1$. 
Therefore, for $a \in \:]a_k, a_{k-1}[$ we have $y(a) \ge y(a_k) > \frac{1}{3k+1} > \frac 1q$.
Thus the inequality~\eqref{ine:dsigma} implies that
\begin{equation} \labell{eq:ineqqq}
\bigl( 1-\tfrac d{q\sqrt a} \bigr) + \tfrac d{\sqrt a} y(a) \,=\, 
1 + \tfrac d{\sqrt a} \left( y(a) - \tfrac 1q \right) \,\le\, \sqrt{\si(k + \tfrac qk)} .
\end{equation}

\MS
\NI {\bf Case 1:}  $\frac 12 \le v_M: = \tfrac d{q\sqrt a}\le \tfrac 34$.

Because $y(a) \ge y(a_k)$ for all $a \in [a_k,a_{k-1}]$ and $y(a_8)=\frac 19$, we must have
\begin{equation*}
\tfrac{q}{18} \,\le\,
\tfrac 14 + \tfrac q2 y(a_k) \,\le\, \tfrac 14 + \tfrac{d}{\sqrt{a}} y(a_k) \,\le\,
\left( 1- \tfrac{d}{q\sqrt a} \right) + \tfrac{d}{\sqrt a} y(a) \,\le\, 
\sqrt{\si'(k + \tfrac qk)} \,,
\end{equation*}
where $\si' \le \frac 78$ is as in Lemma~\ref{le:7eps}~(iii).

Note that the squared error of the last two $\eps_i$ is at least
$2 \bigl( \frac14 \bigr)^2 = \frac 18$.  
Therefore, if also $m_1 \ne m_7$, we have $\si' < \frac 17-\frac 18 = \frac 1{56}$.  
But, for each $k\in [2,8]$, the inequality 
$$
\tfrac q{18} \le  
\sqrt{\tfrac1{56}(k + \tfrac qk)}
$$
holds only if $q^2 \le 6(k+\frac qk)$.
Since this quadratic inequality holds for $q=0$ and does 
not hold when $q=3k+1$, it does not hold for any $q\ge 3k+1$.
Therefore,  for each $k$ we have $m_1=m_7$,   
and $\si' \le \frac 78$.

Now suppose that $k=8$, and consider the inequality
$$
\tfrac 14 + \tfrac q2 y(a_8) \,=\,\tfrac 14 + \tfrac q{18} \le\sqrt{\tfrac 78 (8 + \tfrac q8)}.
$$
This holds when $q=0$ but does not hold for $q \ge 63$.   
Thus $q \le 62$ so that $\tfrac d{\sqrt a} \frac 19 \le \sqrt {\frac 78 \left( 8+\frac{62}{8}\right)} - \frac 14$.
Since $a \le 7 \frac 18$ we get $d \le 83$. 
The same argument works for the other~$k$, 
and we obtain the following upper bounds for~$q$ and then for~$d$. 
$$
\begin{array}{|c||c|c|c|c|c|c|c|}\hline
k=&8&7&6&5&4&3\\ 
\hline
q\le &  62& 58 & 53 &  49 & 45  & 41  \\ \hline
d\le &  83& 77 & 71 &  66 & 61  & 56 \\\hline 
\end{array}
$$

\MS
\NI {\bf Case 2:}  $v_M \le \tfrac 12$.

Since the squared error $\ell_N \de_N^2$ on the last block is now
at least $\frac 12$, we must have $m_1=m_7$ and $\si'<\frac 12$.  
Further, by Lemma~\ref{le:7eps}~(iii), $\tfrac d{q\sqrt a} \ge \frac 13$. 
Therefore \eqref{eq:ineqqq} gives
$$
\tfrac 12 + \tfrac q3 y(a_k) \,\le\, \tfrac 12 + \tfrac{d}{\sqrt{a}} y(a_k) \,\le\, 
\sqrt{\tfrac 12 (k + \tfrac qk)} \,.
$$
This gives the following upper bounds for $q$ and $d$.
$$
\begin{array}{|c||c|c|c|c|c|c|c|}\hline
k=&8&7&6&5&4&3\\ 
\hline
q\le & 62 & 57 & 53 & 49 & 45 & 42 \\ \hline
d\le & 55 & 51 & 47 & 43 & 40 & 37 \\\hline 
\end{array}
$$

\MS
\NI {\bf Case 3:}  $\tfrac 34\le v_M \le 1$.

Now~\eqref{eq:ineqqq} gives
$$
\tfrac {3}{4}\,q\, y(a_k) \,\le\, \tfrac d{\sqrt a}\,y(a_k) \,\le\, 
\sqrt{\si(k + \tfrac qk)} \,.
$$
If $\si \le \frac 17$, this is not satisfied when $q \ge 3k+1$ for any $k \in \{2, \dots, 8\}$.  
Thus $m_1=m_7$.

Further, taking $\si=1$ we obtain the following upper bounds for $q$ and $d$.
$$
\begin{array}{|c||c|c|c|c|c|c|c|}\hline
k=&8&7&6&5&4&3\\ 
\hline
q\le & 44 & 40 & 37 & 33 & 30 & 26 \\ \hline
d\le & 88 & 81 & 74 & 67 & 60 & 53 \\ \hline 
\end{array}
$$

\MS
\NI {\bf Case 4:} {\it $1\le v_M$.}%\tfrac d{q\sqrt a}$.}
 
In this case, \eqref{eq:ineqqq}
gives
$$
q\, y(a_k) \,\le\, 1+q \left( y(a_k)-\tfrac 1q \right) \,\le\,
1+ \tfrac{d}{\sqrt a} \bigl( y(a_k) - \tfrac 1q \bigr) \,\le\,
1+ \tfrac{d}{\sqrt a} \bigl( y(a) - \tfrac 1q \bigr)
\,\le\, \sqrt{\si(k + \tfrac qk)} \,.
$$
We have already seen in Case~3 that $q\, y(a_k) \le \sqrt{\si(k + \tfrac qk)}$ is impossible
for $\si \le \frac 17$.
Thus $m_1=m_7$.

Further, taking $\si=1$ we obtain the following upper bounds for $q$ and $d$.
$$
\begin{array}{|c||c|c|c|c|c|c|c|}\hline
k=&8&7&6&5&4&3\\ 
\hline
q\le & 31 & 28 & 25 & 22 & 19 & 17 \\ \hline
d\le & 82 & 75 & 68 & 61 & 54 & 46 \\ \hline 
\end{array}
$$
Taking for each $k$ the worst upper bound for $d$ in the different cases,
we obtain Table~\ref{t:3}.
This completes the proof of Lemma~\ref{le:7less}.
\end{proof}

\begin{corollary} \labell{c:zk} \begin{itemize}
\item[(i)]
The only classes $(d;\mm) \in \Ee$ such that 
$\ell(z_k) = \ell(\mm)$ 
and such that $\mu(d;\mm)(z_k) > \sqrt{z_k}$ are
\begin{eqnarray*}
7 \tfrac 2{15} \colon \left(  64;  24^{\times 7},  3^{\times 7},  1^{\times 2} \right)  
&\;\;\mbox{ and }\;\;& 
7 \tfrac 2{13} \colon \left(  40;  15^{\times 7},  2^{\times 6},  1^{\times 2} \right)  
\end{eqnarray*}
\item[(ii)]
There are no classes $(d;\mm) \in \Ee$ such that 
$\ell(a) = \ell(\mm)$ and
$\mu(d;\mm)(a) > \sqrt{a}$ 
for some $a \in \;]7 \frac 19,8[$
not of the form $7 \frac 1k$ or $7 \frac 2{2k+1}$.
\end{itemize}
\end{corollary}

\begin{proof}  
(i)
The computer code {\tt SolLess[a,D]} given in Appendix~\ref{a:SolLess}
finds for a rational number~$a$ and a natural number~$D$ all classes 
$(d;\mm) \in \Ee$
with $\ell(\mm) = \ell(a)$ and $\mu(d;\mm)(a) > \sqrt{a}$ and $d\le D$.
For $k \in \{1,\dots, 8\}$ we choose $D = D(z_k)$ as given by Table~\eqref{t:2}.
The code {\tt SolLess[a,D]} with $D=D(z_k)$ and $a=z_k$
tells us that for $k=7$ and $k=6$, the only such classes
are the ones given in the corollary, while for the other~$k$ there are no such classes.
Finally, one checks that the two classes in~(i) reduce to~$(0;-1)$ under standard Cremona moves,
and hence belong to~$\Ee$.

\MS
(ii)
The computer code {\tt InterSolLess[k,D]} given in Appendix~\ref{a:comp.inter}
provides for a natural number~$D$ a finite list of candidate classes $(d;\mm) \in \Ee$
with 
$\ell (\mm) = \ell(a)$ and 
$\mu(d;\mm)(a) > \sqrt{a}$ and $d \le D$
for some $a \in \;]7\frac 1{k+1}, 7 \frac 1k[$.
For $k \in \{1,\dots, 8\}$ we choose $D = D_k$ as given by Table~\eqref{t:3}.
The code {\tt InterSolLess[k,D]} with $D=D_k$ 
tells us that for $k \neq 4$ there are no candidate classes, 
while for $k=4$ the only candidate class is
$(d;\mm) = \left(59; (22^7, 5^3, 4, 1^3) \right)$.
Since the length of the second block is~4
and the length of last block is~$\ge 2$,
the~$a$ in question must be $[7;4,3]$ or $[7;4,1,2]$.
The second possibility is excluded by Lemma~\ref{le:irrel}~(i)
applied to the third block.
Moreover, at $a= [7;4,3] = 7 \frac 3{13}$ we have 
$\mu (d;\mm) (a) = \frac{2062}{767} < \sqrt{a}$, which excludes also the first possibility.
\end{proof}

\NI {\bf Proof of Theorem~\ref{thm:78}.}
Recall from Proposition~\ref{prop:7easy} that $c(7 \frac 19) = \sqrt{7 \frac 19}$.
Moreover, by Lemma~\ref{le:I} any class $(d;\mm) \in \Ee$ with $\mu(d;\mm)(8) > \sqrt 8$ 
must lie in~$\Ee_8$.
By looking at the list of elements in $\Ee_8$ given in Lemma~\ref{le:eekfin} one checks 
that the only such class is $\left( 6; 3,2^{\times 7} \right)$.
By using Lemma~\ref{le:I} once more,
we conclude that all constraints on $[7 \frac 19, 8]$ come from 
the ten classes of Lemma~\ref{le:7k} 
and the two classes from Corollary~\ref{c:zk}.

In the paragraph just before Theorem \ref{thm:78} we worked out
the constraint $\mu(d;\mm)$ given by the class centered at $7 \frac18$.
Similar computations show
that all the eight classes in Table~\ref{t:1} 
behave as described there.
In order to prove Theorem~\ref{thm:78}, it therefore remains to check
that the four classes from Lemma~\ref{le:7k} that do not appear in 
Theorem~\ref{thm:78} give no further constraints.
However, one can calculate the corresponding  functions $\mu(d;\mm)$ just as before, obtaining the following data.\footnote
{We also calculated the number $N(A,B)$ of integer points in the triangle $T^a_{A,B}$ and the number~$s$ of integer points on its slant edge because of their relevance to  
Remark~\ref{rmk:hid}.}
\begin{equation}\labell{t:5}
 \begin{array}{|c|l|l|r|c|r|r|}
 \hline
 a&(d;\mm)& (A,B)& (A',B')&\mu(a)&N(A,B)&s \\
 \hline
 7\frac 18 & (384;144^{\times 6},143,18^{\times 8})&(-1,144)&(1025,0) 
      &\tfrac {1025}{384}&74322&18\\ 
 \hline
 7\frac 1{7}&(168;63^{\times 6},62, 9^{\times 7})&(-1,63)&(449,0)
      &\tfrac{449}{168}&14373&9\\
 \hline
 7\frac 16&(96;36^{\times 6},35, 6^{\times 6})&(-1,36)&(257,0)
      &\tfrac{257}{96}&4758&6\\
 \hline
 7\frac 13&(24;9^{\times 6},8,3^{\times3})&(-1,9)&(65,0)
      &\tfrac{65}{24}&327&3\\
 \hline
 \end{array}
 \end{equation}
In all cases the new constraint takes the same value at its center point as the old one but the slope to the left is steeper (because $A=-1$) and it is flat (i.e.\ with $B'=0$) rather than increasing to the right.
This completes the proof.\QED

%%%%%%%%%%%%%%%%%%%%%%%%%%%%%%%%%%%%%%%%%%%%%%%%%
\subsection{The interval $[8,9]$} \label{ss:89}
 %%%%%%%%%%%%%%%%%%%%%%%%%%%%%%%%%%%%%%%%%%%%%%%%
 
In this section we compute $c(a)$ on the interval $[8,9]$.
We first prove that $c(a) = \sqrt a$ for $a \ge 8 \frac 1{36}$.

\begin{lemma} \label{le:8eps}  
Suppose that $\mu (d;\mm)(a) > \sqrt{a}$ for some 
$a \in \left[ 8\frac 1{36}, 9 \right)$ 
with $\ell(a)=\ell(\mm)$.
Then $d \le 16$ and $m_1=\dots = m_8$.
\end{lemma}
 
\begin{proof}  
Note that $y(a) \ge y(8\frac 1{36}) = \frac {19}{36} > \frac 1q$ for all $q \ge 2$.
Assume first that $q \ge 12$.
Then $\delta := y(a)-\frac 1q \ge \frac{19}{36} -\frac 1{12} = \frac 49$.
Suppose that $m_1 \neq m_8$. Then $\si \le \frac 18$, and hence $v_M \ge \frac 34$.
This and $\delta > \frac 16$ shows that $\frac{\si}{v_M \delta} < 1$,
which is impossible by Lemma~\ref{le:7eps}~(iv).
In order to prove that $d \le 16$, note that
\begin{itemize}
\item[]
if $v_M \in [\frac 13, \frac 12]$, then $\frac{\si'}{v_M} \le \frac{1/2}{1/3} = \frac 32$;
\item[]
if $v_M \in [\frac 12, \frac 23]$, then $\frac{\si'}{v_M} \le \frac{7/9}{1/2} = \frac{14}9$;
\item[]
if $v_M \ge \frac 23$, then $\frac{\si}{v_M} \le \frac 32$.
\end{itemize}
Lemma~\ref{le:7eps}~(iv) with $\sqrt a \le 3$ therefore shows that
$$
d \,\le\, \tfrac{3}{4/9} \left( \tfrac{14}{9}\tfrac{1}{4/9}-1 \right) 
\,=\, \tfrac{27}{4} \tfrac 52 \,<\, 17 
$$
and hence $d \le 16$.

Assume now that $q \le 11$.
Note that $a \le 8 \frac{q-1}{q}$ and $\delta = y(a)-\frac 1q \ge y(8 \frac 1q) - \frac 1q$.
Lemma~\eqref{le:7eps}~(iv) therefore shows that
$$
d \,\le\, \frac{\sqrt{8 \tfrac{q-1}{q}}}{y(8 \tfrac 1q) - \tfrac 1q} \left( \sqrt q -1 \right).
$$
The RHS is $< 17$ for all $q \in \{2, \dots, 11\}$, and so $d \le 16$.
Suppose that $m_1 \neq m_9$. Then $\si \le \frac 18$.
If $q \le 8$, then $\sqrt{\si q}-1 \le 0$, contradicting~(iv) of Lemma~\ref{le:7eps}.
If $q \in \{9,10,11\}$, then 
$$
v_M \,=\, \tfrac{d}{q\sqrt a} \,\le\, \tfrac{16}{9 \sqrt 8} \,<\, \tfrac 23 , 
$$
and hence $\eps \cdot \eps \ge \frac 78 + 2\cdot \frac 19 >1$, a contradiction.
\end{proof}
 
\begin{prop}\labell{prop:8eps}
$c(a) = \sqrt a$ for $a \in \left[ 8\frac 1{36}, 9 \right)$.
\end{prop}

\begin{proof}  
Suppose to the contrary that 
$\mu(d;\mm)(a)>\sqrt{a}$ for some $a\ge 8\frac 1{36}$.
By Lemma~\ref{le:I} we may choose $a_0$ with $\ell(a_0) =\ell(\mm)$ in the interval $I$ containing $a$
on which this inequality holds.    
 
We first claim that $a_0> 8$. 
By Lemma~\ref{le:I} it suffices to see that $\ell(\mm) > 8$.  
One can prove this by explicit calculation since $\Ee_8$ is finite.   
In fact, the last obstruction given by the elements of $\Ee_8$ is that centered on $a=8$ which is discussed in Remark~\ref{rmk:8}. 
As we saw there, this is not effective when $a>8\frac 1{36}$.
  
It then follows that $a_0\ge 8\frac 1{36}$.  
For if not, because $I$ contains $a\ge 8\frac 1{36}$, it must also contain $8\frac 1{36}$.  But clearly
$\ell(z)> \ell(8\frac 1{36})=8+36 = 44$ for $z\in (8,8\frac 1{36})$.  
Therefore the minimum of $\ell(z)$ on $I$ cannot occur in this interval.

We may therefore apply Lemma~\ref{le:8eps} to $a_0$.  
Hence $d\le 16$ and $m := m_1=m_8$.
Since $\sum m_i = 3d-1 \le 47$, we must have $m \le 5$.  
It remains to check that there are no solutions to the Diophantine equations~\eqref{eq:ee} for any choice of $m \le 5$. 

Suppose first that $m=5$. We then look for solutions of 
\begin{equation} \label{e:m=5}
3d-1 = 40 + \sum_{i>8}m_i,
\quad
d^2+1 = 200 + \sum_{i>8}m_i^2 .
\end{equation}
The second equation shows that $d \in \{15,16\}$. 
For $d=15$, \eqref{e:m=5} becomes $4=\sum_{i>8}m_i$, $26=\sum_{i>8}m_i^2$,
which has no solution.
For $d=26$, \eqref{e:m=5} becomes $7=\sum_{i>8}m_i$, $57=\sum_{i>8}m_i^2$,
which has no solution either.
For $m \le 5$ there are no solution either.
\end{proof}
 
\begin{cor} $c(a) = \frac {17}{6}$ for 
$a \in [8,8\frac 1{36}]$.
\end{cor}

\begin{proof}
The class $(d;\mm)=(6;3,2^{\times 7})$ gives 
$c(a) \ge \mu(d;\mm)(a) = \frac {17}{6} = \sqrt{8 \frac 1{36}}$ for 
$a\ge 8$.  Therefore $c$ must be constant on this interval because   
it cannot decrease.  
\end{proof}

%%%%%%%%%%%%%%%%%%%%%%%%%%%%%%%%
\appendix
%%%%%%%%%%%%%%%%%%%%%%%%%%%%%%%%

%%%%%%%%%%%%%%%%%%%%%%%%%%%%%%%%%%%%%%%%%%%%%%%%%%%%%%%%%%%%%
\section{Weight expansions and Farey diagrams} \label{app:wt}
%%%%%%%%%%%%%%%%%%%%%%%%%%%%%%%%%%%%%%%%%%%%%%%%%%%%%%%%%%%%
\numberwithin{theorem}{section}

In this section, we show that the weight expansion~$\ww(a)$ described above agrees with the expansion considered in~\cite{M}. For clarity, we call the latter the Farey weight expansion; 
see Definition~\ref{def:Fw}.
It arose from a procedure of constructing an outer approximation to an ellipsoid by repeated blowing up. After explaining this, we establish the equivalence of the two definitions in Corollary~\ref{cor:x}.
No doubt, versions of this result are already known. However, since it is not hard, 
we  give a direct proof in our context.

\begin{defn}  
Let $(\rho_i= p_i/q_i)$, $i=0,\dots,N$, be a sequence of rational numbers in lowest terms, 
with $\rho_0=0/1, \rho_1=1/1$ and $\rho_i>0$ for $i>0$.
We say that the two elements $\rho_j, \rho_k$ are {\bf adjacent} in 
$(\rho_i)$ if they are neighbors when the numbers $\rho_0,\dots,\rho_N$ are arranged in increasing order.
Further $(\rho_i)$ is called
a {\bf Farey expansion} of the rational number~$a$ if the following conditions hold:
\begin{itemize}
\item[(i)] 
$\rho_N=a$;
\item[(ii)]  $q_i< q_{i+1}$ for all $i\ge 1$;
\item[(iii)] adjacent pairs $p/q, p'/q'$ of elements of $(\rho_i)$
have the property that
\begin{equation}\labell{eq:1}
|pq'-p'q|=1;
\end{equation}
\item[(iv)] condition {\rm (iii)} does not hold if any term is removed from this expansion.
\end{itemize}
\end{defn}

\begin{example}
\rm 
The Farey expansion of $\frac 47$ is $1, \frac 12, \frac 23, \frac 35, \frac 47$
which may be arranged as 
$$
\tfrac 12 < \tfrac 47 < \tfrac 35 < \tfrac 23 < \tfrac 11 .
$$ 
\end{example}

\begin{lemma} 
\NI {\rm(i)} 
Every positive rational number~$a$ has a Farey expansion.

\s
\NI {\rm(ii)} 
This expansion is unique.  
Moreover, if $\rho_1,\dots,\rho_N$ is the Farey expansion of $a=\rho_N$, 
then for all $n<N$, $\rho_1,\dots,\rho_n$ is the Farey expansion of~$\rho_n$.
\end{lemma}

\begin{proof}[Sketch of proof]  
Given positive fractions $\rho_i := \frac{p_i}{q_i}$ and $\rho_j := \frac{p_j}{q_j}$, 
we define their Farey  sum  to be
$$
\rho_i \oplus \rho_j := \frac{p_{i}+p_j}{q_{i}+q_j}.
$$
If $0<a<1$, the expansion is constructed inductively starting with $\rho_0 = 0$ and $\rho_1=1$, in such a way that 
$\rho_{i+1} := \rho_i\oplus \rho_j$ 
where~$j$ is the largest number~$<i$ such that
$a$ lies between $\rho_i$ and $\rho_j$.   
Thus $\rho_2=\frac 12$, and $\rho_3$ is 
either $\frac 13$ (if $a<\rho_2$) 
or $\frac 23$ (if $a>\rho_2$).  
The construction stops when $a=\rho_N$.
 
If $a$ lies between $k$ and $k+1$, then the expansion begins with the terms 
$\rho_i := \frac i1$, $i=1,\dots, k+1$.  
Then $\rho_{k+2} = \rho_k \oplus \rho_{k+1} = \frac{2k+1}2$ and the expansion proceeds as in the previous case.
Further details may be found in Hardy and Wright, \cite[Ch.~III]{HW}.
\end{proof}

\begin{figure}[ht]
 \begin{center}
  \psfrag{v-1}{$(0,1)=v_{-1}$}
  \psfrag{v0}{$(1,0)=v_0$}
  \psfrag{v1}{$(1,1)$}
  \psfrag{v2}{$(1,2)$}
  \psfrag{v3}{$(2,3)$}
  \psfrag{v4}{$(3,5)$}
  \psfrag{e-1}{$\vareps_{-1}$}
  \psfrag{e0}{$\vareps_0$}
  \psfrag{e1}{$\vareps_1$}
  \psfrag{e2}{$\vareps_2$}
  \psfrag{e3}{$\vareps_3$}
  \psfrag{e4}{$\vareps_4$}
  \leavevmode\epsfbox{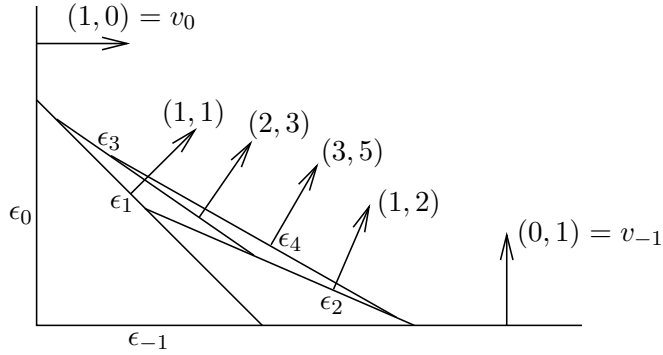}
 \end{center}
\caption{The Farey diagram for $p/q = 5/3$. 
Here $v_1=(1,1),\; v_2=(1,2),\; v_3=(2,3),\; v_4=(3,5)$.
The edge $\vareps_2$ meets $\vareps_1$ and $\vareps_{0}$,
while $\vareps_3$ meets $\vareps_1$ and $\vareps_{2}$,
and $\vareps_4$ meets $\vareps_2$ and $\vareps_{3}$.}
\label{figure.farey}
\end{figure}

One can build a diagram in $\R^2$ corresponding to a given Farey expansion by associating to each
fraction $p_i/q_i$ a line segment $\vareps_i$ (called an {\it edge}) with normal vector $v_i := (q_i,p_i)$ of slope $p_i/q_i$. 
See Figure~\ref{figure.farey}.
One starts with the first quadrant whose edges $\vareps_{-1}, \vareps_0$ are the positive coordinate axes with  (inward) normals $(q_{-1},p_{-1}) = (1,0)$
and $(q_0,p_0)=(0,1)$, and builds up a sequence of edges by cutting along certain directions.  
The first cut is along an edge $\vareps_1$ 
going from $\vareps_{-1}$ to $\vareps_{0}$ with normal $v_1=(1,1)$.  
In general,  
if $\rho_i$ is the Farey sum of $\rho_j$ with $\rho_{i-1}$ for some $j<i-1$ then the 
$i$th cut is along an edge $\vareps_i$ with normal $v_i$ that meets the edges $\vareps_j$ and $\vareps_{i-1}$ (but none of the others). 
The collection of edges $\vareps_1,\dots,\vareps_N$ is called the {\bf Farey diagram}; the {\bf extended Farey diagram} also includes the edges $\vareps_{-1},\vareps_0$.

As described in \cite[\S3]{M}, adding a new edge whose normal is the sum of the two adjacent normals corresponds to a (smooth) blow up, since in the toric model, each blow up corresponds to cutting off a corner of the moment polytope. Therefore we can think of the process of constructing the Farey expansion 
for~$a$ as the process of blowing up the first quadrant repeatedly and in as efficient a way as possible, 
in order to obtain a (smooth) polytope with one edge whose normal has slope~$a$.
In the language of~\cite{M}, this is an outer approximation; 
see Figure~3.1 and Lemma~3.8~ff.\ in~\cite{M}. 
For further discussion of the relation between weight sequences and 
the resolution of singularities by blow up, see the end of~\cite{Mcf}.
This contains a description of the Riemenschneider staircase that links the 
weight expansion for~$a$ to the Hirzebruch--Jung continued fraction expansions  
for the two singular points at the vertices of the toric model of the ellipsoid~$E(1,a)$.

Given such a sequence of edges $\vareps_1,\dots,\vareps_N$ one can define an associated sequence of 
{\bf Farey labels} $\la_1,\dots,\la_N$ as follows,
starting with the last $\vareps_N$ that is labeled by $\la_N := 1$.
 
\MS
\begin{itemize}
\item[($\ast$)]
{\it   
If  $\vareps_j, j>n,$ is labeled by $\la_j$, label $\vareps_n$ with the sum of the labels of the edges $\vareps_j, j>n,$ that intersect $\vareps_n$.}
\end{itemize}

\begin{defn}\labell{def:Fw} 
If $a=p/q$ has Farey diagram with labels $\la_i, 1\le i\le N$, 
the {\bf Farey weights} of $a$ are the numbers $u_i := \la_i/\la_1$, for $i=1,\dots,N$. 
\end{defn}

Note that reflection in the line $p=q$ converts the Farey diagram for $p/q$ into that for $q/p$. Therefore the Farey weights for $p/q$ and $q/p$ are equal.

These Farey weights are the weights considered in~\cite{M}. 
Our aim in this section is to show
that these agree with the weights $\ww(a)$ of Definition~\ref{def:wa}.

\begin{example}\rm  
One can see from Figure~\ref{figure.farey} that 
when $p/q = 5/3$ the labels $\la_i$ (in decreasing order) are 
$\la_4 = 1$, $\la_3=1$, $\la_2=2$, $\la_1=3$ 
which gives the Farey weights $1, \frac 23, \frac 13, \frac 13$. 
These agree with the weight expansion constructed in Definition~\ref{def:wa}.
\end{example}

In the following we will denote the distinct 
Farey labels for $a=p/q$ by $h_1>h_2>\dots> h_S>0$ 
and will suppose that they occur with multiplicities $n_1,\dots,n_S$.
Thus we write
$$
(\la_1,\la_2,\dots,\la_N) \,=\, \bigl(h_1^{\times n_1},\dots, h_S^{\times n_S}\bigr).
$$
 
\begin{prop}\labell{prop:x} 
Let $h_1>\dots> h_S>0$ be the distinct Farey labels for $a=p/q$ and suppose that they occur with multiplicities $n_1,\dots,n_S$.  
\begin{itemize}
\item[(i)]
If $a\in(k,k+1]$ then $n_1=k$ and $h_1=q, h_2= p-kq$;

\item[(ii)] If $a\in [1/(k+1),1/k)$ then $n_1=k$ and $h_1=p, h_{2} = q-kp$;
  
\item[(iii)] In both cases the $h_i$ for $1\le i<S$ satisfy the recursion relation
$$
h_{i} = n_{i+1}h_{i+1} + h_{i+2},
$$
where $h_{S+1} := 0$.
\end{itemize}
\end{prop}

\begin{cor}\label{cor:x}  
For all $a>1$ the weights $\ww(a)$ of Definition~\ref{def:wa} are the Farey weights of~$a$.
\end{cor}

\begin{proof} 
If we write $\ww(a)$ as
$
\ww(a) = \Bigl(1^{ \ell_1}, x_2^{\times \ell_2},\dots,x_K^{\times \ell_K}\Bigr)
$
then Definition~\ref{def:wa} implies that the~$x_i$ are characterized by the properties that  
$x_1=1$, $x_i>x_{i+1}\ge 0$ and the recursive relation
$$
x_i = \ell_{i+1}x_{i+1}+ x_{i+2}.
$$
Since the $\la_i$ are positive and nonincreasing, Proposition~\ref{prop:x} shows that Farey weights $\la_i/\la_1$ have precisely the same characterization.
\end{proof}

\NI 
{\bf Proof of Proposition \ref{prop:x}.}\;\,
Reflection in the line $p=q$ converts the Farey diagram for~$p/q$ into that for~$q/p$.  
Therefore statements~(i) and (ii) are equivalent.  
We will prove all three statements together by an inductive argument.

We use the extended diagram obtained by adding to the edges $\vareps_1,\dots,\vareps_N$ 
the edge $\vareps_0$ with normal $v_0 = (0,1)$ and the edge $\vareps_{-1}$
with normal $v_{-1} = (1,0)$. When $a>1$ we order them as 
$\vareps_{-1}, \vareps_0, \vareps_1,\dots$, and then
label these as in~($\ast$) above; 
when $a<1$ we order them as 
$\vareps_0, \vareps_{-1}, \vareps_1,\dots$, and then label them using~($\ast$).

If $k<a<k+1$, the Farey expansion starts with $1,2,\dots,k, k+1$ 
and then contains  
further elements between $k$ and $k+1$.  
It follows that $n_1=k$.  Further, because the only edges meeting 
$\vareps_0$ are $\vareps_1, \dots,\vareps_{k+1}$, we have 
\begin{equation}\labell{eq:n1}
\la_{0} = k\la_1+\la_{k+1} = n_1h_1+h_2.
\end{equation}
Similarly, because the only edges meeting $\vareps_{-1}$ are $\vareps_{0}$ and $\vareps_1$, 
we have $\la_{-1}= \la_{0}+\la_1$.  
Therefore~(i) is equivalent to
\SSS

\NI
(iv) {\it $n_1=k$ and 
$\la_0 = p,\;\; \la_{-1}= p+q$ when $p/q>1$}.\SSS

\NI
Similarly, if $a\in (1/(k+1),1/k)$ we find that  
$$
n_1=k,\quad \la_{-1} = k\la_1+\la_{k+1} = n_1h_1+h_2,\quad \la_{0}=\la_{-1} +\la_1.
$$
Hence (ii) is equivalent to 
\SSS

\NI
(v) {\it $n_1=k$ and
$\la_{-1} = q,\;\; \la_0= p+q$   when} $p/q<1$.\MS

We  argue by induction on~$N$, the length of the Farey expansion of~$a$.  
By symmetry, it suffices to consider the case when $a\in (k,k+1]$.  
The result is clear when $a=k+1$ 
(and also for the trivial case $a=1$ which has a single label $\la_1=1$). 
Point~(i) is easily checked when $N=k+2$
since then $p/q = (2k+1)/2$. Similarly, one can check it for the two numbers 
$(3k+1)/3, (3k+2)/3$ with $N=k+3$.
(Note that $v_{k+2} = (2,2k+1)$ and $v_{k+3}$ is either 
$(3,3k+1)=(1,k) \oplus (2,2k+1)$ or $(3,3k+2)= (2,2k+1) \oplus (1,k+1)$.)
Now, consider the matrix 
$$
A_k \,=\,
\begin{pmatrix}
k+1 & -1 \\
-k & 1
\end{pmatrix}
$$
that takes the vectors $(1,k), (1,k+1)$ to 
$(1,0) = v_{-1}, (0,1)=v_0$.
Then
$$
A_k 
\begin{pmatrix}
\,q\, \\ \,p\,
\end{pmatrix}
\,=\,
\begin{pmatrix}
(k+1)q-p \\ -kq+p
\end{pmatrix}
\,=:\,
\begin{pmatrix}
\,q'\, \\ \,p'\,
\end{pmatrix} .
$$
Therefore if  $p/q =: k + x$, we find $p'/q' = x/(1-x)$.
In particular,  $p'/q'>1$ if and only if $x>1-x$, that is, exactly if $n_2=1$.  

Because $\det A_k=1$, $A_k$ preserves the Farey addition relation between adjacent 	normals.  
Hence if $v_1=(1,1), v_2,\dots, v_N$ are the normals in the diagram for $p/q$, the normals for the diagram for
$p'/q'$ are
$$
A_k v_{k+2}=(1,1),\, A_kv_{k+3},\, \dots,\, A_kv_N.
$$  
In fact one could construct the diagrams for $p/q$ and $p'/q'$ so that there is an  affine transformation obtained by following $A_k$ by a suitable translation that takes
the standard diagram for $p/q$ to the extended diagram for $p'/q'$.  
Therefore, if  ${p'}/{q'}> 1$, the labels $\la_k,\la_{k+1}, \dots, \la_N$ for $p/q$ equal 
the labels $\la_{-1}', \la_0', \la_1',\dots,\la_{N-k-1}'$ of the extended diagram for $p'/q'$.  
Hence:

\begin{itemize}
\item[]
{\it if  ${p'}/{q'}> 1$, then $n_2=1$ and the multiplicities for 
$p'/q'$ are $n_3,\dots,n_S$ 
with corresponding labels $h_3,\dots,h_S$.}
\end{itemize}

\NI
Therefore because the recursive relation (iii) holds for $p'/q'$ it holds for $p/q$ and $i\ge 3$. 
Further, by equation~\eqref{eq:n1}) and (iv) applied to $p'/q'$,  
$\la_0' = n_3h_3 + h_4 = p'$ and 
$\la_{-1}' = \la_0' + \la_1' = p' + h_3$.
Therefore, since $n_2=1$,
$$
h_2 = \la_{k+1} = \la_0' = n_3h_3 + h_4 = p',\quad
h_1 = \la_1 = \la_{-1}' = n_2 h_2 + h_3.
$$
This shows that (iii) holds for $p/q$.
Moreover, $h_2=p' = p-kq$ and, by (i) for $p'/q'$,
we find $h_1 = p'+h_1' = p'+ q'= q$.

This completes the proof when $p'/q'>1$.  
When $p'/q'<1$, the proof is similar. By (v),
the labels $\la_0',\la_{-1}',\la_1',\dots$ (note the reordering) for the extended diagram 
for $p'/q'$ are $\la_k+\la_{k+1},\la_{k},\la_{k+2},\dots$, with first multiplicity
$n_1'=n_2-1$. Further details will be left to the reader.
\proofend

%%%%%%%%%%%%%%%%%%%%%%%%%%%%%%%%%%%%%%%%%%%%%
\section{Computer programs} \labell{app:comp}
%%%%%%%%%%%%%%%%%%%%%%%%%%%%%%%%%%%%%%%%%%%%%

\subsection{Computing $c(a)$ at a point~$a$}  \label{a:SolLess}
In this section we describe a {\tt Mathematica} program
{\tt SolLess[a,D]}
that finds for a rational number $a$ and a natural number~$D$ 
all classes $(d;\mm) \in \Ee$
with $\ell(\mm) = \ell(a)$ and $\mu(d;\mm)(a) > \sqrt{a}$ and $d\le D$.
We have applied this program in the proof of Theorem~\ref{thm:78}
to eight numbers $z_k = 7 \frac 2{2k+1}$ in $[7 \frac 19,8]$.
The present program can be used for all $a$.
By removing one line, one obtains a program finding {\it all}\/ 
obstructive solutions at~$a$ with $d \le D$ 
(not just those with $\ell (\mm) = \ell(a)$).

Recall from Remark~\ref{rmk:computer1} that instead of using the code {\tt SolLess},
one can use the algebraic method from the proof of Lemma~\ref{le:7k}
to find all obstructive classes $(d;\mm)$ at $z_k$ with $\ell(\mm)=\ell(z_k)$.
We have chosen to use this code for convenience, 
and because it might be helpful for understanding the more involved code 
of~\S~\ref{a:comp.inter}.

\MS
We start with computing the weight expansion $\ww (a)$ of a rational number~$a$.
For convenience, we use that the multiplicities of $\ww (a)$ are given by the
continued fraction expansion of~$a$.

\begin{verbatim}
W[a_] := Module[{aa=a,M,i=2,L,u,v},
                 M = ContinuedFraction[aa];
                 L = Table[1, {j,M[[1]]}];
                 {u,v} = {1,aa-Floor[aa]};
                 While[i <= Length[M],
                       L = Join[L, Table[ v, {j,M[[i]]}] ];   
                       {u,v} = {v,u - M[[i]] v};
                       i++];
                 Return[L] ]
\end{verbatim}

\ni
For instance, {\tt W[3+2/3]} yields {\tt \{1,1,1,2/3,1/3,1/3\}}.

\b \ni
We next give for each natural number~$k$ a list of $4$~vectors,
from which we will construct candidates for the vectors $\mm$.

\begin{verbatim}
P[k_] := Module[{kk=k,PP,T0,i},
                 T0  = Table[0,{u,1,k}];
                 T0p = ReplacePart[T0,1,1];
                 T1  = Table[1,{u,1,k}];
                 T1m = ReplacePart[T1,0,-1];
                 PP = {T0,T0p,T1,T1m};
                 Return[PP] ]       
\end{verbatim}

\ni
For instance, {\tt P[3]} yields 
{\tt \{0, 0, 0\}, \{1, 0, 0\}, \{1, 1, 1\}, \{1, 1, 0\} }.

\b
\ni
Our next task is to construct for given $a$ all candidate vectors $\mm$.
To this end we first take a given multiplicity vector {\tt M}, 
say $(k_1, k_2, k_3)$,
and associate to it all vectors of length $k_1+k_2+k_3$ such that
that the $j$\,th block is a vector from {\tt P[$k_j$]}.
In the example we thereby obtain $4^3$ vectors .
We use the sets {\tt P[k]} and a recursion:

\begin{verbatim}
Difference[M_] := Module[{V=M,vN,V1,l,L={},D,PP,i,j,N},
                          l = Length[V];
                          If[ l == 1, L = P[ V[[1]] ]];
                          If[ l >  1,
                              vN = V[[-1]];
                              V1 = Delete[V,-1];
                              D  = Difference[V1]; 
                              PP = P[vN];
                              i  = 1;
                              While[ i <= Length[D],
                                     j=1;
                                     While[j <= Length[PP],
                                           N = Join[ D[[i]], PP[[j]] ];
                                           L = Append[L,N];
                                           j++];
                                     i++]
                            ];
                          Return[L] ]
\end{verbatim}

\ni
We now take a positive rational number $a$ and $d \in \NN$ 
and compute all solutions of the Diophantine equation with~$d$ given 
that are obstructive at~$a$:
We first take the multiplicity vector $\tt M$ of $\ww(a)$,
and then round down each of its entries, getting~$F$.
In view of Lemma~\ref{le:atmost1}, an 
obstructive 
multiplicity vector~$\mm$
must be of the form {\tt F+D[[i]]}, 
where {\tt D[[i]]} is the i\,th vector from the list {\tt Difference[W[a]]}.
We therefore run through this list, and each time
check whether {\tt V=F+D[[i]]}
is a solution of the Diophantine system, has last entry positive,
and is obstructive: $\mu(d;V)(a) > \sqrt{a}$.
If all three conditions are fulfilled, we add {\tt V} to our list, 
and also retain~$d$.

\begin{verbatim}
Sol[a_,d_] := Module[{aa=a,dd=d,M,F,D,i,V,L={}},
                 M = ContinuedFraction[aa];
                 F = Floor[ dd/Sqrt[aa] W[aa] ];
                 D = Difference[M];
                 i=1;
                 While[i <= Length[D],
                       V = Sort[F+D[[i]], Greater];
                       SV = Sum[ V[[j]], {j,1,Length[V]} ];
                       If[ {SV, V.V} == {3dd-1, dd^2+1} 
                            && V[[-1]] > 0
                            && W[aa].V / dd >= Sqrt[aa],
                            L = Append[L, V]
                         ];
                       i++];
                 Return[{dd,Union[L]}] ]
\end{verbatim}
For instance, 
{\tt Sol[7 + 1/8, 48]} yields 
$$
{\tt \{48, \{\{18, 18, 18, 18, 18, 18, 18, 3, 2, 2, 2, 2, 2, 2, 2\}\}\}}.
$$

\ni
\begin{remark}
{\rm
(i)
We were not at all economical when constructing the list {\tt Difference[M]}:
In view of Lemma~\ref{le:atmost1}, for an 
obstructive vector {\tt F+D[[i]]} 
there is at most one~$k_j$ such that the vector {\tt P[$k_j$]}
appearing in {\tt D[[i]]} can have both $0$ and~$1$ as entries.
We have chosen this form of the program to make it more readable.

\s
(ii)
In the main body of the paper, we applied this program only to the eight numbers
$z_k=7 \frac 2{2k+1}$,
and for these numbers we know that $m_1=m_7$ by Lemma~\ref{le:7less}.
We did not use this information so as to make the program applicable also at other points, 
e.g.~to $7 \frac 1k$ in order to check Lemma~\ref{le:7k} (at least for all $d \le 2000$ or so).

\s
(iii)
Recall from Lemma~\ref{le:I} that for every $(d;\mm)$ 
that gives an obstruction at~$a$ we have
$\ell(a) \ge \ell(\mm)$.
The condition {\tt V[[-1]] > 0} asked in {\tt Sol[a,d]} is therefore equivalent to
$\ell(a) = \ell(\mm)$.
By removing this condition, we obtain a program finding {\it all}\/ obstructive 
solutions~$(d;\mm)$ at~$a$.
\diam
}
\end{remark}

\ni
We finally collect, for given $a$ and $D \in \NN$, all
solutions that are obstructive at~$a$
and have $d \le D$:

\begin{verbatim}
SolLess[a_,D_] := Module[{aa=a,DD=D,d=1,Ld,L={}},
                          While[d <= D,
                                Ld = Sol[aa,d];
                                If[ Length[ Ld[[2]] ] > 0,
                                    L = Append[L,Sol[aa,d]]
                                  ];
                                d++];
                          Return[L] ]           
\end{verbatim}

\subsection{Computing $c(a)$ on an interval}  \label{a:comp.inter}

In this section we describe a {\tt Mathematica} program
{\tt InterSolLess[k,D]}
that provides for a natural number~$D$ a finite list of candidate classes $(d;\mm) \in \Ee$
with 
$\ell (\mm) = \ell(a)$ and 
$\mu(d;\mm)(a) > \sqrt{a}$ and $d \le D$
for some $a \in \;]7\frac 1{k+1}, 7 \frac 1k[$, $a \neq z_k$, where $z_k = [7;k,2]$.
We have applied this program in the proof of Theorem~\ref{thm:78} to the eight intervals
$]7 \frac 1{k+1}, 7 \frac 1k [$, $k \in \{1, \dots, 8\}$.
Throughout we assume that $a,b$ and the~$m_i$ are positive integers.

\MS
Our first goal is to list for a given pair $a,b$
all solutions of the Diophantine system
\begin{equation} \label{e:dioab}
\left\{
\begin{array} {rcl}
 a &=& \sum_i m_i \\
 b &=& \sum_i m_i^2 
 \end{array}\right.
\end{equation}
To illustrate our method, 
let us find in an algorithmic way the solutions $\mm$ of~\eqref{e:dioab}
for $(a,b)=(4,6)$.
It suffices to list solutions $\mm = (m_1,m_2,\dots,m_M)$ with $m_1 \ge m_2 \ge \dots \ge m_M$.
We must have $m_1 \le \lfloor \sqrt 6 \rfloor =2$.
We therefore try with $m_1=1$ and $m_1=2$.
For a solution~$\mm$, the next numbers $(m_2,m_3,\dots)$ must fulfill~\eqref{e:dioab}
with $(a,b)=(4-1,6-1)=(3,5)$ and $(a,b)=(4-2,6-4)=(2,2)$.
In the first case (when $m_1=1)$, we only need to try with $m_2=1$.
The next numbers $(m_3,\dots)$ of a solution must then fulfill~\eqref{e:dioab}
with $(a,b)= (3-1,5-1)=(2,4)$.
Then $a^2=b$, so that the only solution is $m_3=2$.
But $m_3=2 > 1=m_2$, whence we discard the solution $(1,1,2)$.
In the second case (when $m_1=2$), we try to find numbers $(m_2,m_3,\dots)$ 
solving~\eqref{e:dioab} with $(a,b)=(4-2,6-4)=(2,2)$. 
We  only need to try with $m_2=1$, and then want to solve~\eqref{e:dioab} with
$(a,b)=(2-1,2-1)=(1,1)$ for $m_3$.
Since $a^2=b$, the only solution is $m_3=1$.
We therefore find the solution $\mm = (2,1,1)$.

The code {\tt Solutions[a,b]} below does the same thing by a recursion.
Note that if $a^2 < b$, then~\eqref{e:dioab} has no solution.
%Also note that if we have already chosen numbers $(m_1,m_2, \dots, m_j)$
%to build a solution~$\mm$, then it suffices to consider only those $m_{j+1}$ 
%with $m_{j+1} \le m_j$.
%
\begin{verbatim}
Solutions[a_,b_] := Solutions[a,b,Min[a,Floor[Sqrt[b]]]] 

Solutions[a_,b_,c_] := Module[{A=a,B=b,C=c,i,m,K,j,V,L={}}, 
                               If[ A^2 < B, L={}];
                               If[ A^2== B, 
                                   If[ A > C, L={}, L={{A}} ] ];
                               If[ A^2 > B,
                                   i=1;
                                   m = Min[Floor[Sqrt[B]],C];
                                   While[i <= m,
                                         K = Solutions[A-i,B-i^2,i];
                                         j=1;
                                         While[j <= Length[K],
                                               V = Prepend[ K[[j]], i];
                                               L = Append[L,V];
                                               j++
                                              ];
                                         i++]
                                 ];
                               Return[Union[L]] ]
\end{verbatim}
Notice that applied to $(a,b) = (3d-1,d^2+1)$, the above algorithm lists all solutions
of our principal Diophantine equation. 
For large~$d$, however, there are many solutions.
We shall therefore directly choose the first $7+k+1$ numbers $m_i$,
using that for 
obstructive solutions the vectors $\mm$ and $\ww(a)$ must be essentially parallel,
and shall then use the code {\tt Solutions} only to choose the remaining $m_{7+k+2}, \dots$.

It will be useful to have a short expression for the sum of the entries of a vector {\tt L}:
\begin{verbatim}
sum[L_] := Sum[ L[[j]], {j,1,Length[L]} ]
\end{verbatim}

For given $k \ge 1$ the following code gives three vectors of length~$7+k$
that have all entries equal to~$0$ except that the last entry of the first vector is~$-1$ 
and the eighth entry of the third vector is~$1$.
\begin{verbatim}
P[k_] := Module[{kk=k,PP,T0,i},
                 T0 = Table[0,{i,7+kk}];
                 Tm = ReplacePart[T0,-1,-1];
                 Tp = ReplacePart[T0,1,8];
                 PP = {Tm,T0,Tp};
                 Return[PP] ]       
\end{verbatim}
For $k=4$, this gives 
${\tt \{0^7, 0, 0, 0, -1\}}, \; {\tt \{0^7, 0, 0, 0, 0\}, \; {\tt \{0^7, 1, 0, 0, 0\}}}$.
We shall use these vectors to take into account that the~$m_i$ may not
be constant on the second block.

\m
Fix $d \in \NN$ and $k \in \{1, \dots, 8\}$.
Assume that $\mm$ is such that $(d;\mm) \in \Ee$ 
and such that 
$\ell (\mm) = \ell(a)$ and 
$\mu(d;\mm)(a) > \sqrt{a}$ 
for some $a \in \;]7\frac 1{k+1}, 7 \frac 1k[$.

In view of Lemma~\ref{le:atmost1} we have $m_1=m_7$.
Moreover, since $|\eps_1| \le \frac 1{\sqrt 7} < \frac 12$ and $a < 7 \frac 1k$, 
$$
m_1 \,=\, \tfrac d{\sqrt a} + \eps_1 \,>\,  \tfrac d{\sqrt{7 \tfrac 1k} }-\tfrac 12
$$
and hence $m_1 \ge {\tt m1 := Round} \left( \frac d{\sqrt{7+\frac 1k}} \right)$.
In the same way we see that 
$m_1 \le {\tt M1 := Round} \left( \frac d{\sqrt{7+\frac 1{k+1}}} \right)$.
The number $m_1=m_7$ must therefore be in the interval $[ {\tt m1,M1} ]$.

Next, consider $m_j$ for $j \in \{ 8, \dots, 7+k \}$.
For $a=7+x$ we have $\frac{1}{k+1} < x < \frac 1k$.
Therefore, 
$$
m_j \,=\, \tfrac{d}{\sqrt a} x + \eps_j \,>\, \tfrac{d}{\sqrt{7 \tfrac 1k}} \tfrac{1}{k+1} -1
$$
and hence $m_j \ge {\tt mx} := 
\Bigg\lceil \frac{d}{\sqrt{7 \frac 1k}} \frac{1}{k+1} \Bigg\rceil -1 =
{\tt Ceiling} \left( \frac{d}{\sqrt{7 \frac 1k}} \frac{1}{k+1} \right) -1$.
In the same way we see that 
$m_j \le {\tt Mx} := 
\Bigg\lfloor \frac{d}{\sqrt{7 \frac 1{k+1}}} \frac 1k \Bigg\rfloor +1 =
{\tt Floor} \left( \frac{d}{\sqrt{7 \frac 1{k+1}}} \frac 1k \right) +1$.
The numbers $m_8, \dots, m_{7+k}$ must therefore be in the interval $[ {\tt mx,Mx} ]$.

Since $a \in\,]7 \frac 1{k+1}, 7 \frac 1k[$, we have $\ell(\mm)=\ell(a) \ge 7+k+1$.
Lemma~\ref{le:irrel}~(i) applied to the second block shows that
$|m_7-(m_8+ \dots +m_{7+k}) - m_{7+k+1})| \le {\tt Ceiling} \left( \sqrt{k+2}\right)-1$.
Note that we can assume that $m_1 \ge m_2 \ge \dots m_{7+k+1}$ and that $m_{7+k+1} \ge 1$.

Using this information about $m_1, \dots, m_{7+k+1}$, 
we build a preliminary list of vectors $\mm$ as follows:
We first take all possibilities for $m_1, \dots, m_{7+k+1}$ into account.
Since the full vector $\mm$ solves the Diophantine equation for~$d$,
the remaining numbers $m_{7+k+2}, \dots, m_M$ must solve the Diophantine equation~\eqref{e:dioab}
with 
$$
a = 3d-1-\sum_{i=1}^{7+k+1}m_i,
\qquad 
b = d^2+1-\sum_{i=1}^{7+k+1}m_i^2 .
$$
Note that we can assume that $a \ge 0$ and $b \ge 0$. 
We then take the list {\tt Solutions[a,b,M[[-1]]} of all solutions
to~\eqref{e:dioab} for which all $m_j$ are at most $m_{7+k+1}$,
and append each such solution to~$(m_1, \dots, m_{7+k+1})$.
It could be that the only solution is $0$ (namely if $a=b=0$);
in this case we remove the entry~$0$. 
\begin{verbatim}
Prelist[k_,d_] := Module[{kk=k,dd=d,u,v,m1,M1,mx,Mx,f,t,
                          PP,M,MM,i=0,j=0,s=1,S,T,K,l,L={}},
                  u  = 1/(kk+1);
                  v  = 1/kk;
                  m1 = Round[dd/Sqrt[7+v]];
                  M1 = Round[dd/Sqrt[7+u]];
                  mx = Floor[dd/Sqrt[7+v] u]-1;
                  Mx = Ceiling[dd/Sqrt[7+u] v]+1;
                  f  = Ceiling[Sqrt[kk+2]-1];
                  t  = -f;
                  PP = P[kk];
                  While[i <= M1-m1,
                    While[j <= Mx-mx,
                      While[s <= 3,
                        While[t <= f,
                              M = Join[ Table[m1+i, {u,7}], Table[mx+j, {u,kk}] ];   
                              M = M + PP[[s]];
                              S = Sum[ M[[u]], {u,8,8+kk-1}];
                              M = Append[M, M[[7]]-S+t];  
                              T=1;
                              If[ M == Sort[M,Greater] && M[[-1]] > 0, T=1, T=0];  
                              S = sum[M];
                              A = 3dd-1-S;
                              B = dd^2+1-M.M;
                              If[ Min[A,B] < 0, T=0];  
                              If[ T==1,
                                  K = Solutions[A,B,M[[-1]]]; 
                                  l=1;
                                  While[l <= Length[K],
                                        MM = Join[ M,K[[l]] ]; 
                                        While[ MM[[-1]] == 0, MM=Drop[MM,-1] ];     
                                        L = Append[L,MM];
                                        l++
                                       ]
                                ]; 
                        t++];
                        t=-f;
                      s++];
                      s=1;
                   j++];
                   j=0;
                  i++];
                  Return[{dd,Union[L]}] ]
\end{verbatim}

Many of the solutions $(d;\mm)$ in {\tt Prelist[k,d]} are not obstructive.
The next code removes most of these solutions:

\begin{verbatim}
InterSol[k_,d_] := Module[{kk=k,dd=d,L,M,T,K={},i=1,l,rest},
                   L = Prelist[kk,dd][[2]];
                   While[i <= Length[L],
                         M = L[[i]];
                         l = Length[M]; 
                         T = 1;
                         If[ l <= 7 + kk + 2, T=0];
                         If[ M[[-2]]-M[[-1]] > 1, T=0 ];
                         If[ M[[-3]] > M[[-2]] + 1 
                             && Abs[ M[[-3]]-M[[-2]]-M[[-1]] ] > 1, T=0 ];
                         If[ kk==1 && l >= 10,
                             If[ M[[9]] - M[[10]] > 1 &&
                                 Abs[ M[[8]] - (M[[9]] + M[[10]]) ] > 1,
                                 T=0 ]];
                         rest = Sum[ M[[j]], {j,8+kk,l} ];
                         If[ M[[7+kk]] - rest >= Sqrt[l-kk-6], T=0 ];   
                         If[ T==1, K = Append[K, M] ]; 
                         i++];
                   Return[{dd,K}] ]
\end{verbatim}
Recall that 
$a \in \;]7\frac 1{k+1}, 7 \frac 1k[$, $a \neq z_k$, where $z_k = [7;k,2]$.
In particular, $\ell(\mm) = \ell(a) > 7+k+2$.
We then test the very end of $\mm$: Since the last block has length at least~2,
we must have $m_{M-1} = m_M+1$ or $m_{M-1} = m_M$
in view of Lemma~\ref{le:atmost1}.

We next exploit Lemma~\ref{le:irrel}:
If $m_{M-2} > m_{M-1}+1$, then we know that the length of the last block is~2,
and $m_{M-2}$ belongs to the before the last block.
Lemma~\ref{le:irrel}~(ii) then shows that $|m_{M-2}-(m_{M-1}+m_M)| \le 1$.
In the next test we apply the same lemma to the special situation where
$k=1$ and where we know that the third block has length~$1$. 

In the last test we apply Lemma~\ref{le:irrel}~(ii) with $j=7+k$. 

\m
\ni
We finally take the union over $d \le D$ of the solutions in {\tt InterSol[k,D]}:
\begin{verbatim}
InterSolLess[k_,D_] := Module[{kk=k,DD=D,LL={},Q,d=1},
                               While[d <= DD,
                                     Q = InterSol[kk,d];
                                     If[Length[Q[[2]]] > 0,
                                        LL = Append[LL,Q]];
                                     d++];
                               Return[LL] ]      
\end{verbatim}

%%%%%%%%%%%%%%%%%%%%%%%%%%%%%%%%%%%%%%%%%%%%%%%%%%%%%%%%%%%%%%%%%%%%%%%%%%


\begin{thebibliography}{cccccc}


\bibitem{BR}
M.~Beck and S.~Robins,
Computing the continuous discretely.
Integer-point enumeration in polyhedra. 
{\it Undergraduate Texts in Mathematics.} 
Springer, New York, 2007. 


\bibitem{B}
P.~Biran,
Symplectic packing in dimension~$4$,
{\it Geom.\ Funct.\ Anal.}~{\bf 7}
(1997), 420--437.

\bibitem{B-99}
P.~Biran,
Constructing new ample divisors out of old ones, 
{\it Duke Math.~J.}~{\bf 98} (1999), 113--135.


\bibitem{B-01}
P.~Biran,
From Symplectic Packing to Algebraic Geometry and back, 
European Congress of Mathematics, Vol.~II (Barcelona, 2000), 507--524, 
{\it Progr.~Math.}~{\bf 202}, Birkh\"auser, Basel, 2001.


\bibitem{CHLS}
K.~Cieliebak, H.~Hofer, J.~Latschev and F.~Schlenk,
Quantitative symplectic geometry,
Dynamics, ergodic theory, and geometry,  1--44,
{\it Math.\ Sci.\ Res.\ Inst.\ Publ.}~{\bf 54},
Cambridge Univ.\ Press, Cambridge, 2007.



\bibitem{EH}
I.~Ekeland and H.~Hofer, 
Symplectic topology and Hamiltonian dynamics. II
{\it Math.\ Z.}~{\bf 203} (1990), 553--567.


\bibitem{Gu} 
L.~Guth,
Symplectic embeddings of polydisks,
{\it Invent. Math.}~{\bf 172} (2008), 477--489.

\bibitem{HW}  G.~H.~Hardy and E.~M.~Wright, 
{\it An Introduction to the Theory of Numbers}, 
OUP, Oxford (1938).

\bibitem{HK}  
R.~Hind and E.~Kerman, 
New obstructions to symplectic embeddings, 
arxiv:0906.4296.

\bibitem{HT}
M.~Hutchings and C.~H.~Taubes,
Gluing pseudoholomorphic curves along branched covered cylinders,~I,
arXiv:math/0701300, to appear in {\it Journ.\ Symp.\ Geom.}

\bibitem{HT2}
M.~Hutchings and C.~H.~Taubes, in preparation.


\bibitem{LL2}
Bang-He Li and T.-J.~Li,
Symplectic genus, minimal genus and diffeomorphisms,
{\it Asian J.~Math.}~{\bf 6} (2002), 123-144.

 
\bibitem{LL}
T.-J.~Li and A.~K.~Liu,
Uniqueness of symplectic canonical class, surface cone and symplectic cone
of $4$-manifolds with $b^+=1$,
{\it J.~Differential.\ Geom.}~{\bf 58} (2001), 331--370.


\bibitem{Mdef}
D.~McDuff, 
From symplectic deformation to isotopy,
Topics in symplectic $4$-manifolds (Irvine, CA, 1996), 85--99,
{\it First Int.\ Press Lect.\ Ser.,~I},
Int.\ Press, Cambridge, MA, 1998.


\bibitem{M}
D.~McDuff,
Symplectic embeddings of $4$-dimensional ellipsoids,
{\it J.~Topol.}~{\bf 2} (2009), 1--22.


\bibitem{Mcf}
D.~McDuff, 
Symplectic embeddings and continued fractions: a survey,
arXiv:0908.4387,
to appear in {\it  Journ.\ Jap.\ Math.\ Soc.}


\bibitem{MP}
D.~McDuff and L.~Polterovich, 
Symplectic packings and algebraic geometry,
{\it Invent.\ Math.}, {\bf 115} (1994), 405--29.


\bibitem{DoM} 
Dorothee M\"uller,
Symplectic embeddings of ellipsoids into polydiscs,
PhD thesis, Universit\'e de Neuch\^atel, 
in preparation.


\bibitem{S}
P.~Seidel,
Lectures on four-dimensional Dehn twists,
Symplectic 4-manifolds and algebraic surfaces, 231--267,
{\it Lecture Notes in Math.}~{\bf 1938}, Springer, Berlin, 2008.


\bibitem{Op} 
E.~Opshtein,
Maximal symplectic packings in $\PP^2$,
{\it Compos.\ Math.}~{\bf 143} (2007), 1558--1575.


\bibitem{T} 
C.~H.~Taubes,
Embedded contact homology and Seiberg--Witten cohomology,~I,
arXiv:0811.3985.


\end{thebibliography}
\end{document}